\documentclass[oneside,reqno]{amsart}
  \usepackage{graphicx,amsfonts,amssymb,amsmath,amsthm,url,amscd} \usepackage{enumerate} \usepackage{color} \usepackage{chngcntr}
\usepackage[usenames,dvipsnames]{xcolor} \usepackage[normalem]{ulem} \usepackage{comment}
\usepackage{graphicx, amsmath, amssymb, amsfonts, amsthm, tikz, amscd,mathtools}
\usetikzlibrary{matrix,arrows,decorations.pathmorphing}
\usepackage[letterpaper]{geometry} \usepackage{verbatim}
\usepackage{calc}
\usepackage{dsfont}

\newcommand{\todo}[1]{{\color{blue} \sf TODO: [#1]}}  \newcommand{\pn}[1]{{\color{ForestGreen} \sf PN: [#1]}} 

\usepackage{graphicx, amsmath, amssymb, amsfonts, amsthm, stmaryrd, tikz, amscd} \usepackage[all]{xy} \usetikzlibrary{matrix,arrows,decorations.pathmorphing}

\usepackage[hyperfootnotes=false, colorlinks, citecolor= RoyalBlue, urlcolor=blue, linkcolor=blue ]{hyperref}

\makeatletter
\let\orgdescriptionlabel\descriptionlabel
\renewcommand*{\descriptionlabel}[1]{%
  \let\orglabel\label
  \let\label\@gobble
  \phantomsection
  \protected@edef\@currentlabel{#1}%
  \let\label\orglabel
  \orgdescriptionlabel{#1}%
}
\def\paragraph{
	\@startsection{paragraph}{4}
	\z@{.5\linespacing\@plus.7\linespacing}{-.5em}%
	{\normalfont\itshape}}
\makeatother

\theoremstyle{plain} 
\newtheorem{theorem} {Theorem}

\theoremstyle{definition}

\theoremstyle{plain} 
\newtheorem{proposition}[theorem] {Proposition}

\theoremstyle{remark}

\newtheorem{definition} {Definition}
\newtheorem*{definition*} {Definition}

\newtheorem{example}[theorem] {Example} \newtheorem*{example*} {Example}

\newtheorem{remark}[theorem] {Remark} \newtheorem*{remark*} {Remark}       

\newtheoremstyle{itplain} 
{6pt} 
{5pt\topsep} 
{\itshape} 
{} 
{\itshape} 
{.}  
{5pt plus 1pt minus 1pt} 
{} 

\theoremstyle{itplain} 
\newtheorem{lemma}[theorem]{Lemma}

\newtheorem*{lemma*}{Lemma}

\newtheorem{corollary}[theorem] {Corollary} \newtheorem*{corollary*} {Corollary}

\theoremstyle{remark} 

\newtheorem*{lemmatest*}{Lemma}

\usepackage{etoolbox} \patchcmd{\section}{\scshape}{\bfseries}{}{} \makeatletter \renewcommand{\@secnumfont}{\bfseries} \makeatother

\newcommand{\dif}{\,\mathrm{d}}

\renewcommand{\Re}{\mathrm{Re}} \renewcommand{\Im}{\mathrm{Im}}

\renewcommand{\geq}{\geqslant} \renewcommand{\leq}{\leqslant} \renewcommand{\mod}{\mathrm{mod}\,}

\newlength{\faktorheight}
\newcommand*{\dfaktor}[3]{
	\mathchoice{
		\settototalheight{\faktorheight}{\ensuremath{#1}}%
		\raisebox{-0.5\faktorheight}{\ensuremath{#1}}
		\backslash
		\settototalheight{\faktorheight}{\ensuremath{#2}}%
		\raisebox{0.5\faktorheight}{\ensuremath{#2}}
		\slash
		\settototalheight{\faktorheight}{\ensuremath{#3}}%
		\raisebox{-0.5\faktorheight}{\ensuremath{#3}}
	}{
		\ensuremath{#1}
		\backslash
		\ensuremath{#2}
		\slash
		\ensuremath{#3}
	}
	{
		\ensuremath{#1}
		\backslash
		\ensuremath{#2}
		\slash
		\ensuremath{#3}
	}
	{
		\ensuremath{#1}
		\backslash
		\ensuremath{#2}
		\slash
		\ensuremath{#3}
	}
}

\newcommand*{\lfaktor}[2]{
	\mathchoice{
		\settototalheight{\faktorheight}{\ensuremath{#1}}%
		\raisebox{-0.5\faktorheight}{\ensuremath{#1}}
		\backslash
		\settototalheight{\faktorheight}{\ensuremath{#2}}%
		\raisebox{0.5\faktorheight}{\ensuremath{#2}}
	}
	{
		\ensuremath{#1}
		\backslash
		\ensuremath{#2}
	}
	{
		\ensuremath{#1}
		\backslash
		\ensuremath{#2}
	}
	{
		\ensuremath{#1}
		\backslash
		\ensuremath{#2}
	}
}

\newcommand*{\faktor}[2]{
	\mathchoice{
		\settototalheight{\faktorheight}{\ensuremath{#1}}%
		\raisebox{0.5\faktorheight}{\ensuremath{#1}}
		\slash
		\settototalheight{\faktorheight}{\ensuremath{#2}}%
		\raisebox{-0.5\faktorheight}{\ensuremath{#2}}
	}
	{
		\ensuremath{#1}
		\slash
		\ensuremath{#2}
	}
	{
		\ensuremath{#1}
		\slash
		\ensuremath{#2}
	}
	{
		\ensuremath{#1}
		\slash
		\ensuremath{#2}
	}
}

\numberwithin{theorem}{section}
\numberwithin{equation}{section}  \DeclareMathOperator{\sgn}{sgn} \DeclareMathOperator{\SL}{SL}           \DeclareMathOperator{\SO}{SO}               \DeclareMathOperator{\JL}{JL}    \def\eps{\varepsilon}   \def\lcm{\operatorname{lcm}} \def\PGL{\operatorname{PGL}} \def\hol{\operatorname{hol}}              \DeclareMathOperator{\coker}{coker}     
    \DeclareMathOperator{\trace}{trace}                \DeclareMathOperator{\SU}{SU}                 
    \DeclareMathOperator{\diag}{diag}             \def\O{\operatorname{O}}                              \DeclareMathOperator{\Res}{Res}         \DeclareMathOperator{\Mat}{Mat}            \DeclareMathOperator{\Tr}{Tr} \DeclareMathOperator{\vol}{vol} \DeclareMathOperator{\nr}{nr}  \DeclareMathOperator{\tr}{tr}  \DeclareMathOperator{\Img}{Im}

\newcommand{\sm}{\left(\begin{smallmatrix}} \newcommand{\esm}{\end{smallmatrix}\right)} \newcommand{\bpm}{\begin{pmatrix}} \newcommand{\ebpm}{\end{pmatrix}}

\newcommand{\lquotient}[2]{ \mathchoice {
    \text{\lower1ex\hbox{$#1$}\Big \backslash \raise01ex\hbox{$#2$}}%
  } {
    #1\,\backslash\,#2 } {
    #1\,\backslash\,#2 } {
    #1\,\backslash\,#2 } }

\newcommand{\rquotient}[2]{ \mathchoice {
    \text{\raise01ex\hbox{$#1$}\Big/\lower1ex\hbox{$#2$}}%
  } {
    #1\,/\,#2 } {
    #1\,/\,#2 } {
    #1\,/\,#2 } }
		
\newcommand{\lrquotient}[3]{ \mathchoice {
    \text{\lower1ex\hbox{$#1$}\Big \backslash \raise01ex\hbox{$#2$}\Big/\lower1ex\hbox{$#3$}}%
  } {
    #1\,\backslash\,#2\,/\,#3 } {
    #1\,\backslash\,#2\,/\,#3 } {
    #1\,\backslash\,#2\,/\,#3 } }

\author{Ilya Khayutin, Paul D. Nelson, Raphael S. Steiner}

\date{\today} \title[Theta functions, fourth moments, and the sup-norm problem II]{Theta functions, fourth moments of eigenforms, and the sup-norm problem II}
\hypersetup{ pdfkeywords={}, pdfsubject={}, pdfcreator={Emacs 24.5.1 (Org mode 8.2.10)}}
\begin{document}

\begin{abstract}
  For an $L^2$-normalized holomorphic newform $f$ of weight $k$ on a hyperbolic surface of volume $V$ attached to an Eichler order of squarefree level in an indefinite quaternion algebra over $\mathbb{Q}$, we prove the sup-norm estimate
  \[
    \| \Im(\cdot)^{\frac{k}{2}} f \|_{\infty} \ll_{\eps} (k V)^{\frac{1}{4}+\eps}
  \]
  with absolute implied constant.  For a cuspidal Maa{\ss} newform $\varphi$ of eigenvalue $\lambda$ on such a surface, we prove that
  \[
    \|\varphi \|_{\infty} \ll_{\lambda,\eps} V^{\frac{1}{4}+\eps}.
  \]
  We establish analogous estimates in the setting of definite quaternion algebras.
\end{abstract}

\maketitle

\setcounter{tocdepth}{1} \tableofcontents

\maketitle

\section{Introduction}\label{sec:introduction}

Let $\Gamma \backslash \mathbb{H}$ be a finite volume hyperbolic surface.  A basic problem in quantum chaos is to understand the limiting behavior of $L^2$-normalized Laplace eigenfunctions $\varphi$ on $\Gamma \backslash \mathbb{H}$.  This behavior can be quantified through weak limits of $L^2$-masses (``quantum ergodicity''), bounds for $L^p$-norms, and so forth.  We consider in this paper the \emph{sup-norm problem}, which consists of bounding the supremum or $L^\infty$-norm of an $L^2$-normalized eigenfunction $\varphi$ with respect to the eigenvalue $\lambda_\varphi$ and/or the geometry of the underlying manifold $\Gamma \backslash \mathbb{H}$.  A general bound in this direction, due to B\'erard \cite{Berard}, asserts that
\begin{equation}\label{eq:varphi-_infty-leq}
  \|\varphi \|_{\infty} \ll_{\Gamma} (1+|\lambda_{\varphi}|)^{\frac{1}{4}} / \log(2 + |\lambda_\varphi|).
\end{equation}
Here and henceforth $A \ll B$ means that there is a constant $C$ such that $|A| \le CB$; we allow $C$ to depend on any subscripts of $\ll$, and write $\eps$ for an arbitrary, but sufficiently small, positive constant, which may change from line to line.

Stronger bounds have been established in the \emph{arithmetic} case that
\begin{itemize}
\item $\Gamma \backslash \mathbb{H}$ is an arithmetic manifold, such as the modular surface $\SL_2(\mathbb{Z}) \backslash \mathbb{H}$ or a congruence cover, and
\item  $\varphi$ is a \emph{Hecke--Maa{\ss} form}, i.e., an eigenfunction not only of the Laplacian, but also of the Hecke operators.  
\end{itemize}
The pioneering result in that case is due to Iwaniec--Sarnak \cite{IS95}, who showed for congruence lattices $\Gamma$ that
\begin{equation}\label{eq:sup_z-in-gamma}
\|\varphi\|_{\infty} \ll_{\Gamma,\eps} (1+|\lambda_{\varphi}|)^{\frac{5}{24}+\epsilon}.
\end{equation}

The above estimates depend in an unspecified manner upon the underlying manifold.  Consider for instance the case that $\Gamma$ is the Hecke congruence subgroup $\Gamma_0(N)
 = \SL_2(\mathbb{Z}) \cap \left(
  \begin{smallmatrix}
    \mathbb{Z} &\mathbb{Z} \\
    N \mathbb{Z} & \mathbb{Z} 
  \end{smallmatrix}
\right)$, so that $\Gamma \backslash \mathbb{H}$ is an arithmetic manifold of volume $N^{1+o(1)}$.  We suppose that $N$ is squarefree.  A direct quantification of the Iwaniec--Sarnak argument (see \cite[\S 10]{MR2587342}) gives the estimate
\begin{equation}\label{eq:sup_z-in-gamma-2}
\|\varphi\|_{\infty} \ll_{\eps} N ^{\frac{1}{2} + \eps} (1+|\lambda_{\varphi}|)^{\frac{5}{24}+\epsilon},
\end{equation}
where we normalize $\varphi$ to have $L^2$-norm one with respect to the \emph{hyperbolic probability measure}, i.e., the multiple of the hyperbolic measure having total volume one.  The \emph{level aspect} case of the sup norm problem is to improve the dependence of the bound \eqref{eq:sup_z-in-gamma-2} upon $N$.  The first improvement in the exponent was a major breakthrough of Blomer--Holowinsky \cite{MR2587342}, achieved thirteen years after the work of Iwaniec--Sarnak.  For a Hecke--Maa{\ss} newform $\varphi$ of eigenvalue $\lambda_{\varphi}$, they managed to show
\begin{equation}\label{eq:varphi_infty-ll-n}
\|\varphi\|_{\infty} \ll_{\lambda_\varphi} N^{ \frac{1}{2}-\frac{1}{37}}
\end{equation}
(with explicit polynomial dependence upon $\lambda_\varphi$).
Subsequently, Templier \cite{MR2734340} and Harcos--Templier \cite{HT2,MR3038127} established several improved bounds, culminating in
\begin{equation}\label{eq:varphi-_infty-ll_eps}
 \|\varphi \|_{\infty} \ll_{\lambda_\varphi,\epsilon} N^{\frac{1}{3}+\epsilon}.
\end{equation}
For squarefree level, the estimate \eqref{eq:varphi-_infty-ll_eps} is comparable in strength to the Weyl bound for the Riemann zeta function, and has long been regarded as a natural limit for the sup-norm problem in the squarefree level aspect \cite[Remarks (i)]{MR3038127}.  It has been extended in various ways -- to number fields \cite{MR3474814, MR4046009, AssingSup}, to levels $N$ that are not necessarily squarefree \cite{MR3713048, saha2019sup, Comtatramified, MR4190047}, and to more general vectors than newforms \cite{HNSminimal, AssingBeyondNew}. 

The motivating result of this paper is the following improvement of \eqref{eq:varphi-_infty-ll_eps}.

\begin{theorem}\label{thm:thetasup-paper:let-n-be}
  Let $N$ be a squarefree natural number.  Let $\varphi$ be a cuspidal Hecke--Maa{\ss} newform for $\Gamma_0(N)$ with trivial character.  Suppose that $\varphi$ is $L^2$-normalized with respect to the hyperbolic probability measure on $\Gamma_0(N) \backslash \mathbb{H}$.  Then
$$
\|\varphi\|_{\infty} \ll_{\lambda_{\varphi}, \epsilon} N^{\frac{1}{4}+\epsilon}.
$$
\label{thm:individual-sup-Gamma_0(N)}
\end{theorem}

Our main results apply not only to $\Gamma_0(N) \backslash \mathbb{H}$, but also to compact arithmetic quotients.  In general, such a manifold is of the shape $\Gamma \backslash \mathbb{H}$, where $\Gamma$ is commensurable with a lattice attached to a maximal order in a quaternion algebra $B$ over a totally real field $F$, with $B$ split at exactly one archimedean place (see \cite{Tits-Classification}).  We are content here to consider the case $F = \mathbb{Q}$, so that $B$ is an indefinite quaternion algebra, characterized up to isomorphism by its reduced discriminant $d_B$.  For each natural number $N$ coprime to $d_B$, we denote by $\Gamma_0^B(N)$ the group of proper (i.e., norm one) units arising from an Eichler order of level $N$ in $B$ (see Section \ref{sec:statements-setup} for details).  For example, if $B = \Mat_{2 \times 2}(\mathbb{Q})$, then we could take for $\Gamma_0^B(N) = \Gamma_0(N)$.
We prove the following theorem.

\begin{theorem}
  \label{thm:thetasup-paper:let-gamm-be}
  Let $\Gamma=\Gamma^B_0(N)$ be as above with the level $N$ being squarefree.  Let $\varphi$ be a cuspidal Hecke--Maa{\ss} newform for $\Gamma$ with trivial character, $L^2$-normalized with respect to the hyperbolic probability measure on $\Gamma \backslash \mathbb{H}$.  Then
  \begin{equation}\label{eq:varph-ll_l-epsil}
    \|\varphi\|_{\infty} \ll_{\lambda_{\varphi}, \epsilon} V^{\frac{1}{4}+\epsilon},
  \end{equation}
  where $V = (d_BN)^{1+o(1)}$ denotes the covolume of $\Gamma$.
  \label{thm:individual-sup-compact}
\end{theorem}
Theorem \ref{thm:thetasup-paper:let-gamm-be} specializes to Theorem \ref{thm:thetasup-paper:let-n-be} upon taking $B = \Mat_{2 \times 2}(\mathbb{Q})$.  It improves upon (the $F = \mathbb{Q}$ case of) Templier's result \cite{MR2734340}, which gave the nontrivial bound $V^{\frac{1}{2} - \frac{1}{24} + \eps}$.   We emphasize that the estimate \eqref{eq:varph-ll_l-epsil} is uniform in the quaternion algebra $B$, hence gives a strong saving in the ``discriminant aspect''; the first nontrivial results in that aspect (for $B$ indefinite, as we have assumed) were established only very recently by Toma \cite{TomaDivSup}, updating an earlier preprint, giving (among other things) the bound $V^{\frac{1}{2}-\frac{1}{30}+\eps}$.   Our method applies equally in the setting of \emph{definite} quaternion algebras, where we improve the exponent $\frac{1}{3}$ of Blomer--Michel \cite{MR2852302,MR3103131} down to $\frac{1}{4}$ in analogy with Theorem \ref{thm:individual-sup-compact} (see Section \S\ref{sec:results-forms} for details).

\begin{remark}
  The dependence on the eigenvalue in \eqref{eq:varph-ll_l-epsil} is of exponential nature. With a more suitable choice of archimedean component of the theta kernel, it seems likely that one could show $\|\varphi \|_{\infty} \ll_{\eps} \lambda_\varphi^{\frac{1}{4} +\eps} V ^{\frac{1}{4} + \eps }$ (we obtain such an estimate for the definite analogue of \eqref{eq:varph-ll_l-epsil}, see Corollary \ref{cor:main-corollary-fourth-moment-adelic-formulation}).  Such a refinement of \eqref{eq:varph-ll_l-epsil} seems to require lengthy archimedean calculations that we feel would distract from the primary novelties of this paper concerning the level aspect.
\end{remark}

\begin{remark}
  In the depth aspect, where $N=p^n$ is a prime power, Hu and Saha \cite{MR4190047} establish for an indefinite quaternion algebra the strong bound $\|\varphi\|_\infty\ll_{\lambda_\varphi,p,d_B, \varepsilon} N^{5/24+\varepsilon}$. The local depth aspect and the global squarefree aspect, that we address in this paper, are arguably disparate.
\end{remark}

\begin{remark}
In a function field setting analogous to that of Theorem \ref{thm:thetasup-paper:let-n-be}, Sawin \cite{MR4307129} has used geometric techniques to establish (among other things) the sup-norm bound $\ll N^{\frac{1}{4} + \alpha_q}$, where $\alpha_q > 0$ tends to zero as the cardinality $q$ of the underlying finite field tends to $\infty$.  We do not see any obstruction to adapting the techniques of this paper to
the function field setting, where we expect they would give the improved bound $\ll_{\eps} N^{\frac{1}{4} + \eps}$.
\end{remark}

By combining the arguments of this paper with those of the prequel \cite{Theta-supnorm-Is} concerning the weight aspect for holomorphic forms, we obtain the following uniform hybrid bound in the weight and level aspects.  Here we encourage the reader to focus first on the case of fixed eigenvalue/weight,
which contains  the primary novelties of this paper.

\begin{theorem} Let $\Gamma = \Gamma^B_0(N)$ be as in Theorem \ref{thm:individual-sup-compact}.  Let $f$ be a  cuspidal holomorphic newform for $\Gamma$ with trivial character and weight $k\ge 2$. Suppose $f$ is $L^2$-normalized with respect to the hyperbolic probability measure on $\Gamma \backslash \mathbb{H}$.  Then
$$
\|\Im(\cdot)^{\frac{k}{2}}f\|_{\infty} \ll_{\epsilon} (kV)^{\frac{1}{4}+\epsilon},
$$
where $V = (d_BN)^{1+o(1)}$ denotes the covolume of $\Gamma$.
\label{thm:individual-sup-holo-compact}
\end{theorem}


\subsection{Selected applications}

A straightforward application of these improved sup-norms is to $L^p$-norms for $2\le p \le \infty$ by means of interpolation. We state here only the split holomorphic case, as in this case, strong $L^4$-bounds were given by Buttcane--Khan \cite{MR3368079} with subconvexity input from \cite{MR3635360}.


\begin{corollary} Let $q$ denote an odd prime and $f$ a  cuspidal holomorphic newform for $\Gamma_0(q)$ with trivial character and weight $k$. Suppose $f$ is $L^2$-normalized with respect to the hyperbolic probability measure on $\Gamma_0(q) \backslash \mathbb{H}$. Then, for $2 \le p \le \infty$ and any $\eta>0$, we have 
$$
\| \Im(\cdot)^{\frac{k}{2}} f \|_p \ll_{k, \eta} \begin{cases} q^{\frac{1}{6}-\frac{1}{3p}+\eta} , & 2 \le  p\le 4, \\ q^{\frac{1}{4}-\frac{2}{3p}+\eta} , & 4 \le p \le \infty, \end{cases}
$$
for $k$ sufficiently large in terms of $\eta$.
\end{corollary}


Further applications of sup-norm bounds include shifted convolution problems and subconvexity results for $L$-function, see for example \cite{MR1990914, MR2207235, MR3991392, MR4190047, NordentoftEisSup}. Often, such an application would go through the use of a uniform version of Wilton's estimate. Following the arguments of \cite[\S 2.7]{MR2207235} with the improved sup-norm bound, one may derive the following corollary.

\begin{corollary}
	Let $\lambda(m)$, $m \in \mathbb{N}$, denote the Hecke eigenvalues, normalized so that the Ramanujan conjecture reads $|\lambda(m)| \ll_\eps  m^{\eps}$, of either a cuspidal Hecke--Maa{\ss} newform or a  cuspidal holomorphic newform of weight $k$ on $\Gamma_0(N)$ with trivial character, where $N$ is squarefree. Then, for any $\alpha \in \mathbb{R}$, one has
	$$
	\sum_{m \le M} \lambda(m) e(m\alpha) \ll_{\epsilon} M^{\frac{1}{2}+\epsilon} \cdot \begin{cases} N^{\frac{1}{4}+\epsilon}, & \text{in the Maa{\ss} case}, \\ N^{\frac{1}{4}+\epsilon} k^{\frac{1}{2}+\epsilon} , & \text{in the holomorphic case}, \end{cases}
	$$
	where the implied constant in the Maa{\ss} case further depends on the eigenvalue of the form.
%

\end{corollary}

As a consequence, we may, for example, improve the main theorem in \cite{MR3991392}.

\begin{corollary} Let $\varphi$ either be a cuspidal Hecke--Maa{\ss} newform or a  cuspidal holomorphic newform with respect to $\Gamma_0(q)$ with $q$ a prime. Let $\chi$ be a primitive Dirichlet character of modulus $m$ with $(m,q)=1$. Suppose $q = m^{\eta}$ with $0<\eta<2$. Then, we have
$$
L(\varphi \otimes \chi, \tfrac{1}{2}) \ll_{\epsilon} \mathcal{C}^{\frac{1}{4}+\epsilon} \left( \mathcal{C}^{-\frac{\eta}{4(2+\eta)}}+\mathcal{C}^{- \frac{2- \eta - 4 \vartheta}{8(2+\eta)}} \right),
$$
where the implied constant depends on the eigenvalue respectively weight of $\varphi$, $\mathcal{C}=qm^2$ is the conductor of the $L$-function, and $\vartheta$ is the current best bound towards the generalized Ramanujan conjecture if $\varphi$ is a Maa{\ss} form and $0$ if $\varphi$ is holomorphic.
\end{corollary}








\subsection{The fourth moment and further applications}


The method underlying most previous works on this problem, including the work of Harcos--Templier giving the bound $\ll_{\epsilon} N^{1/3+\epsilon}$, is based on the amplification method introduced in the original paper of Iwaniec--Sarnak. Recently, Steiner \cite{MR4099641} and Khayutin--Steiner \cite{Theta-supnorm-Is} introduced a new method based on analysis of fourth moments over families. The key observation of these papers was that such a fourth moment naturally arises as the $L^2$-norm of a theta kernel. Alternatively, Blomer \emph{et al.} \cite{BeyondSphericalSup} have demonstrated that one may use Vorono\"i summation for Rankin--Selberg convolutions in place of a theta kernel.
Prior to the application to fourth moments, theta kernels have played similar roles in the study of quantum variance \cite{nelson-variance-73-2s,nelson-variance-IIs,nelson-variance-3s,Nelson-TwistedSym2}, numerical computations \cite{MR3356036}, and in the proof of Waldspurger's formula \cite{MR783511}.  In each of these earlier works, theta kernels apparently served as a substitute for parabolic Fourier expansions, giving a tool for establishing analogues on compact quotients (where such expansions are not available) of results known already for non-compact quotients.  The present work differs in that our main result is new even for the non-compact quotients $\Gamma_0(N) \backslash \mathbb{H}$.


In this paper, we follow generally the theta kernel strategy of the prequel \cite{Theta-supnorm-Is} and prove a fourth moment bound from which one may deduce the Theorems \ref{thm:individual-sup-Gamma_0(N)}, \ref{thm:individual-sup-compact}, and \ref{thm:individual-sup-holo-compact} after some additional analysis near any cusps.  In what follows, we let $\Gamma = \Gamma^B_0(N)$ be a lattice as in Theorem \ref{thm:individual-sup-compact} and denote by $V =  (d_B N)^{1+o(1)}$ the volume of $\Gamma \backslash \mathbb{H}$.

The formulation of our results requires some quantification of the closeness of a point $z \in \Gamma \backslash \mathbb{H}$ to the cusps.  If $\Gamma \backslash \mathbb{H}$ is non-compact (i.e., $d_B = 1$), then we may assume that $\Gamma = \Gamma_0(N)$, and we set
\begin{equation*}
  H(z) = \max_{\gamma \in A_0(N)} \Im(\gamma z),
\end{equation*}
where $A_0(N)$ denotes the lattice of Atkin--Lehner operators for $\Gamma_0(N)$ (see Section \S\ref{sec:results-split-case} for another formulation of the definition of $H$).  If $\Gamma \backslash \mathbb{H}$ is compact, then we set $H(z) = 0$.

\begin{theorem}\label{thm:thetasup-paper-merge-20220418:let-gamma-=}
  Let $\Gamma = \Gamma^B_0(N)$ be as in Theorem \ref{thm:individual-sup-compact}.  Fix $\Lambda > 0$, and let $(\varphi_i)_i$ be an orthonormal set of cuspidal Hecke--Maa{\ss} newforms with Laplace-eigenvalue bounded by $\Lambda$ on the hyperbolic surface $\Gamma \backslash \mathbb{H}$ equipped with the hyperbolic probability measure. Then, for any two points $z,w \in \Gamma \backslash \mathbb{H}$, we have
	\begin{equation}
	\sum_i \left( |\varphi_i(z)|^2-|\varphi_i(w)|^2 \right)^2 \ll_{\epsilon, \Lambda} V^{1+\epsilon} \left(1+ V[H(z)^2+H(w)^2] \right).
		\label{eq:main-fourth-moment-thm-ineq}
	\end{equation}
	
	Similarly, for an orthonormal set $(f_i)_i$ of cuspidal holomorphic newforms for $\Gamma$ of weight $k$ and trivial character with respect to the hyperbolic probability measure on $\Gamma \backslash \mathbb{H}$, we have
	\begin{multline*}
	\sum_i \left( |\Im(z)^{\frac{k}{2}}f_i(z)|^2-|\Im(w)^{\frac{k}{2}}f_i(w)|^2 \right)^2 \\
	 \ll_{\epsilon} (Vk)^{1+\epsilon} \left(1+V^{\frac{1}{2}}[H(z)+H(w)]+Vk^{-\frac{1}{2}}[H(z)^2+H(w)^2] \right)
	\end{multline*}
	for any two points $z,w \in \Gamma \backslash \mathbb{H}$.
	\label{thm:intro-thm}
\end{theorem}

In the case that the hyperbolic surface $\Gamma \backslash \mathbb{H}$ is compact, we may integrate $z$ and $w$ over the whole surface and get an essentially sharp bound on the fourth moment of fourth norm in the level aspect, thereby extending a result of Blomer \cite{MR3082245} to the case of cocompact lattices $\Gamma$.

\begin{corollary}
  With notation and assumptions as in Theorem \ref{thm:thetasup-paper-merge-20220418:let-gamma-=}, and assuming further that $\Gamma \backslash \mathbb{H}$ is compact, we have
	\begin{equation*}
	\sum_i \|\varphi_i\|_4^4 \ll_{\epsilon, \Lambda} V^{1+\epsilon},
	\end{equation*}
	\begin{equation*}
	\sum_i \|\Im(\cdot)^{\frac{k}{2}}f_i\|_4^4 \ll_{\epsilon} (Vk)^{1+\epsilon}.
	\end{equation*}
	\label{cor:intro-thm}
\end{corollary}

This result may also be recast as a double average of triple $L$-functions by means of Watson's formula \cite[Theorem 3]{watson-2008s}.

The final application of Theorem \ref{thm:intro-thm} we mention is to the diameter of compact  arithmetic hyperbolic surfaces $\Gamma \backslash \mathbb{H}$ \cite{SteinerSmallDiameter}. Here, one may use the sharp bound on the ``fourth moment" of exceptional eigenforms, together with a strong density estimate for the exceptional eigenvalues, to get an optimal estimate on the almost diameter and an estimate on the diameter of the same strength as if one were to assume the Selberg eigenvalue conjecture.

\subsection{The added complexity of the level aspect}


Compared to the weight aspect treated in the prequel, the level aspect requires many new ideas.  Here we tacitly restrict to the case of \emph{squarefree} level; the general case would require a more nuanced discussion.  In some sense, the level aspect may be understood as intermediate in difficulty between the holomorphic and eigenvalue aspects.  Indeed, relative to known techniques, the difficulty in the sup-norm problem is reflected in the essential support of the matrix coefficient of the automorphic form being bounded.  In the weight, (squarefree) level and eigenvalue aspects, the matrix coefficient concentrates on a space of dimension one, two and three, respectively.


We now briefly recall the main idea of the theta approach and discuss some of the new challenges that arise in the level aspect. We focus first on the case of Hecke--Maa{\ss} forms on $\Gamma_0(N) \backslash \mathbb{H}$, as in Theorem \ref{thm:individual-sup-Gamma_0(N)}.  Take $R=\sm \mathbb{Z} & \mathbb{Z} \\ N \mathbb{Z} & \mathbb{Z} \esm$, so that that the set of proper units of $R$ is precisely $\Gamma_0(N)$. For $\ell \mid N$, let $R(\ell)= \sm \mathbb{Z} & \mathbb{Z} \slash \ell \\ N \mathbb{Z} \slash \ell & \mathbb{Z} \esm$ denote the partially dualized lattices of the order $R$. Let $\sigma_z \in \SL_2(\mathbb{R})$ be any matrix taking $i$ to $z \in \mathbb{H}$.  Let $\varphi$ be an \emph{arithmetically-normalized} cuspidal Hecke--Maa{\ss} newform. The theta identity at the heart of the argument then reads
\begin{equation}
  \langle \theta(z,w;\cdot), \varphi \rangle =  \frac{1}{V} \varphi(z)\overline{\varphi(w)},
  \label{eq:intro-theta-key-identity}
\end{equation}
where $V$ denotes the covolume of $\Gamma_0(N)$ and the theta function is given by
\begin{equation}
  \theta(z,w;s) = \Im(s) \sum_{\sm a & b \\ c & d \esm \in \sigma_z^{-1} R \sigma_w} e^{- \pi  (a^2+b^2+c^2+d^2) \Im(s)} e^{2 \pi i  (ad-bc) \Re(s)}.
  \label{eq:intro-theta}
\end{equation}
By Bessel's inequality, the left-hand side of \eqref{eq:main-fourth-moment-thm-ineq} is in essence captured by the $L^2$-norm of the difference of the theta kernels $\theta(z,z;\cdot)-\theta(w,w;\cdot)$. From here, one may then proceed as in the prequel by covering a fundamental domain by Siegel sets and making use of the orthogonality relations in the unipotent direction. One ends up with a weighted sum over matrices $\gamma_1, \gamma_2 \in R(\ell)$ satisfying $\det(\gamma_1)=\det(\gamma_2)$ and for which the entries of $\sigma_z^{-1} \gamma_{i} \sigma_z$, $i=1,2$ satisfy certain bounds (and similarly for $w$).  The bounds imposed on these entries depend crucially upon the precise choice of Siegel domains, so it is important that we make a good choice.  Like in the prequel, we split the count according to whether $\trace(\gamma_1) = \trace(\gamma_2)$ or not.

In the case of non-equal trace, the na\"ive choice of Siegel domains consisting of $\Gamma_0(N) \backslash \SL_2(\mathbb{Z})$-translates of the standard Siegel domain for $\SL_2(\mathbb{Z})$ leads to a rather challenging counting problem. In order to get a sharp bound on \eqref{eq:main-fourth-moment-thm-ineq}, one faces the challenge of counting, for each divisor $\ell$ of $N$ and each $T$ with $\ell^{-1/2} \ll T \ll 1$, the sextuples of integers $(a_1,b_1,c_1,a_2,b_2,c_2)$ satisfying 
\begin{equation}
  (c_i y N / \ell)^2 + 2 ( a_i  - c_i xN/\ell)^2 + y^{-2} ( 2 a _i x + b / \ell - c_i x^2 N/ \ell)^2 \leq T^2, \quad i=1,2, \label{eq:intro-type-I-ineq}
  \end{equation}
\begin{equation}
  b_1 c_1 \equiv b_2 c_2 \pmod{\ell^2/N}. \label{eq:intro-type-I-congruence}
\end{equation}
We would need to know that the number of such sextuples is roughly $\O(\ell T^2)$ in the range $N^{-1} \ll y \ll N^{-1/2}$ and $|x| \le \frac{1}{2}$.  We do not know how to establish such a bound directly when, for instance, $\ell = N$.  On the other hand, when $\ell = 1$, the congruence condition is void and, using arguments of Harcos--Templier, we can prove the required bound with some room to spare, namely, for $T$ up to $N^{1/2}$. Our solution to this dichotomy is thus to decrease the size of the Siegel domains associated to larger $\ell$ at the expense of increasing those associated to smaller $\ell$.  This solution may be implemented most simply by applying an Atkin--Lehner involution to the covering of $\Gamma_0(N) \backslash \mathbb{H}$ by $\SL_2(\mathbb{Z})$-translates of the standard fundamental domain for $\SL_2(\mathbb{Z})$.  With this maneuver, we reduce to considering the range $T \ll N^{\frac{1}{2}} \ell^{-1}$.  We are then able to prove the required bound by forgoing the congruence condition, reducing the problem to counting triples of integers $(a_i,b_i,c_i)$ satisfying \eqref{eq:intro-type-I-ineq}, which we carry out using geometry of numbers techniques.  We refer subsequently to this type of counting problem, where we count traceless matrices $\gamma \in R(\ell)^0$ with a bound on the entries of $\sigma_{z}^{-1}\gamma \sigma_z$, as ``Type I''.

In the case of equal trace, we need to count sextuples of integers $(a_1,b_1,c_1,a_2,b_2,c_2)$ satisfying \eqref{eq:intro-type-I-ineq} and
\begin{equation}
  a_1^2+b_1c_1 \tfrac{N}{\ell^2} = a_2^2+b_2c_2 \tfrac{N}{\ell^2}.
  \label{eq:intro-type-II}
\end{equation}
We need to bound this count by $O(\ell T)$ in the same ranges as before.  We refer to this type of counting problem as ``Type II''.  The key observation is that $(a_1,b_1,c_1)$ turns out to determine $(a_2,b_2,c_2)$ up to a small number of possibilities.  This allows us to reduce Type II estimates to Type I estimates.

The above arguments suffice for non-compact quotients, i.e., for the proof of Theorem \ref{thm:thetasup-paper:let-n-be}.  They rely on the use of matrix coordinates $\sm
  a & b \\
  c & d
\esm$ with respect to which the lattices $\Gamma_0(N)$ are described by the simple congruence condition $c \equiv 0 \pod{N}$.  We were unable to find an analogously straightforward way to separate the variables in the compact setting (e.g., using fixed quadratic subalgebras of $B$).  In the case that $B$ is definite, the Type I counts were treated in a coordinate-free way by Blomer--Michel \cite{MR2852302,MR3103131}, who controlled the successive minima of the ternary quadratic lattice underlying $\Gamma_0^B(N)$ in terms of only the content, level, and discriminant of that lattice.  We extend their arguments to the case that $B$ is indefinite by defining analogous archimedean quantities that control the disparity of the reduced norm and a majorant, such as the square of the Frobenius norm of $\sigma_z^{-1} \gamma \sigma_z$ for $\gamma \in R(\ell)^0$.  Following the same strategy as in the non-compact case, it remains then only to reduce Type II estimates to Type I estimates.  This reduction is perhaps the most subtle part of our counting arguments.  It requires us to establish the analogue in the compact setting of the key observation noted following \eqref{eq:intro-type-II}.  For example, in case that $B$ is definite, writing $R$ for an Eichler order of level $N$, we need to show that for each $n \ll T^2$, the number of elements $\gamma \in R$ with trace $0$ and norm $n$ is essentially $\O(1)$, uniformly in $N$ and $B$.  We eventually managed to do so through a delicate argument involving commutators and representations of binary quadratic forms.

\subsection{Organization of the paper}

The complete statement of results may be found in Section \S\ref{sec:resuls}. In Section \S\ref{sec:reduction-proof}, we reduce the proof of our main result to that of two auxiliary collections of results:
\begin{itemize}
\item those concerning matrix counting, and
\item those reducing the required estimates for theta functions to matrix counting.
\end{itemize}
The latter including the appropriate splicing of a fundamental domain into Siegel sets may be found in Section \S\ref{sec:real-manifolds}. In Section \S\ref{sec:Theta-L2norm}, we summarize the required properties of the theta functions. 
The proof of said properties we defer to Appendix \ref{sec:appendix-theta}.

Sections \S\ref{sec:order-theor-prel} and \S\ref{sec:estim-succ-minima} are dedicated to the anisotropic extension of the lattice counting argument of Blomer--Michel which we subsequently apply to the Type I counting problem in Section \S\ref{sec:type-i-estimates}.

The final section, \S\ref{sec:type-ii-estimates}, treats the crucial Type II counting problem.


\subsection{Acknowledgements}

We would like to thank E. Assing, V. Blomer, F. Brumley, G. Harcos, Y. Hu, S. Marshall, A. Saha, W. Sawin, and R. Toma for their helpful feedback on an earlier draft. We would also like to thank P. Sarnak for fruitful discussions on this and surrounding topics as well as his continued encouragement and support.

I.K. is deeply grateful for support of the AMS Centennial fellowship and the Sloan Research Fellowship.

This paper was completed while P.N. was at the Institute for Advanced Study during the academic year 2021-2022, where he was supported by the National Science Foundation under Grant No. DMS-1926686.  Some work on this project also occurred during a short-term visit of P.N. to the Institute for Advanced Study in February 2020.

R.S. wishes to extend his gratitude to the Institute for Advanced Study, where he was supported by the National Science Foundation Grant No. DMS -- 1638352 and the Giorgio and Elena Petronio Fellowship Fund II, and the Institute for Mathematical Research (FIM) at ETH Z\"urich.

\section{Statement of results}

\label{sec:resuls}

\subsection{Setup}\label{sec:statements-setup}
Let $B$ be a quaternion algebra over $\mathbb{Q}$.  We denote by $d_B$ its reduced discriminant, or equivalently, the product of the primes at which $B$ ramifies. We write $G$ for the linear algebraic group over $\mathbb{Q}$ given by $G(L)=\lfaktor{L^\times}{(B\otimes L)^\times}$ for any $\mathbb{Q}$-algebra $L$.  Then $G$ is an inner form of $\PGL_2$, and all rational forms of $\PGL_2$ arise in this way. Denote by $[G]$ the adelic quotient $G(\mathbb{Q}) \backslash G(\mathbb{A})$. We fix the probability Haar measure on $[G]$. Let $K_\infty$ be a compact maximal torus of $G(\mathbb{R})$.  We assume that $K_\infty$ comes equipped with a choice of isomorphism $\kappa : \mathbb{R} / \pi \mathbb{Z} \xrightarrow{\sim} K_\infty$. In the split case $B=\Mat_{2 \times 2}(\mathbb{Q})$, we identify $G=\operatorname{PGL}_2$ and set $\kappa(\theta)=\sm \cos(\theta) & \sin(\theta) \\ -\sin(\theta) & \cos(\theta) \esm$. 

Let $R$ be an \emph{Eichler order} in $B$, i.e., an intersection of two maximal orders.  We denote by $N$ the \emph{level} of $R$.  It is a natural number, coprime to $d_B$, characterized as follows: for each prime $p \nmid d_B$, there is an isomorphism $B_p\coloneqq B\otimes \mathbb{Q}_p \cong \Mat_{2 \times 2}(\mathbb{Q}_p)$ under which $R_p \coloneqq R \otimes_{\mathbb{Z}} \mathbb{Z}_p$ maps to the order $\left(
  \begin{smallmatrix}
    \mathbb{Z}_p &\mathbb{Z}_p \\
    N \mathbb{Z}_p& \mathbb{Z}_p
  \end{smallmatrix}
\right)$.  We may then identify $G(\mathbb{Q}_p)$ with $\PGL_2(\mathbb{Q}_p)$ and the image of $R_p^\times$ with a finite index subgroup of $\PGL_2(\mathbb{Z}_p)$.  We assume that $N$ is squarefree, so that $d_B N$ is likewise squarefree. We denote by $K_R$ the compact open subgroup of $G(\mathbb{A}_f)=\prod_p ' G(\mathbb{Q}_p)$ given by the image of $\prod_p R_p^\times$. 

Fix $k \in 2\mathbb{Z}$.  Let $\mathcal{A}$ denote the set of cusp forms $\varphi : [G] \rightarrow \mathbb{C}$ having the following properties:
\begin{itemize}
\item $\varphi(g \kappa(\theta)) = e^{i k \theta} \varphi(g)$ for all $\theta$.
\item $\varphi$ is an eigenfunction for some fixed Casimir operator for $G(\mathbb{R})$, with eigenvalue $\lambda_{\varphi}$. For the sake of concreteness, we scale the Casimir operator such that it agrees with the standard Laplace operator on the locally symmetric space $G(\mathbb{R}) \slash K_{\infty}$, which identifies with either $\mathbb{H}$ or $S^2$.
\item $\varphi$ is $K_R$-invariant: $\varphi(g k) = \varphi(g)$ for $k \in K_R$.
\item $\varphi$ belongs to the \emph{newspace} for $R$, i.e., $K_R$ is the largest subgroup of $G(\mathbb{A}_f)$ keeping $\varphi$ invariant.  Equivalently, $\varphi$ is orthogonal the space of $K_{R'}$-invariant cusp forms for every Eichler orders $R'$ strictly containing $R$.
\item $\varphi$ is an eigenform for almost all Hecke operators.
\end{itemize}
If $k \geq 2$, then we write $\mathcal{A}^{\hol} \subseteq \mathcal{A}$ for the subspace of automorphic lifts of holomorphic forms, or equivalently, the kernel of the raising (resp. lowering) operator attached to $K_\infty$ if $B$ is definite (resp. indefinite). 

Denote by $\mathcal{F}$ a maximal orthonormal subset of $\mathcal{A}$. 
Analogously, we define $\mathcal{F}^{\hol} \subseteq \mathcal{A}^{\hol}$ if $k \ge 2$. Because of the multiplicity-one theorem for $\operatorname{GL}_2$ and its inner forms, the bases $\mathcal{F}, \mathcal{F}^{\hol}$ are unique up to rescaling each element by a scalar of unit magnitude. We note that the sets $\mathcal{A}$, $\mathcal{A}^{\hol}$, $\mathcal{F}$, and $\mathcal{F}^{\hol}$ depend on $k$; while we suppress this dependence from the notation, $k$ is one of the main parameters of interest.

We will consider several subfamilies of $\mathcal{F}$ and $\mathcal{F}^{\hol}$.  Here a minus sign in the exponent signifies the indefinite case, a plus sign the definite case.
\begin{itemize}
\item If $B$ is indefinite and $k=0$, then we take $\mathcal{F}^{-} \coloneqq \mathcal{F}$ and let $\mathcal{F}^{-}_{\lambda}$ (resp. $\mathcal{F}^{-}_{\le L}$) denote the subsets defined by taking the Casimir eigenvalue equal to $- \lambda$ (respectively, at most $L$ in magnitude).
\item If $B$ is indefinite and $k \ge 2$, then we take $\mathcal{F}^{-, \hol} \coloneqq \mathcal{F}^{\hol}$.
\item If $B$ is definite and $k=0$, then we let $\mathcal{F}^{+}_m \subset \mathcal{F}$ be the subset of forms,  whose associated automorphic representation at infinity is isomorphic to the unique irreducible unitary representation of $\SU_2(\mathbb{C})$ of degree $m+1$. In other words, their eigenvalue with respect to the Casimir operator equals to $-m(m+1)$.
\item If $B$ is definite and $k \ge 2$, then we let $\mathcal{F}^{+, \hol}=\mathcal{F}^{\hol}$.
\end{itemize}


\subsection{The split case}\label{sec:results-split-case}
Assume for the moment that $B$ is split.  We may suppose then that
\begin{equation} \label{eq:split-case-class-rep} B = \Mat_{2 \times 2}(\mathbb{Q}), \quad G = \PGL_2, \quad R = \begin{pmatrix}
    \mathbb{Z}  & \mathbb{Z}  \\
    N \mathbb{Z} & \mathbb{Z}
  \end{pmatrix},
\end{equation}
\begin{equation}\label{eqn:split-case-K-infty-kappa}
  K_\infty = \operatorname{PSO}_2(\mathbb{R}),
  \quad
  \kappa(\theta) =
  \begin{pmatrix}
    \cos \theta  & \sin \theta  \\
    -\sin \theta & \cos \theta
  \end{pmatrix},
\end{equation}
and may identify
\[
  [G] / K_\infty K_R \cong \Gamma_0(N) \backslash \mathbb{H}.
\]
We define
\[
  H : [G] / K_\infty K_R \rightarrow \mathbb{R}_{>0},
\]
as follows.  Let $A_0(N)<\operatorname{GL}_2(\mathbb{Q})^+$ denote the group generated by $\Gamma_0(N)$ and all Atkin--Lehner operators.  If $g \in [G] / K_\infty K_R$ identifies with $z \in \Gamma_0(N) \backslash \mathbb{H}$, then we set $$H(g) = H(z) \coloneqq \max_{\gamma \in A_0(N)} \Im(\gamma z).$$ Since the Atkin--Lehner operators constitute scaling matrices for the various cusps of $\Gamma_0(N)$ (cf. \S\ref{sec:cusps-AL-operators}), the function $H$ may be understood as a normalized height or as quantifying closeness to the cusps. Let $\mathfrak{a} \in P^{1}(\mathbb{Z})$ be a cusp of $\Gamma_0(N)$, and let $\sigma_{\mathfrak{a}}\in \SL_2(\mathbb{Z})$ such that $\sigma_{\mathfrak{a}} \infty = \mathfrak{a}$. Then,
\begin{equation}
	H(z) = \max_{\mathfrak{a}} \frac{\Im(z_{\mathfrak{a}})}{w_{\mathfrak{a}}},
	\label{eq:alt-height-description}
\end{equation}
where $\mathfrak{a}$ runs over all cusps of $\Gamma_0(N)$, $z_{\mathfrak{a}}= \sigma_{\mathfrak{a}}^{-1} z$, and $w_{\mathfrak{a}}$ is the cusp width of $\mathfrak{a}$.




\subsection{Results on forms}
\label{sec:results-forms}

We adopt the following asymptotic notation $\preccurlyeq$:
$$
A_1 \preccurlyeq A_2 \quad \iff \quad A_1 \ll_{\eps} \left(d_BN(1+k)(1+\mu)\right)^{\eps} A_2,
$$
where $\mu$ is a quantity relating to the eigenvalues with respect to the Casimir operator of the automorphic forms of relevance to the inequality. Concretely, when talking about the families $\mathcal{F}^{-}_{\lambda},\mathcal{F}^{-}_{\le L}, \mathcal{F}^{+}_m, \mathcal{F}^{\pm,\hol}$ we mean $\mu=|\lambda|,L,m, k$, respectively.

\begin{theorem}\label{thm:main-result-fourth-moment-adelic-formulation}
  Let $g_1,g_2 \in [G]$.  If $B$ is indefinite, then
  \begin{equation}\label{eqn:main-theorem-adelic}
    \sum _{\varphi \in \mathcal{F}^{-}_{\le L} } (|\varphi(g_1)|^2-|\varphi(g_2)|^2)^2
    \preccurlyeq_{L}
    d_B N \left( 1+ d_BN \left[H(g_1)^2+H(g_2)^2 \right] \right),
  \end{equation}
  for $L > 0$, and
  \begin{multline}\label{eqn:main-theorem-adelic-hol}
    \sum _{\varphi \in \mathcal{F}^{-,\hol} }  (|\varphi(g_1)|^2-|\varphi(g_2)|^2)^2 \\
    \preccurlyeq d_B N k \left( 1+(d_BN)^{\frac{1}{2}} \left[H(g_1)+H(g_2)\right]+ d_B N k^{-\frac{1}{2}} \left[H(g_1)^2+H(g_2)^2 \right] \right).
  \end{multline}
  for $k \ge 2$ even. In both cases, the term involving $H(g_{1,2})$ is only present if $B$ is split.

  If $B$ is definite, then
  \begin{equation}\label{eqn:main-theorem-definite-adelic}
    \sum _{\varphi \in \mathcal{F}^{+}_{m} } (|\varphi(g_1)|^2-|\varphi(g_2)|^2)^2
    \preccurlyeq
    d_B N(m+1)^2,
  \end{equation}
  for $m \in \mathbb{N}_0$, and
  \begin{equation}\label{eqn:main-theorem-definite-adelic-hol}
    \sum _{\varphi \in \mathcal{F}^{+,\hol} } |\varphi(g_1)|^4
    \preccurlyeq
    d_B N k,
  \end{equation}
  for $k \in 2 \mathbb{N}$.
\end{theorem}

\begin{remark} In the indefinite holomorphic case \eqref{eqn:main-theorem-adelic-hol}, one may have the same bound for the fourth moment rather than the squared difference under the assumption that the weight satisfies $k \gg_{\eta} (d_BN)^{\eta}$ for some $\eta>0$, in which case the implied constant also depends on $\eta$ and the implied constant in the assumed lower bound for the weight.
\end{remark}

\begin{corollary}
  \label{cor:main-corollary-fourth-moment-adelic-formulation}
For $k \ge 2$ and $\varphi \in \mathcal{F}^{\hol}$, we have
$$
\| \varphi\|_{\infty} \preccurlyeq (d_BN k)^{\frac{1}{4}}.
$$
For $k=0$ and $\varphi \in \mathcal{F}$, we have
$$
\| \varphi \|_{\infty} \preccurlyeq_{\lambda_{\varphi}} (d_BN)^{\frac{1}{4}}.
$$
If $B$ is definite, then we have more precisely
$$
\| \varphi \|_{\infty} \preccurlyeq (d_BN)^{\frac{1}{4}} (1+|\lambda_{\varphi}|)^{\frac{1}{4}}.
$$
\end{corollary}

By a well-known procedure, these statements may be translated into the classical language, thus giving rise to the theorems in the introduction. For further details, see for example \cite[\S3.2 \& \S3.6]{MR1431508} for the indefinite case and \cite{MR3103131} for the definite case.

\subsection{Counting problems: setup}
\label{sec:counting-setup}

\subsubsection{Lattices locally dual to $R$}
Let $\ell$ be a divisor of the squarefree number $d_B N$.  We denote by $R(\ell)$ the lattice in $B$ whose local components $R(\ell)_p$ are given
\begin{itemize}
\item for $p$ dividing $\ell$, by the lattice $R_p^\vee \subseteq B_p$ dual to $R_p$, and
\item otherwise, by $R_p$.
\end{itemize}

\subsubsection{Reduced trace and norm}
We denote by $\tr$ and $\det$ the reduced trace and reduced norm on $B$, and also on its completions.  We use a superscripted $0$, as in $R^0$ or $R(\ell)^0$, to denote the kernel of the reduced trace.

\subsubsection{Coordinates tailored to $K_\infty$} \label{sec:Kinftynotation}
\def\i{\mathbf{i}} \def\j{\mathbf{j}} \def\k{\mathbf{k}}
Define $B_\infty\coloneqq B\otimes \mathbb{R}$. If $B$ is indefinite, then $B_\infty\cong \Mat_{2 \times 2}(\mathbb{R})$ is split; otherwise, $B_\infty$ is isomorphic to the real Hamilton quaternions.
The exponential series identifies $B_\infty^0$ with the Lie algebra of $G(\mathbb{R})$.  We write $\i \in B_\infty^0$ for the derivative at the identity of $\kappa$, so that $\kappa(\theta) = \exp( \theta \i)$.  Then, $\i^2 = -1$. 
We may find $\j \in B_\infty^0$ with $\j^2 = \pm 1$ ($+1$ if $B$ is indefinite, $-1$ if $B$ is definite) so that $B_\infty = \mathbb{R}(\i) \oplus \mathbb{R}(\i) \j$. We note that $\j$ is not uniquely determined, but any two choices differ by multiplication by a norm one element of $\mathbb{R}(\i)$. We set $\k = \i\j$.  Then, $\i,\j,\k$ give an $\mathbb{R}$-basis of $B_\infty^0$.  For real numbers $a,b,c$, we set $[a,b,c] \coloneqq a \i + b \j + c \k$.  A general element of $B_\infty$ may then be written $[a,b,c] + d$, where we identify the real number $d$ with a scalar element of $B_\infty$.  In these coordinates,
\begin{equation}\label{eqn:tr-det-coordinates}
  \tr([a,b,c]+d) = 2 d,
  \quad
  \det([a,b,c]+d) = a^2 \mp (b^2 + c^2) + d^2.
\end{equation}

\begin{example}
  Suppose that $B_\infty = \Mat_{2 \times 2}(\mathbb{R})$ and that $\kappa$ is as in \eqref{eqn:split-case-K-infty-kappa}.  Then, with suitable choices,
  \[
    [a,b,c] + d =
    \begin{pmatrix}
      d + c & b + a \\
      b- a & d-c
    \end{pmatrix}.
  \]
\end{example}

\subsubsection{Archimedean regions}\label{sec:archimedean-regions}
For $T > 0$ and $\delta \in (0,1]$, we denote by $\Omega(\delta,T)$ the set of all elements $[a,b,c] + d$ of $B_\infty$ for which
\[
  a^2 + b^2 + c^2 + d^2 \leq T^2, \quad b^2 + c^2 \leq \delta T^2.
\]
With $\Omega^{\star}(\delta,T)$, we denote the subset of \emph{non-zero} elements of $\Omega(\delta,T)$. Likewise, for $T>0$ and $\delta \in (0,1]$, we let $\Psi(\delta,T)$ denote the set of all elements $[a,b,c]+d$ of $B_{\infty}$ for which
\[
  a^2 + b^2 + c^2 + d^2 \leq T^2, \quad a^2 + d^2 \leq \delta T^2;
\]
and $\Psi^{\star}(\delta,T)$ its subset consisting of \emph{non-zero} elements.

\subsection{Counting problems: results}
\label{sec:counting-results}

We adopt the following asymptotic notation for counting estimates (compare with the notation $\preccurlyeq$ introduced in \S\ref{sec:results-forms}):
\[
  A_1 \prec A_2 \quad \iff \quad A_1 \ll_\eps (d_B N (1+T))^\eps A_2.
\]

Recall from \S\ref{sec:results-split-case} the height function $H$ defined in the split case.  In the non-split case, we adopt the convention in the following results that any terms involving $H$ (in minima or sums) should be omitted.

\begin{theorem}[Type I estimates]
  \label{thm:typeI} 
  Let $g \in G(\mathbb{R})$. Then, the first successive minima (see Definition \ref{defn:succ-min}) of $g^{-1} R(\ell)^0 g$ with respect to $\Omega(\delta, 1)\cap B_{\infty}^0$ is $\gg \min \left\{ \ell^{-\frac{1}{2}} , \ell^{-1} \delta^{-\frac{1}{2}} H(g)^{-1} \right\}$. Furthermore, we have
  \[
    |g^{-1} R(\ell)^0 g \cap \Omega(\delta, T)| \prec 1 + \left(\ell^{\frac{1}{2}} +\ell \delta^{\frac{1}{2}} H(g) \right) T + \left( \frac{\ell^{\frac{3}{2}} \delta^{\frac{1}{2}} }{(d_B N)^{\frac{1}{2}}} + \ell \delta H(g) \right) T^2 + \frac{\ell^2 \delta}{d_B N} T^{3}.
  \]
	
  If $B$ is non-split, we further have that the first successive minima of $g^{-1} R(\ell)^0 g$ with respect to $\Psi(\delta, 1)\cap B_{\infty}^0$ is at least $\gg \ell^{-\frac{1}{2}}$ and
  \[
    |g^{-1} R(\ell)^0 g \cap \Psi(\delta, T)| \prec 1 + \ell^{\frac{1}{2}} T + \frac{\ell^{\frac{3}{2}} }{(d_B N)^{\frac{1}{2}}} T^2 + \frac{\ell^2 \delta^{\frac{1}{2}}}{d_B N} T^{3}.
  \]
\end{theorem}
\begin{theorem}[Type II estimates]
  \label{thm:typeII}
  Let $g \in G(\mathbb{R})$ and $n \in \frac{1}{\ell} \mathbb{Z}$.  We have
  \[
    |g^{-1} R(\ell)^0 g \cap \Omega(\delta, T) \cap \det{}^{-1}(\{n\})| \prec 1 + \ell \delta^{\frac{1}{2}} H(g) T +\frac{\ell^2}{d_B N} \delta T^2.
  \]
\end{theorem}

The proof of these results occupies \S\ref{sec:order-theor-prel} onwards.  In \S\ref{sec:reduction-proof}, we explain how these results imply our main fourth moment bound, Theorem \ref{thm:main-result-fourth-moment-adelic-formulation}.



\section{Division and reduction of the proof}\label{sec:reduction-proof}

\subsection{Traversing the genus}
\label{sec:prelim-proof-reduction}

Recall that $K_R$ is defined as the image of the subgroup $\prod_p R_p^{\times}$ in $G(\mathbb{A}_f)$; it is a compact open subgroup of $G(\mathbb{A}_f)$. In due course,  we will consider the conjugated sets $h_f K_R h_f^{-1}$, for $h_f \in G(\mathbb{A}_f)$. These are precisely the compact open subgroups $K_{R'}$ associated to the Eichler orders $R'$ in the genus of $R$.  We note that $R'$ has the same level as $R$, and may be given explicitly by the following intersection:
$$
R' = h_f (R \otimes \widehat{\mathbb{Z}}) h_f^{-1} \cap B(\mathbb{Q}).
$$
We further note that the action of $G(\mathbb{A}_f)$ on the genus of $R$ commutes with partial dualization in the sense that
$$
R'(\ell) = h_f (R(\ell) \otimes \widehat{\mathbb{Z}}) h_f^{-1} \cap B(\mathbb{Q}).
$$
This observation permits us to formulate the required $L^2$-estimates for our differences of theta kernels in terms of integration over archimedean, rather than adelic, arguments.  To that end, we introduce the notation
$$
R(\ell; h) \coloneqq h_\infty^{-1} R'(\ell) h_\infty = h_\infty^{-1} ( h_f ( R(\ell) \otimes \widehat{\mathbb{Z}} ) h_f^{-1} \cap B(\mathbb{Q})) h_\infty
$$
for $h=(h_\infty,h_f) \in G(\mathbb{A})$. We note that for $h \in G(\mathbb{R})$ (i.e., $h_f = 1$), the set $R(\ell; h)$ is just $h^{-1}R(\ell)h$. Since taking the trace commutes with conjugation, we may extend the notation to kernels of the reduced trace without concern for confusion regarding the order of operation, i.e.,
$$
R(\ell; h)^0 = (h_\infty^{-1} R'(\ell) h_\infty)^0 = h_\infty^{-1} R'(\ell)^0 h_\infty = h_\infty^{-1} ( h_f ( R(\ell)^0 \otimes \widehat{\mathbb{Z}} ) h_f^{-1} \cap B(\mathbb{Q})) h_\infty.
$$
If $B$ is split, then the class number of $R$ is one and we have fixed the representative as in \eqref{eq:split-case-class-rep}. In this case, we find for $h \in G(\mathbb{A})$ that $h^{-1}R h=h'^{-1}Rh'$, where $h' \in G(\mathbb{R})$ has the same image under the isomorphism $[G]\slash K_{\infty}K_R \cong \Gamma_0(N) \backslash \mathbb{H}$ as $h$ does.  In particular, we have the equality of height functions (see \S\ref{sec:results-split-case}) $H(h)=H(h')$.

\begin{remark}
  By considering a right translate of $\varphi \in \mathcal{F}$ and thereby moving the maximal compact $K_{\infty}$ and the Eichler order $R$ around, one could reduce the statement of the main Corollary \ref{cor:main-corollary-fourth-moment-adelic-formulation} to the case that $g$ is the identity.
However, in the split case, our counting arguments \emph{do} depend on the particular order in the genus.  Moreover, our method relies on a difference of theta kernels defined relative to different $g$.  Such a reduction would thus be premature.
\end{remark}




\subsection{Estimating fourth moments via lattice sums}

In \S\ref{sec:Theta-L2norm}, we introduce certain theta kernels.  A spectral expansion of their $L^2$-norms will yield the fourth moments of interest, while a ``geometric'' expansion, using Siegel domains and Fourier expansions, bounds those $L^2$-norms in terms of certain lattice sums.  We now state the latter bounds.

\begin{proposition}
  \label{prop:reductiontocountmaass}
  Suppose $B$ is indefinite. Then, for $g_1, g_2 \in [G]$, there exists $\ell|d_BN$, $g \in \{g_1, g_2\}$, and $0< T \preccurlyeq \frac{(d_B N(k+1))^{\frac{1}{2}}}{\ell}$ (here the notation $\preccurlyeq$ is as in \S\ref{sec:results-forms}) so that, for $k=0$,
  \begin{equation}\label{eq:reductiontocountmaass}
    \sum _{\varphi \in \mathcal{F}^{-}_{\le L} }
    (|\varphi(g_1)|^2 - |\varphi(g_2)|^2)^2
    \preccurlyeq_{L}
    1+
    \frac{d _B N}{\ell T^2}
    \sum _{\substack{
        \gamma_1, \gamma_2 \in  R(\ell;g)  \cap \Omega^{\star}(1, T): \\
        \det(\gamma_1) = \det(\gamma_2)
      }
    }
    1,
  \end{equation}
  while for $k > 0$,
  \begin{equation}\label{eq:reductiontocountholosmall}
    \sum _{\varphi \in \mathcal{F^{-,\hol}} }
    (|\varphi(g_1)|^2 - |\varphi(g_2)|^2)^2
    \preccurlyeq
    1+
    \frac{d _B N k}{\ell T^2}
    \sum _{\substack{
        \gamma_1, \gamma_2 \in R(\ell;g)  \cap \Omega^{\star}(1, T): \\
        \det(\gamma_1) = \det(\gamma_2)
      }
    }
    1.
  \end{equation}
\end{proposition}

\begin{proposition}
  \label{prop:reductiontocountholo}
  Suppose $B$ is indefinite. Let $g \in [G]$, and assume that $k \gg (d_BN)^{\eta}$ for some arbitrarily small $\eta >0$.  Then, there exists $\ell|d_BN$ and $ 
  0<T \preccurlyeq \frac{(d_B N k)^{\frac{1}{2}}}{\ell}$ so that
  \begin{equation}\label{eq:reductiontocountholo}
    \sum _{\varphi \in \mathcal{F}^{-,\hol} }
    |\varphi(g)|^4
    \preccurlyeq_{\eta}
    1+
    \frac{d _B N k}{\ell T^2}
    \sum _{\substack{
        \gamma_1, \gamma_2 \in  R(\ell;g)  \cap \Omega^{\star}(k^{-1+\eps}, T): \\
        \det(\gamma_1) = \det(\gamma_2)>0
      }
    }
    1.
  \end{equation}
\end{proposition}

\begin{proposition}
  \label{prop:reductiontocountspherical}
  Suppose $B$ is definite and the weight is $k=0$. Then, for $g_1,g_2 \in [G]$ and $m \in \mathbb{N}_0$, there exists $\ell|d_BN$, $
  0< T \preccurlyeq \frac{(d_B N (m+1))^{\frac{1}{2}}}{\ell}$ and $\frac{1}{m^2+1} \preccurlyeq \delta \le 1$ so that
  \begin{equation}\label{eq:reductiontocountspherical}
    \sum _{\varphi \in \mathcal{F}^{+}_{m}}
    \left(|\varphi(g_1)|^2-|\varphi(g_2)|^2\right)^2
    \preccurlyeq
    1+
    \frac{d_B N}{\ell \delta^{\frac{1}{2}} T^2}
    \sum _{\substack{
        \gamma_1, \gamma_2 \in R(\ell;g)  \\
        \gamma_1, \gamma_2 \in \Omega^{\star}(\delta, T)\cup \Psi^{\star}(\delta, T): \\
        \det(\gamma_1) = \det(\gamma_2)
      }
    }
    1.
  \end{equation}
\end{proposition}

\begin{proposition}
  \label{prop:reductiontocountdefholo}
  Suppose $B$ is definite. Then, for $g \in [G]$, there exists $\ell|d_BN$ and $
  0< T \preccurlyeq \frac{(d_B N k)^{\frac{1}{2}}}{\ell}$ so that
  \begin{equation}\label{eq:reductiontocountdefholo}
    \sum _{\varphi \in \mathcal{F}^{+,\hol}}
    |\varphi(g)|^4
    \preccurlyeq
    1+\frac{d _B N k}{\ell T^2}\left(
      \sum _{\substack{
          \gamma_1, \gamma_2 \in R(\ell;g)  \cap \Omega^{\star}(k^{-1+\eps}, T): \\
          \det(\gamma_1) = \det(\gamma_2)
        }
      }1
      +
      \sum _{\substack{
          \gamma_1, \gamma_2 \in R(\ell;g)  \cap \Omega^{\star}(1, T): \\
          \det(\gamma_1) = \det(\gamma_2)
        }
      }
      k^{-2027}\right)
  \end{equation}
\end{proposition}

\subsection{Reduction to ternary lattices}

In this section, we reduce the vital counting problem involving quaternary quadratic form to problems involving only ternary quadratic forms. The key observation is that we may orthogonally decompose the quaternion algebra $B_{\infty}$ into its trace part and its traceless part $B_{\infty}^{0}$. Thus, for any $\alpha = \frac{1}{2}\tr(\alpha) + \alpha^0 \in \mathbb{R} \oplus B_{\infty}^{0}$, we have
\begin{equation}
  \det(\alpha)= \tfrac{1}{4}\tr(\alpha)^2+\det(\alpha^0). 
  \label{eq:normdecomp}
\end{equation}
We further note that the trace is invariant under conjugation. Hence, we have $\tr(R(\ell;g)) \subseteq \mathbb{Z}$. We conclude that
\begin{equation}
  R(\ell;g) \subseteq \tfrac{1}{2} \mathbb{Z} \oplus \tfrac{1}{2} R(\ell;g)^{0},
  \label{eq:latticesplitting}
\end{equation}
is a sublattice of the direct sum of the lattices $\frac{1}{2}\mathbb{Z}$ in $\mathbb{R}$ and $\tfrac{1}{2} R(\ell;g)^{0}$ in $B_{\infty}^0$. Using this decomposition, we deduce
\begin{equation}
  \sum_{\substack{\gamma_1,\gamma_2 \in R(\ell;g) \cap\Omega^{\star}(\delta,T) \\ \det(\gamma_1) = \det(\gamma_2)}} 1 \ll_{\eps} T^{\eps} \left|R(\ell;g)^0 \cap \Omega(\delta,2T) \right|^2 
  + T \sum_{\substack{\gamma_1,\gamma_2 \in R(\ell;g)^0  \cap\Omega(\delta,2T) \\ \det(\gamma_1) = \det(\gamma_2)}}1
  \label{eq:OmegaTypesplit}
\end{equation}
by distinguishing the two cases of equal and non-equal trace and applying the divisor bound to the equality
$$
\det(\gamma_1)=\det(\gamma_2) \Leftrightarrow \tfrac{1}{4}\tr(\gamma_1)^2 - \tfrac{1}{4}\tr(\gamma_2)^2 = \det(\gamma_2^0)-\det(\gamma_1^0).
$$
We remark that we have forfeited the congruence condition $\det(\gamma_1^0) \equiv \det(\gamma_2^0) \ \mod(1)$, and this forfeiture will be reflected in the suboptimality of our final counting estimates on larger scales when $\ell>1$. We circumnavigate these larger scales by an appropriate choice of a covering domain (cf. Lemma \ref{lem:generalL2bound}).

Note that we may further bound the diagonal contribution by considering its largest fiber:
\begin{equation}
  \sum_{\substack{\gamma_1,\gamma_2 \in R(\ell;g)^0 \cap\Omega(\delta,2T) \\ \det(\gamma_1) = \det(\gamma_2)}}1 \le |R(\ell;g)^0 \cap\Omega(\delta,2T)| \times \max_{\substack{n \in \frac{1}{\ell} \mathbb{Z} \\ |n|\le 4T^2}} |R(\ell;g)^0 \cap\Omega(\delta,2T) \cap \det{}^{-1}(\{n\})|.
  \label{eq:OmegaTypeIIsplit}
\end{equation}
Arguing along the same lines, we also arrive at
\begin{equation}
  \sum_{\substack{\gamma_1,\gamma_2 \in  R(\ell;g) \cap\Psi^\star(\delta,T) \\ \det(\gamma_1) = \det(\gamma_2)}} 1 \ll_{\eps} T^{\eps} \left|R(\ell;g)^0 \cap \Psi(\delta,2T) \right|^2 
  + \delta^{\frac{1}{2}}T \sum_{\substack{\gamma_1,\gamma_2 \in R(\ell;g)^0 \cap\Psi(\delta,2T) \\ \det(\gamma_1) = \det(\gamma_2)}}1
  \label{eq:PsiTypesplit}
\end{equation}
and
\begin{equation}
  \sum_{\substack{\gamma_1,\gamma_2 \in R(\ell;g)^0 \cap\Psi(\delta,2T) \\ \det(\gamma_1) = \det(\gamma_2)}}1 \le |R(\ell;g)^0 \cap\Psi(\delta,2T)| \times \max_{\substack{n \in \frac{1}{\ell} \mathbb{Z} \\ |n|\le 4T^2}} |R(\ell;g)^0 \cap\Omega(1,2T) \cap \det{}^{-1}(\{n\})|.
  \label{eq:PsiTypeIIsplit}
\end{equation}
Note that in this last inequality, we have passed from $\Psi(\delta,2T)$ to the larger set $\Omega(1,2T)=\Psi(1,2T)$; the resulting bound remains adequate for us thanks to the additional saving of $\delta^{\frac{1}{2}}$ in \eqref{eq:PsiTypesplit}.

\subsection{Proof of Theorem \ref{thm:main-result-fourth-moment-adelic-formulation}}
\label{sec:proof-thm1}

Theorem \ref{thm:main-result-fourth-moment-adelic-formulation} is an immediate consequence of the following pair of lemmas
together with
Propositions \ref{prop:reductiontocountmaass} through \ref{prop:reductiontocountdefholo}.

\begin{lemma}\label{lem:thetasup-paper-merge-20220418:we-have-beginm}
  We have
\begin{multline}
  \sum_{\substack{\gamma_1,\gamma_2 \in  R(\ell;g) \cap\Omega^{\star}(\delta,T) \\ \det(\gamma_1) = \det(\gamma_2)}} 1 \prec \ell T^2 \left(1 + \ell^{\frac{1}{2}} H(g) \delta^{\frac{1}{2}} + \ell^{\frac{1}{2}} H(g) \delta T +  \frac{\ell^2 }{d_B N} \delta T^2  \right) \\
  \times \left(1 + \ell^{\frac{1}{2}} H(g) \delta^{\frac{1}{2}} + \ell^{\frac{1}{2}} H(g) \delta^{\frac{1}{2}} T + \frac{\ell^2 }{d_B N} \delta T^2 \right). 
  \label{eq:Omegafullcount}
\end{multline}  
\end{lemma}
\begin{proof}
  Recall, from the discussion of Section \S\ref{sec:prelim-proof-reduction}, that we may express $R(\ell;g)^0$, for $g \in G(\mathbb{A})$, as $(g')^{-1}R'(\ell)^0g'$, where $R'$ is an Eichler order of the same level and $g' \in G(\mathbb{R})$, with $H(g)=H(g')$ in the case that $B$ is split. We may thus apply the results of Section \S\ref{sec:counting-results}.
  
Since $\tr(R(\ell;g)) \subseteq \mathbb{Z}$, we find that the first successive minimum of $R(\ell;g)$ with respect to $\Omega(\delta,1)$ is at least the minimum of $1$ and the first successive minimum of $R(\ell;g)^0$ with respect to $\Omega(\delta,1)\cap B_{\infty}^{0}$. The latter is $\gg \min\{ \ell^{-\frac{1}{2}}, \ell^{-1} \delta^{-\frac{1}{2}} H(g)^{-1} \} =: \Lambda$ by Theorem \ref{thm:typeI}, where the term involving $H(g)$ is to be omitted if $B$ is non-split. Thus we find that $R(\ell;g) \cap \Omega^{\star}(\delta,T)$ is empty for $T \ll \Lambda$, in which case there is nothing to show.  Next, assume instead that $T \gg \Lambda$. Then, inequalities \eqref{eq:OmegaTypesplit} and \eqref{eq:OmegaTypeIIsplit}, in conjunction with Theorems \ref{thm:typeI} and \ref{thm:typeII}, imply the required bound.
\end{proof}

\begin{lemma}
  Assume that $B$ is non-split.  Then

\begin{equation}
  \sum_{\substack{\gamma_1,\gamma_2 \in  R(\ell;g) \cap\Psi^{\star}(\delta,T) \\ \det(\gamma_1) = \det(\gamma_2)}} 1 \prec \ell T^2 \left(1 +  \frac{\ell}{(d_BN)^{\frac{1}{2}}} T +   \frac{\ell^2 }{d_B N} \delta^{\frac{1}{2}} T^2  \right)^2.
  \label{eq:Psifullcount}
\end{equation}
\end{lemma}
\begin{proof}
As in the proof of Lemma \ref{lem:thetasup-paper-merge-20220418:we-have-beginm}, we find that the first successive minimum of $R(\ell;g)$ with respect to $\Psi(\delta,1)$ is at least the minimum of $\delta^{-\frac{1}{2}}$ and the first successive minimum of $R(\ell;g)^0$ with respect to $\Psi(\delta,1)\cap B_{\infty}^{0}$. The latter is $\gg \ell^{-\frac{1}{2}}$ by Theorem \ref{thm:typeI}. Therefore, $R(\ell;g) \cap \Psi^{\star}(\delta,T)$ is empty for $T \ll \ell^{-\frac{1}{2}} \le 1 \le \delta^{-\frac{1}{2}}$, in which case there is nothing to show. If $T \gg \ell^{-\frac{1}{2}}$, then inequalities \eqref{eq:PsiTypesplit} and \eqref{eq:PsiTypeIIsplit} together with Theorems \ref{thm:typeI} and \ref{thm:typeII} yield the required bound.
\end{proof}

\subsection{Proof of Corollary \ref{cor:main-corollary-fourth-moment-adelic-formulation}}

Let $\varphi \in \mathcal{F}$, respectively $\mathcal{F}^{\hol}$, be $L^2$-normalized. Assume first that $B$ is non-split. Then, since $[G] \slash K_{\infty} K_{R}$ is compact and equipped with a probability measure, we may find $g_2$ in $[G]$ such that $|\varphi(g_2)| \le 1$. Hence, Corollary \ref{cor:main-corollary-fourth-moment-adelic-formulation} follows immediately from Theorem \ref{thm:main-result-fourth-moment-adelic-formulation} by positivity and the particular choice of $g_2$.

We now turn our attention to the case that $B$ is split, in other words when $d_B=1$. Here, we need to supplement Theorem \ref{thm:main-result-fourth-moment-adelic-formulation} with the additional information that for $H(g) \ge N^{-\frac{1}{2}}$, we have
\begin{align}
  |\varphi(g)| &\preccurlyeq_{\lambda_{\varphi}} N^{\frac{1}{4}}
                 &&\text{if $\varphi \in \mathcal{F}^{-}$,}
             \label{eq:Maasslargeheight}
  \\
  |\varphi(g)| &\preccurlyeq (kN)^{\frac{1}{4}}
          &&\text{if $\varphi \in \mathcal{F}^{-,\hol}$.}
             \label{eq:hololargeheight}
\end{align}
The former is recorded in \cite[Prop. 3.1 \& 3.2]{MR3372076}, for example, and the latter may be deduced from the Fourier expansion along the lines of Xia \cite{Xiasupnorm}. We include a brief proof here for the sake of completeness.
\begin{lemma} Assume $B$ is split and let $\varphi \in \mathcal{F}^{-,\mathrm{hol}}$ be an $L^2$-normalized holomorphic cuspidal newform of squarefree level $N$ and even weight $k \ge 2$. Then, we have for all $g \in [G]$,
  \begin{equation}
    |\varphi(g)| \preccurlyeq H(g)^{o(1)} \left( k^{\frac{1}{4}}H(g)^{-\frac{1}{2}}+ k^{-\frac{1}{4}}H(g)^{\frac{1}{2}} \right) .
    \label{eq:Holonearcuspramanujan}
  \end{equation}
  If $H(g) \ge \frac{k}{2 \pi}$, then we have the stronger bound
  \begin{equation}
    |\varphi(g)| \preccurlyeq 1.
    \label{eq:Holonearcuspexpdecay}
  \end{equation}
\end{lemma}
\begin{proof} Suppose that $g$ corresponds to $z=x+iy \in \Gamma_0(N) \backslash \mathbb{H}$. As $|\varphi(g)|$ is further invariant under the Atkin--Lehner operators we may further assume that $z$ has maximal imaginary part under the action of the group $A_0(N)$, thus $H(g)=y$.  We shall subsequently make use of the Fourier expansion of $\varphi$ at $\infty$:
$$
|\varphi(g)| = \left| y^{\frac{k}{2}} \sum_{n=1}^{\infty} a_n e(n(x+iy)) \right |.
$$
We may bound the Fourier coefficients by appealing to Deligne's bound for the Hecke eigenvalues. This implies $|a_n| \ll_{\eps} n^{\frac{k-1}{2}+\eps} |a_1|$.\footnote{At the ramified primes $p|N$ the stronger bound $|a_{p^l}| \le (p^{l})^{\frac{k}{2}-1} |a_1|$ holds and is a far less deep result (cf. \cite[Thm. 3]{MR0268123}).} We find
$$
|\varphi(g)| \ll_{\eps} |a_1| (2\pi)^{-\frac{k}{2}} y^{\frac{1}{2}-\eps} \sum_{n=1}^{\infty} (2 \pi n y)^{\frac{k-1}{2}+\eps} e^{-2 \pi n y}.
$$
The above sum, we may bound by comparison to the corresponding integral. For this manner, we note that the function $x^{\alpha}e^{-x}$ increases up to $x= \alpha$ and then decreases. We may also bound the first Fourier coefficient $a_1$ by a result of Hoffstein--Lockhart \cite{HL94} (cf. \cite[Eq. (31)]{MR2207235}\footnote{Notice the different normalization of the measure and $a_1$ in Section \S 2.}). The bound reads $|a_1| \ll_{\eps} (Nk)^{\eps} (4\pi)^{\frac{k}{2}} \Gamma(k)^{-\frac{1}{2}}$. We thus arrive at
$$\begin{aligned}
  |\varphi(g)| &\ll_{\eps} (Nk)^{\eps}y^{-\eps} \frac{2^{\frac{k}{2}}y^{\frac{1}{2}}}{\Gamma(k)^{\frac{1}{2}}} \left( \frac{1}{y} \Gamma\left( \frac{k+1}{2} \right) + \left(\frac{k-1}{2}\right)^{\frac{k-1}{2}+\eps}e^{-\frac{k-1}{2}} \right) \\
  &\ll_{\eps} (Nk)^{\eps} y^{-\eps} \left( k^{\frac{1}{4}}y^{-\frac{1}{2}}+ k^{-\frac{1}{4}}y^{\frac{1}{2}} \right),
\end{aligned}$$ where we have made use of Stirling's approximation. If $y \ge \frac{k}{2 \pi }$, then the maximum summand occurs when $n=1$ and we may derive the improved bound
$$\begin{aligned}
  |\varphi(g)| &\ll_{\eps} (Nk)^{\eps} y^{-\eps} \frac{2^{\frac{k}{2}}y^{\frac{1}{2}}}{\Gamma(k)^{\frac{1}{2}}} \left( \frac{1}{y} \Gamma\left( \frac{k+1}{2} \right) + \left(2 \pi y\right)^{\frac{k-1}{2}+\eps}e^{-2 \pi y } \right) \\
  & \ll_{\eps} (Nk)^{\eps} y^{-\eps}.
\end{aligned}$$
\end{proof}
To deduce \eqref{eq:hololargeheight} from the lemma, we consider separately the cases $N^{-\frac{1}{2}} \leq H(g) \leq \frac{k}{ 2 \pi }$ and $H(g) \geq \frac{k}{2 \pi }$, applying \eqref{eq:Holonearcuspramanujan} in the former case and \eqref{eq:Holonearcuspexpdecay} in the latter.

We may now deduce the split case of Corollary \ref{cor:main-corollary-fourth-moment-adelic-formulation}, as follows.  Our task is to bound $\varphi(g_1)$ suitably for $g_1 \in [G]$.  We may assume that $H(g_1) \leq N^{-\frac{1}{2}}$, as otherwise the estimates \eqref{eq:Maasslargeheight} and \eqref{eq:hololargeheight} are adequate.  In that case, we choose $g_2 \in [G]$ arbitrarily with $H(g_2) = N^{-\frac{1}{2}}$ and apply Theorem \ref{thm:main-result-fourth-moment-adelic-formulation} and positivity to estimate the difference $\varphi(g_1) - \varphi(g_2)$.  We then apply the triangle inequality, together with \eqref{eq:Maasslargeheight} and \eqref{eq:hololargeheight}  to control the contribution of $\varphi(g_2)$.  Recalling that $d_B = 1$, we readily obtain the required estimates.


\section{Arithmetic quotients as real manifolds}
\label{sec:real-manifolds}

\subsection{Measure normalizations} 

For indefinite $B$, we fix an isomorphism $G(\mathbb{R})\cong \operatorname{PGL}_2(\mathbb{R})$ sending $K_\infty$ to $\operatorname{PSO}_2(\mathbb{R})$. We fix the Haar measure $\dif g = \frac{\dif y \dif x}{y^2}\frac{\dif \theta}{2 \pi}$ for $g=\sm y^{1/2} & xy^{-1/2} \\ 0 & y^{-1/2} \esm \kappa(\theta)$ on $\operatorname{SL}_2(\mathbb{R})$. The push-forward of this measure to the hyperbolic plane $\mathbb{H}\cong\faktor{\operatorname{SL}_2(\mathbb{R})}{\operatorname{SO}_2(\mathbb{R})}$ is then the measure $\frac{\dif x \dif y}{y^2}$.  The Haar measure on $\operatorname{PGL}_2(\mathbb{R})$ is fixed so that its restriction to $\operatorname{PSL}_2(\mathbb{R})$ coincides with the push-forward of the Haar measure from $\operatorname{SL}_2(\mathbb{R})$.

If $B$ is definite we fix an isomorphism $G(\mathbb{R})\cong \operatorname{SO}_3(\mathbb{R})$ sending $K_\infty$ to $\operatorname{SO}_2(\mathbb{R})$. We fix a Haar measure on $\operatorname{SO}_3(\mathbb{R})$ so that the measure of the $2$-sphere $S^2\cong \faktor{\operatorname{SO}_3(\mathbb{R})}{\operatorname{SO}_2(\mathbb{R})}$ is $4\pi$.

\subsection{Volumes} \label{sec:unadelic-volume}

Recall, that we fixed the measure on $[G]$ to be the probability Haar measure. Hence, the volume of the quotient $[G]/K_R$ is $1$. In due course, we shall also require the volume of said quotient when viewed as a real manifold with respect to our fixed Haar measure on $G(\mathbb{R})$. More specifically, we will need the volume with respect to the measure on $G'(\mathbb{R})$, where $G'$ is the linear algebraic group defined over $\mathbb{Q}$ whose rational points are the proper unit quaternions $B^1$. There is an obvious isogeny map $G'\to G$, where $G'$ is the simply-connected form and $G$ is the adjoint one. Define $R_p^1=R_p\cap G'(\mathbb{Q}_p)$ to be the proper unit quaternions in the local order $R_p$, and set $K_R^1=\prod_{p} R_p^1$. Then, the map $[G']/K_R^{1}\to[G]/K_R$ is a homeomorphism that pushes forward the probability Haar measure on $[G']/K_R^{1}$ to the probability Haar measure on $[G]/K_R$, see Lemma \ref{lem:isogeny2isomorphism}. In general, this map is not bijective if $K_R$ is replaced by a general compact open subgroup of $G(\mathbb{A}_f)$ and the fact that the map is indeed a homeomorphism is due to $K_R$ being the projectivized group of units of an Eichler order.

By Borel's finiteness of class numbers \cite{Borel-cl-finite}, $[G']/K_R^1$ is a finite collection of $G'(\mathbb{R})$-orbits with representatives $\delta_1,\ldots,\delta_h\in G'(\mathbb{A})$. Define $\Gamma_i=G'(\mathbb{Q})\cap \delta_i K_R^1 \delta_i^{-1}$; the intersection is taken in $G'(\mathbb{A}_f)$ but regarded as a subset of $G'(\mathbb{Q})$ and hence also of $G'(\mathbb{R})$.  In particular, $\Gamma_i$ is a lattice in $G'(\mathbb{R})$.  It follows that
\begin{equation*}
	[G]/K_R\cong [G']/K_R^{1}=\bigsqcup_{i=1}^h {\Gamma_i}\backslash{G'(\mathbb{R})}.
\end{equation*}
This is a finite disjoint union of finite-volume homogeneous spaces for the real Lie group $G'(\mathbb{R})$.  We define $\operatorname{covol}(\Gamma_i)$ to be the measure of a fundamental domain for the action of $\Gamma_i$ on $G'(\mathbb{R})$ with respect to the fixed Haar measure on $G'(\mathbb{R})$, either $\frac{\dif x \dif y}{y^2} \frac{\dif \theta}{2\pi}$ in the indefinite case or the measure giving volume $4\pi$ to $\faktor{\operatorname{SO}_3(\mathbb{R})}{\operatorname{SO}_2(\mathbb{R})}$ in the definite case. We finally set
\begin{equation*}
	V = V_{d_B,N}=\sum_{i=1}^h \operatorname{covol}(\Gamma_i).
\end{equation*}

If $B$ is indefinite, then $G'(\mathbb{R})\cong \operatorname{SL}_2(\mathbb{R})$ is non-compact and strong approximation implies that $h=1$ and we can write $[G]/K_R\cong \Gamma \backslash \operatorname{SL}_2(\mathbb{R})$. In this case, $V$ is the volume of the hyperbolic surface $\lfaktor{\Gamma}{\mathbb{H}}$ with respect to the volume form $\frac{\dif x \dif y}{y^2}$. If $B$ is definite, then in general $h$ can be large and $[G]/K_\infty K_R$ is a finite collection of quotients of $2$-spheres by discrete rotation groups.


Recall that we have denoted by $d_B$ the reduced discriminant of $B$ and by $N$ the squarefree level of the Eichler order $R$.  The volume is given in both cases by
\begin{equation}
	V = V_{d_B,N} = 
	\frac{\pi}{3} d_BN \prod_{p | d_B} \left(1-\frac{1}{p}\right) \prod_{p | N} \left(1+\frac{1}{p} \right)
	= (d_BN)^{1+o(1)}.
	\label{eq:Unadel-covolume}
\end{equation} 
This follows from a corresponding mass formula, see \cite[Thm 39.1.8]{voightQA} in the indefinite case and \cite[Thm 25.1.1 \& Thm 25.3.18]{voightQA} in the definite one.  The space is furthermore compact if and only if $B$ non-split, i.e., $d_B > 1$.

\subsection{Siegel domains}

The main purpose of this section is to provide a specific Siegel-domain covering in order to bound the $L^2$-norm of the (difference of) theta kernels in \S\ref{sec:L2-theta-red-counting}.  Let $M$ be a squarefree natural number, and set $U = \sm \widehat{\mathbb{Z}} & \widehat{\mathbb{Z}} \\ M \widehat{\mathbb{Z}} & \widehat{\mathbb{Z}} \esm \cap \SL_{2}(\mathbb{A}_f)$.  The theta functions of interest will turn out to be right invariant in the symplectic variable under $U$ with $M = d_B N$, but the present discussion applies to any squarefree $M$.


\subsubsection{Cusps and Atkin--Lehner operators}
\label{sec:cusps-AL-operators} Since $M$ is squarefree, a representative set of cusps for $\Gamma_0(M)$ is given by the ratios $\frac{\ell}{M}$, where $\ell$ runs through the positive divisors of $M$. The width of the cusp $\frac{\ell}{M}$ (understood here as with respect to the group $\Gamma_0(M)$) is given by $\ell$ \cite{Iw97}. For each $\ell|M$, we choose an element $\tau_{\ell} \in \SL_2(\mathbb{Z})$ satisfying
\begin{equation}
	\tau_{\ell} \equiv \begin{pmatrix} 0 & 1 \\ -1 & 0 \end{pmatrix} \mod{\ell}, \quad \tau_{\ell} \equiv \begin{pmatrix} 1 & 0 \\ 0 & 1 \end{pmatrix} \mod{ M / \ell}.
	\label{eq:tau-ell}
\end{equation}
Then, the cusp $\tau_{\ell} \infty$ is $\Gamma_0(M)$-equivalent to the cusp $\frac{\ell}{M}$. Hence, writing $n(x)=\sm 1 & x\\ & 1 \esm$, we see that the elements $\tau_{\ell}n(j)$, where $j=0,\dots,\ell-1$ and $\ell|M$, give a complete system of representatives for $\Gamma_0(M) \backslash \SL_{2}(\mathbb{Z})$.  The normalized matrices $\tilde{\tau}_{\ell} \coloneqq \tau_{\ell}a(\ell)$, where $a(y)= \diag(y^{\frac{1}{2}},y^{-\frac{1}{2}})$, are scaling matrices for the respective cusps. Furthermore, the matrices $\tilde{\tau_{\ell}}$ are Atkin--Lehner operators for $\Gamma_0(M)$ and give a set of representatives for $A_0(M) \slash \Gamma_0(M)$.

\subsubsection{Coverings}

The basic idea of the following lemma is to apply the Fricke involution to the tiling of $\Gamma_0(M) \backslash \mathbb{H}$ by translates of the standard fundamental domain for $\SL_2(\mathbb{Z}) \backslash \mathbb{H}$.
\begin{lemma} Let $F : [\SL_2] \rightarrow \mathbb{R}_{\geq 0}$ be a measurable function that is right $U$
	invariant and of weight $0$. Then,
	$$
	\int_{[\SL_2]} F(g) dg \le \frac{1}{V_{1,M}} \sum_{\ell|M} \int_{\frac{\sqrt{3}}{2} \frac{l^2}{M}}^{\infty} \int_0^{\ell} F  \left(\tau_{\ell} \sm y^{1/2} & x y^{-1/2} \\ & y^{-1/2} \esm K_{\infty} U) \right) dx \frac{dy}{y^2}
	$$
	\label{lem:generalL2bound}
\end{lemma}
\begin{proof} Let $f: \mathbb{H} \to \mathbb{R}_{\geq 0}$ be given by 
	$$f(z)=F \left( \sm y^{1/2} & x y^{-1/2} \\ & y^{-1/2} \esm K_{\infty} U \right).$$
	Then, $f$ is $\Gamma_0(M)$ invariant on the left and we have
	$$
	\int_{[\SL_2]} F(g) dg = \frac{1}{V_{1,M}} \int_{\Gamma_0(M)\backslash \mathbb{H}} f(z) \frac{dxdy}{y^2} = \frac{1}{V_{1,M}} \int_{\Gamma_0(M)\backslash \mathbb{H}} f( \tilde \tau_{M}z) \frac{dxdy}{y^2}.
	$$
The standard Siegel set $\{z \in \mathbb{H} | \ 0\le\Re(x)\le 1 \text{ and } \Im(z) \ge \frac{\sqrt{3}}{2}\}$ contains a fundamental domain for $\SL_2({\mathbb{Z}})$. Using that the $\tau_{\ell}n(j)$, for $j=0,\dots,\ell-1$ and $\ell |M$, form a representative set for $\Gamma_0(M) \backslash \SL_2(\mathbb{Z})$ and that $f$ is nonnegative, we may bound
	$$ \int_{\Gamma_0(M)\backslash \mathbb{H}} f( \tilde \tau_{M}z) \frac{dxdy}{y^2} \le \sum_{\ell|M} \int_{\frac{\sqrt{3}}{2}}^{\infty} \int_0^{\ell} f( \tilde{\tau}_{M}\tau_{\ell} z) dx \frac{dy}{y^2}.
	$$
	Since $\tilde{\tau}_{M} \tilde{\tau}_{\ell} = \gamma \tilde{\tau}_{\frac{M}{\ell}}$ for some $\gamma \in \Gamma_0(M)$, we have the identity
	\begin{equation*}
		f(\tilde{\tau}_{M}\tau_{\ell} z) = 
		f( \tilde{\tau}_{\tfrac{M}{\ell}}a(\ell)^{-1} z) =
		f(\tau_{\tfrac{M}{\ell}} a(\tfrac{M}{\ell})a(\ell)^{-1} z).
	\end{equation*}
	Substituting this identity above and applying the change of variables $\frac{M}{l^2}z \mapsto z$ gives the desired result.
\end{proof}

\section{Theta kernels and their $L^2$-norms}\label{sec:Theta-L2norm}

\subsection{Theta kernels and lifts}

In this section, we summarize the required results on theta kernels and their lifts. The necessary theory is developed in Appendix \ref{sec:appendix-theta}.

\subsubsection{Theta functions} \label{sec:theta-functions-definitions} 
The theta kernels constructed in Appendix \ref{sec:appendix-theta} are modular functions on $\O_{\det}(\mathbb{A})\times \SL_2(\mathbb{A})$. The group $G$ acts by conjugation on the quadratic space $(B,\det)$, preserving the quadratic form. This gives an embedding $G\hookrightarrow \operatorname{O}_{\det}$.  We are mainly concerned here with the pullback of the theta kernels to $G(\mathbb{A})\times \SL_2(\mathbb{A})$.  We denote that pullback by $\Theta(g,s)$. The function $\Theta(g,s)$ is right $K_R \times U_R^1$ invariant, where $U_R^1 \coloneqq \sm \widehat{\mathbb{Z}} & \widehat{\mathbb{Z}} \\ d_BN \widehat{\mathbb{Z}} & \widehat{\mathbb{Z}} \esm \cap \SL_2(\mathbb{A}_f)$, and of moderate growth.  We caution that it is \emph{not} a theta kernel for the Howe dual pair of the orthogonal group of the traceless quaternions and $\SL_2$. 

We shall require several explicit expressions of the theta kernels $\Theta$. Define functions $P, u, X$ on $B_\infty$ by setting, for $\gamma = [a,b,c] + d \in B_\infty$,
\begin{equation}\label{eq:pgamma-coloneqq-a2}
  P(\gamma) \coloneqq a^2 + b^2 + c^2 + d^2, \quad u(\gamma) \coloneqq b^2 + c^2, \quad X(\gamma) \coloneqq d + i a.
\end{equation}
In other words, by identifying $\i \in B_{\infty}$ with $i \in \mathbb{C}$, we have that
\begin{itemize}
\item $X$ is the projection from $B_\infty = \mathbb{C} \oplus \mathbb{C} \j$ to the summand $\mathbb{C}$,
\item $u$ is the squared magnitude of the projection onto the other summand $\mathbb{C} \j$, and
\item $P$ is the sum of the squared magnitudes of the two projections.
\end{itemize}
Upon recalling the notation $R(\ell;g)$ from \ref{sec:prelim-proof-reduction}, we define for $g \in G(\mathbb{A})$ and $z=x+iy \in \mathbb{H}$ the theta functions
\begin{equation}
	\theta_{g, \ell}^{-}(z) \coloneqq
	y^{1 + \frac{k}{2}}
	\sum _{\gamma \in R(\ell;g)}
	X(\gamma)^k
	e^{-2 \pi y P(\gamma)}
	e(x \det(\gamma)), 
	\label{eq:defthetamaass}
\end{equation}
if $k \ge 6$
\begin{equation}
	\theta_{g,\ell}^{-,\hol}(z) \coloneqq
	\frac{k-1}{4 \pi }
	y^{ k/2}
	\sum _{
		\substack{
			\gamma \in R(\ell;g) \\
			\det(\gamma) > 0 
		}
	}
	\det(\gamma)^{k-1}
	\overline{X(\gamma)}^{-k}
	e(z \det(\gamma)),
	\label{eq:defthetaholo}
\end{equation}
\begin{equation}
	\theta_{g,\ell}^{+,m}(z) \coloneqq (2m+1)
	y^{1+ m}
	\sum _{
		\gamma \in R(\ell;g) 
	}
	\det(\gamma)^{m}
	P_{m}\left( \frac{|X(\gamma)|^2-u(\gamma)}{\det(\gamma)} \right)
	e(z \det(\gamma)), 
	\label{eq:defthetaplus}
\end{equation}
\begin{equation}
	\theta_{g,\ell}^{+,\hol}(z) \coloneqq (k+1)
	y^{1+ \frac{k}{2}}
	\sum _{
		\gamma \in R(\ell;g) 
	}
	X(\gamma)^{k}
	e(z \det(\gamma)), 
	\label{eq:defthetaplusholo}
\end{equation}
where $P_{m}$ is the $m$-th Legendre polynomial. When $\ell = 1$, we abbreviate by dropping the subscript, e.g.,  $\theta_{g}^{-} \coloneqq \theta^{-}_{g,1}$. We are now ready to express $\Theta$ by means of strong approximation. Set
$$ s_{\infty}= \sm y^{1/2} & xy^{-1/2} \\ & y^{-1/2} \esm \sm \cos(\theta) & \sin(\theta) \\ - \sin(\theta) & \cos(\theta) \esm \in \SL_2(\mathbb{R}).$$
Then, $\Theta(g, s_{\infty}U_R^1) = \theta_g(z) e^{i \kappa \theta}$ for some $\kappa \in 2\mathbb{Z}$ and a choice of $\theta_{g}$ from \eqref{eq:defthetamaass}-\eqref{eq:defthetaplusholo}. The precise choice and value of $\kappa$ depends on the family $\mathcal{G}$ under consideration and may be read off Table \ref{table:choice-theta}.  (For our study of $\mathcal{F}^{-,\hol}$, the precise choice of $\Theta$ depends upon the size of $k$.)  \\ \renewcommand{\arraystretch}{1.5}
\begin{table}[h]
\begin{tabular}{ |c | c |c|c|c|c| }
	\hline
	$B$ & \multicolumn{3}{ |c|}{indefinite} &  \multicolumn{2}{ |c|}{definite} \\
	\hline
	$\mathcal{G}$ & $\mathcal{F}^{-}$ & \multicolumn{2}{ |c|}{ $\mathcal{F}^{-,\hol}$} & $\mathcal{F}^{+}_m$ & $\mathcal{F}^{+,\hol}$ \\ \hline
	$\theta_g$ & $\theta_g^{-} \ (k=0)$ & $ \theta_g^- $ & $ \theta_g^{-,\hol} $ & $\theta_g^{+,m}$  & $\theta_g^{+,\hol}$ \\ \hline
	$\kappa$ & $0$ & $k$ & $k$ & $2m+2$ & $k+2$ \\ \hline
\end{tabular}
\vspace{0.4em}
\caption{Families and the choice of $\Theta$.}
\label{table:choice-theta}
\end{table} \renewcommand{\arraystretch}{1}

\vspace{-5mm} Furthermore, the theta functions $\Theta$ satisfy the relation
\begin{equation}
	\Theta(g,\tau_{\ell}s_{\infty}U_R^1)
	=
	\frac{\mu(\gcd(\ell,d_B))}{\ell}
	\theta_{g,\ell}(z)e^{i \kappa \theta},
	\label{eq:thetamaasstranform}
\end{equation}
for $\ell | d_B N$ with $\mu$ the M{\"o}bius function and $\tau_{\ell}$ as in \eqref{eq:tau-ell}.

\subsubsection{Jacquet--Langlands lifts}\label{sec:JL-lifts}
Set $U_R$ to be the image of
\begin{equation*}
\left\{g\in \sm
 \widehat{\mathbb{Z}} & \widehat{\mathbb{Z}} \\ d_BN \widehat{\mathbb{Z}} & \widehat{\mathbb{Z}} 
\esm
  \colon \det g \in \widehat{\mathbb{Z}}^\times\right\}
\end{equation*}
in $\PGL_2(\mathbb{A}_f)$.  This is a compact open subgroup of $\PGL_2(\mathbb{A}_f)$.  For each $\varphi$ in the families $\mathcal{F}^{-}, \mathcal{F}^{-,\hol}, \mathcal{F}^{+}_m, \mathcal{F}^{+,\hol}$, we consider the Jacquet--Langlands transfer $\pi^{\JL}$ to $[\PGL_2]$ of the representation $\pi$ generated by $\varphi$. In the case that $G$ is split, we let $\pi^{\JL}=\pi$. The space of vectors in $\pi^{\JL}$ that are $U_R$-invariant and $K_{\infty}$-isotypical of minimal nonnegative weight is one-dimensional \cite{MR0401654,MR337789}.  We define the \emph{arithmetically-normalized} Jacquet--Langlands lift $\varphi^{\JL}\in L^2([\SL_2])$ of $\varphi$ to be the $U_R^1$-invariant restriction\footnote{It is more natural to consider $\varphi^{\JL}$ as a function on $[\PGL_2]$, but we allow ourselves to reduce to the restriction because the map $\faktor{[\SL_2]}{U_R^1}\to\faktor{[\PGL_2]}{U_R}$ is a homeomorphism, see Appendix \ref{sec:appendix-theta}.} to $[\SL_2]$ of a vector in this $1$-dimensional space, that has a Whittaker function at $\sm y^{1/2} & xy^{-1/2} \\ & y^{-1/2} \esm \in \SL_2(\mathbb{R}) \hookrightarrow \SL_2(\mathbb{A})$ given by
\begin{itemize}
	\item 
	$2 \sqrt{y} K_{it}(2 \pi y)e(x) \text{ if } \varphi \in \mathcal{F}^{-}_{\frac{1}{4}+t^2}$, and
	\item 
	$y^{\frac{\kappa}{2}} e(x+iy)$ if $\varphi$ is in either of the families $\mathcal{F}^{-,\hol}$, $ \mathcal{F}^{+}_m$, $\mathcal{F}^{+,\hol}$, where $\kappa=k$, $2m+2$, $k+2$ depends on the family as before.
\end{itemize}

The bounds by Hoffstein--Lockhart \cite{HL94} and Iwaniec \cite{MR1067982} then imply the following bounds for the $L^2$-norm of the arithmetically-normalized Jacquet--Langlands lift (cf. \cite[(30), (31)]{MR2207235}\footnote{Notice the different normalization of the measure and $a_1$ in Section \S 2.}). One may also compare with the geometric normalization in \cite[Thm. 6.1]{2019arXiv190810346P}. If $B$ is indefinite and $\varphi \in \mathcal{F}^{-}_{\frac{1}{4}+t^2}$, we have 
\begin{equation} \|\varphi^{\JL}\|_2^2 = (d_BN(1+|t|))^{o(1)} \cosh(\pi t)^{-1}. 
	\label{eq:indefsphericalJLnorm}
\end{equation}
In the other cases, i.e., when $\varphi$ lies in either of the families $\mathcal{F}^{-,\hol}$, $ \mathcal{F}^{+}_m$, or $\mathcal{F}^{+,\hol}$, we have
\begin{equation}
	\|\varphi^{\JL}\|_2^2 = (d_BN\kappa)^{o(1)} \frac{\Gamma(\kappa)}{(4 \pi)^{\kappa}},
	\label{eq:holomorphicJL-norm}
\end{equation}
where $\kappa$ depends on the family in accord with Table \ref{table:choice-theta}.

\subsubsection{Explicit theta lifting}


The key identity is summarized in the following proposition.

\begin{proposition}
  \label{prop:thetamaassmainidentity} Let $g \in [G]$. Let $\mathcal{G}$, $\Theta$, and $\kappa$ according to Table \ref{table:choice-theta}. Then, for $\varphi \in \mathcal{G}$, we have
  \begin{equation}\label{eq:thetamaassinner}
    \frac
    {
      \langle \Theta(g,\cdot), \varphi^{ \JL} \rangle
    }
    {
      \langle \varphi^{\JL}, \varphi^{\JL} \rangle
    }
    =
    \frac{|\varphi(g)|^2}{V_{d_B,N}}.
  \end{equation}
\end{proposition}

\begin{proof}
  The proof is carried out in the Appendix.  In short, Proposition
  \ref{prop:theta-lift-all-families} implies that for $\varphi\in\mathcal{G}$, the theta lift $\varphi_{\Phi}$ of $\varphi$ -- defined in \eqref{eq:theta-lift-G'}, depending upon the precise family $\mathcal{G}$ -- satisfies $\varphi_{\Phi}=(V_{d_B,N})^{-1}\varphi^{\mathrm{JL}}$. The claim then follows from Propositions \ref{prop:theta-adelic2classic} and \ref{prop:theta-double-lift}.
\end{proof}

\subsection{$L^2$-norms of theta kernels}
\label{sec:L2-theta-red-counting}

\subsubsection{Proofs of Propositions \ref{prop:reductiontocountmaass} through \ref{prop:reductiontocountdefholo}}

The proofs are similar, so we discuss the first in detail and then explain the non-overlapping parts of the rest. Recall the notation $\preccurlyeq$ from \S\ref{sec:results-forms}.  We denote by $\Theta^-, \Theta^{-,\hol}, \Theta^{+,m}, \Theta^{+,\hol}$ the various functions  ``$\Theta$'' defined as in \S\ref{sec:theta-functions-definitions}. 

\begin{proof}[Proof of Prop. \ref{prop:reductiontocountmaass}] Let $\mathcal{G}$ denote either $\mathcal{F}^{-}_{\le L}$ or $\mathcal{F}^{-,\hol}$ according to whether $k=0$ or $k \ge 2$. By
  Proposition \ref{prop:thetamaassmainidentity}, we may write
  \begin{equation*}
    \frac{\|\varphi^{\JL}\|^2}{V_{d_B,N}}
    \left( 
      |\varphi(g_1)|^2 - |\varphi(g_2)|^2
       \right)
    =
    \langle \Theta^-(g_1,\cdot) - \Theta^-(g_2,\cdot), \varphi^{\JL} \rangle.
  \end{equation*}
   By Bessel's inequality, it follows that
  \begin{equation}
    \sum_{\varphi \in \mathcal{G}}  \frac{\|\varphi^{\JL}\|^2}{(V_{d_B,N})^2} \left(|\varphi(g_1)|^2-|\varphi(g_2)|^2\right)^2  \le \|\Theta^{-}(g_1,\cdot)-\Theta^{-}(g_2,\cdot)\|_2^2.
    \label{eq:MaassBessel}
  \end{equation}
  We now bound the right-hand side of \eqref{eq:MaassBessel} (and in particular, verify that it is finite). Since $\Theta^{-}(g,\cdot)$ is $K_{\infty}$-isotypical, Lemma \ref{lem:generalL2bound} and equation \eqref{eq:thetamaasstranform} give the bound
  \begin{equation}
    \ll \frac{1}{V_{1,d_BN}} \sum_{\ell|d_BN} \int_{\frac{\sqrt{3}}{2} \frac{\ell^2}{d_BN}}^{\infty} \int_0^{\ell} \frac{1}{\ell^2} |\theta_{g_1,\ell}^{-}(z)-\theta_{g_2,\ell}^{-}(z)|^2 dx \frac{dy}{y^2}.
    \label{eq:Maassdomainsplit}
  \end{equation}
  We insert the definition \eqref{eq:defthetamaass} into the inner integral and evaluate, giving
  \begin{equation*}
    \frac{1}{\ell^2} \int_0^{\ell} |\theta_{g_1,\ell}^{-}(z)-\theta_{g_2,\ell}^{-}(z)|^2 dx
    =   \frac{1}{\ell} y^{2+k} \sum_{n \in \frac{1}{\ell} \mathbb{Z}} \left| \sum_{i=1}^{2} (-1)^{i} \sum_{\substack{\gamma \in R(\ell;g_i)  \\ \det(\gamma)=n}} X(\gamma)^k e^{-2\pi y P(\gamma)}  \right|^2.
  \end{equation*}
  Note that the sum over $i$ kills the contribution from $\gamma = 0$, so we may omit that contribution in what follows.  We separate the two sums by Cauchy--Schwarz and bound $X(\alpha)$ by $P(\alpha)^{\frac{1}{2}}$, giving
  \begin{equation}
    \frac{1}{\ell^2} \int_0^{\ell} |\theta_{g_1,\ell}^{-}(z)-\theta_{g_2,\ell}^{-}(z)|^2 dx 
    \ll \sum_{i=1}^{2} \frac{1}{\ell} \sum_{n \in \frac{1}{\ell} \mathbb{Z}} \left|  \sum_{\substack{ 0 \neq\gamma \in R(\ell;g_i) \\ \det(\gamma)=n}} P(\gamma)^{\frac{k}{2}} e^{-2\pi y P(\gamma)}  \right|^2 y^{2+k}.
    \label{eq:Maassunipotent2}
  \end{equation}
  We now treat the contributions from $i=1,2$ individually.  We commence with the integral in the variable $y$.  Let
$$
Q(s,x)= \frac{1}{\Gamma(s)} \int_x^{\infty} t^{s} e^{-t} \frac{dt}{t} \le 1
$$
denote the normalized incomplete Gamma function.  Setting $g \coloneqq g_i$, we find
\begin{multline}
  \frac{(4 \pi)^{k+1}}{\Gamma(k+1)} \int_{Y}^{\infty} \left|  \sum_{\substack{0 \neq \gamma \in R(\ell;g) \\ \det(\gamma)=n}} P(\gamma)^{\frac{k}{2}} e^{-2\pi y P(\gamma)}  \right|^2 y^{2+k} \frac{dy}{y^2} \\
  = \sum_{\substack{0 \neq\gamma_1,\gamma_2 \in R(\ell;g) \\ \det(\gamma_1)=\det(\gamma_2)=n}} \frac{2}{P(\gamma_1)+P(\gamma_2)} \left( \frac{2\sqrt{P(\gamma_1)P(\gamma_2)}}{P(\gamma_1)+P(\gamma_2)} \right)^{k} Q\left(k+1,  2\pi Y (P(\gamma_1)+P(\gamma_2) ) \right) \\
  \le \sum_{\substack{0 \neq \gamma_1,\gamma_2 \in R(\ell;g) \\ \det(\gamma_1)=\det(\gamma_2)=n}} \frac{2}{P(\gamma_1)+P(\gamma_2)} Q\left(k+1, 2\pi Y (P(\gamma_1)+P(\gamma_2)) \right).
  \label{eq:Maassyint}
\end{multline}
Since $Q(s,x)$ is super-polynomially small in both $s$ and $x$ as soon as $x \gg s$, we see by dyadically partitioning $\max_i\{P(\gamma_i)^{\frac{1}{2}}\}$ that \eqref{eq:Maassyint} is further bounded by
$$
\ll_{A} \sum_{j} \frac{1}{T_j^2} \left(1+\frac{T_j^2Y}{k+1} \right)^{-A} \sum _{\substack{
    \gamma_1, \gamma_2 \in R(\ell;g)  \cap \Omega^{\star}(1, T_j) \\
    \det(\gamma_1) = \det(\gamma_2)=n}} 1
$$
for any $A \ge 0$, where $T_j=2^j, j \in \mathbb{Z}$. By putting all of these estimates together, we arrive at
\begin{multline*}
  \frac{1}{(V_{d_B,N})^2}\sum_{\varphi \in \mathcal{G}} \|\varphi^{\JL}\|^2 \left(|\varphi(g_1)|^2-|\varphi(g_2)|^2\right)^2 \\
  \preccurlyeq_{A} \frac{\Gamma(k+1)(4\pi)^{-k}}{V_{1,d_BN}} \sum_{i=1}^{2} \sum_{\ell|d_B N} \frac{1}{\ell} \sum_j \frac{1}{T_j^2} \left(1+ \frac{\ell^2}{d_BN} \frac{T_j^2}{k+1} \right)^{-A} \sum _{\substack{ \gamma_1, \gamma_2 \in R(\ell;g_i) \cap \Omega^{\star}(1, T_j) \\ \det(\gamma_1) = \det(\gamma_2)}} 1
\end{multline*}
for any $A \ge 0$. Let us recall from \eqref{eq:Unadel-covolume} that $V_{d_B,N}, V_{1,d_BN} = (d_BN)^{1+o(1)}$ and that for $\varphi \in \mathcal{F}^{-}_{\le L}$ we have $\|\varphi^{\JL}\| \succcurlyeq_{L} 1$ (see \eqref{eq:indefsphericalJLnorm}). In order to conclude the first part of the proposition, we note that the range of $T$ may be limited from above to $\preccurlyeq (d_BN)^{\frac{1}{2}} \ell^{-1}$ by the super-polynomial decay and \emph{any} polynomial bound on the second moment matrix count, which was noted in \S \ref{sec:proof-thm1} and is the subject of the remaining sections \S \ref{sec:prel-geom-numb} through \S \ref{sec:type-ii-estimates}. 
The second part of the proposition follows along the same lines, but we need to use the bound $\|\varphi^{\JL}\|^2 \succcurlyeq \Gamma(k) (4 \pi)^{-k}$ for $\varphi \in \mathcal{F}^{-,\hol}$ instead (see \eqref{eq:holomorphicJL-norm}).
\end{proof}

\begin{proof}[Proof of Prop. \ref{prop:reductiontocountholo}] We follow the recipe of the previous proof, only this time for the the family $\mathcal{F}^{-,\hol}$ to which the theta function $\Theta^{-,\hol}$ corresponds. As we shall see, the latter already possesses a finite $L^2$-norm. Hence, we need not consider a difference of theta functions. After the initial steps, we arrive at

  \begin{multline}
    \sum_{\varphi \in \mathcal{F}^{-,\hol}} \frac{\|\varphi^{\JL}\|^2}{(V_{d_B,N})^2} |\varphi(g)|^4 \ll \frac{1}{V_{1,d_BN}} \frac{k^2 \Gamma(k-1)}{(4 \pi )^k} \\
    \times \sum_{\ell |d_BN} \frac{1}{\ell} \sum_{n \in \frac{1}{\ell} \mathbb{N}} \frac{1}{n} \left| \sum_{\substack{\gamma \in R(\ell;g)\\ \det(\gamma)=n}} \left( \frac{\det(\gamma)^{\frac{1}{2}}}{\overline{X(\gamma)}} \right)^k \right|^2 Q(k-1, 2 \sqrt{3} \pi n \tfrac{\ell^2}{d_BN}).
    \label{eq:indefholoinitialstep}
  \end{multline}

  We further simplify using the lower bound $\|\varphi^{\JL}\|^2 \succcurlyeq \Gamma(k) (4 \pi)^{-k}$ (see \eqref{eq:holomorphicJL-norm}), the approximations $V_{d_B,N}, V_{1,d_BN} = (d_B N)^{1+o(1)}$, the super polynomial decay of normalized incomplete Gamma function, as well as the identities
$$
|X(\gamma)|^2 = \det(\gamma)+u(\gamma) \quad \text{ and } \quad 2u(\gamma)+\det(\gamma) = P(\gamma).
$$
We obtain
\begin{equation}
  \frac{1}{d_B N k}\sum_{\varphi \in \mathcal{F}^{-,\hol}} |\varphi(g)|^4 
  \preccurlyeq_A \sum_{\ell|d_BN} \frac{1}{\ell} \sum_{n \in \frac{1}{\ell}\mathbb{N}} \frac{1}{n} \left(1+ \frac{\ell^2}{d_BN} \frac{n}{k} \right)^{-A} \left|\sum_{\substack{\gamma \in R(\ell;g) \\ \det(\gamma)=n}} \left(1+ \frac{u(\gamma)}{n} \right)^{-\frac{k}{2}}   \right|^2,
  \label{eq:holosecondlast}
\end{equation}
for any $A \ge 0$.  By the triangle inequality, we reduce to estimating similar expressions, but with the sum over $\gamma$ restricted by one of the following conditions:
\begin{enumerate}[(i)]
\item  $u(\gamma) \le k^{-1+\eps} \det(\gamma)$, 
\item $k^{-1+\eps}\det(\gamma) \le u(\gamma) \le \det(\gamma)$,  or
\item $\det(\gamma) \le u(\gamma)$. 
\end{enumerate}

In case (i), we bound $(1+u(\gamma)/n)^{-\frac{k}{2}} \le 1$. Furthermore, we have $\det(\gamma) \asymp P(\gamma)$ and $u(\gamma)  \ll k^{-1+\epsilon} P(\gamma)$. Hence, after dyadically partitioning the range of $P(\gamma)^{\frac{1}{2}}$, we arrive at
$$
\preccurlyeq_A \sum_{\ell|d_BN} \frac{1}{\ell} \sum_j \frac{1}{T_j^2} \left(1+ \frac{\ell^2}{d_BN} \frac{T_j^2}{k} \right)^{-A} \sum _{\substack{\gamma_1, \gamma_2 \in R(\ell;g)  \cap \Omega^{\star}(k^{-1+\eps}, T_j) \\ \det(\gamma_1) = \det(\gamma_2)>0}} 1
$$

In case (ii), we use that $(1+u(\gamma)/n)^{-\frac{k}{2}}\le (1+k^{\eps-1})^{-\frac{k}{2}}$ has super-polynomial decay in $k$ (and hence also in $(d_BN)$ as $k \gg_{\eta} (d_BN)^{\eta}$ by assumption). As in case (i), we have $ \det(\gamma) \asymp P(\gamma)$, but this time only $u(\gamma)\le P(\gamma)$. We arrive at a contribution of
$$
\preccurlyeq_{A,\eta} (kd_BN)^{-A} \sum_{\ell|d_BN} \frac{1}{\ell} \sum_j \frac{1}{T_j^2} \left(1+ \frac{\ell^2}{d_BN} \frac{T_j^2}{k} \right)^{-A} \sum _{\substack{\gamma_1, \gamma_2 \in R(\ell;g) \cap \Omega^{\star}(1, T_j) \\ \det(\gamma_1) = \det(\gamma_2)>0}} 1.
$$

In case (iii), we bound
$$
\left(1 + \frac{u(\gamma)}{n} \right)^{-\frac{k}{2}} \le 2^{-\frac{k}{4}} \left(1 + \frac{u(\gamma)}{n} \right)^{-\frac{k}{4}}.
$$
The factor $2^{-\frac{k}{4}}$ we use for super-polynomial decay in $(kd_BN)$ as before. The other factor we use as follows
$$
\frac{1}{n} \left(1+\frac{u(\gamma)}{n}\right)^{-1} \left(1+\frac{\ell^2}{d_BN}\frac{n}{k}\right)^{-\frac{k}{4}+1} \left(1+\frac{u(\gamma)}{n} \right)^{-\frac{k}{4}+1} \le \frac{1}{u(\gamma)} \left(1+\frac{\ell^2}{d_BN}\frac{u(\gamma)}{k}\right)^{-\frac{k}{4}+1}.
$$
Hence, \eqref{eq:holosecondlast} is bounded by
\begin{multline*}\preccurlyeq_{A,\eta} (kd_BN)^{-A} \sum_{\ell|d_BN} \frac{1}{\ell} \sum_{\substack{\gamma_1, \gamma_2 \in R(\ell;g)\\ \det(\gamma_1)=\det(\gamma_2)}} \frac{1}{u(\gamma_1)+u(\gamma_2)} \left(1+ \frac{\ell^2}{d_BN} \frac{u(\gamma_1)+u(\gamma_2)}{k} \right)^{-\frac{k}{4}+1}	\\
\preccurlyeq_{A,\eta} (kd_BN)^{-A} \sum_{\ell|d_BN} \frac{1}{\ell} \sum_j \frac{1}{T_j^2} \left(1+\frac{\ell^2}{d_BN} \frac{T_j^2}{k} \right)^{-\frac{k}{4}+1} \sum _{\substack{\gamma_1, \gamma_2 \in R(\ell;g)  \cap \Omega^{\star}(1, T_j) \\ \det(\gamma_1) = \det(\gamma_2)>0}} 1,
\end{multline*}
where we have dyadically partitioned $\max_i\{P(\gamma_i)^{\frac{1}{2}}\} \asymp \sqrt{u(\gamma_1)+u(\gamma_2)}$. The proof of the proposition is now concluded as in the previous case.
\end{proof}

\begin{proof}[Proof of Prop. \ref{prop:reductiontocountspherical}]
	
  We treat the definite spherical case in the same spirit as the indefinite spherical case.  We readily arrive at the estimate
  \begin{multline}
    \sum_{\varphi \in \mathcal{F}^{+}_{m}} \frac{\|\varphi^{\JL}\|^2}{(V_{d_B,N})^2} \left( |\varphi(g_1)|^2 - |\varphi(g_2)|^2 \right)^2 \ll \frac{1}{V_{1,d_BN}} \frac{(2m+1)^2 \Gamma(2m+1)}{(4 \pi )^{2m+2}} \\
    \times \sum_{i=1}^2 \sum_{\ell |d_BN} \frac{1}{\ell} \sum_{n \in \frac{1}{\ell} \mathbb{N}} \frac{1}{n} \left| \sum_{\substack{\gamma \in R(\ell;g_i)\\ \det(\gamma)=n}} P_m\left( \frac{|X(\gamma)|^2-u(\gamma)}{n} \right) \right|^2 Q(2m+1, 2 \sqrt{3} \pi n \tfrac{\ell^2}{d_BN}).
    \label{eq:defsphericalinitialsteps}
  \end{multline}

  We simplify the inequality by using the lower bound $\|\varphi^{\JL}\|^2 \succcurlyeq \Gamma(2m+2) (4 \pi)^{-2m-2}$ (see \eqref{eq:holomorphicJL-norm}), the approximations $V_{d_B,N}, V_{1,d_BN}=(d_BN)^{1+o(1)}$, and the super-polynomial decay of the normalized incomplete Gamma function $Q$.  We obtain
  \begin{multline}
    \frac{1}{d_BN (m+1)} \sum_{\varphi \in \mathcal{F}^{+}_{m}} \left( |\varphi(g_1)|^2 - |\varphi(g_2)|^2 \right)^2 \\
    \preccurlyeq_A \sum_{i=1}^2 \sum_{\ell |d_BN} \frac{1}{\ell} \sum_{n \in \frac{1}{\ell} \mathbb{N}} \frac{1}{n} \left(1+\frac{\ell^2}{d_BN} \frac{n}{m+1} \right)^{-A} \left| \sum_{\substack{\gamma \in R(\ell;g_i)\\ \det(\gamma)=n}} P_m\left( \frac{|X(\gamma)|^2-u(\gamma)}{n} \right) \right|^2,
    \label{eq:defsphericalinitialstepsclean}
  \end{multline}
  for any $A \ge 0$. We proceed further by appealing to the Bernstein inequality \cite{SergeBernstein1931} for the Legendre polynomials:
  \begin{equation} P_m(t) \le \min\left\{1 , \sqrt{\frac{2}{\pi m}} \frac{1}{(1-t^2)^{\frac{1}{4}}} \right\}, \quad \text{for } |t| \le 1.
    \label{eq:Bernstein}
  \end{equation}
 We recall that $\det(\gamma)= |X(\gamma)|^2 + u(\gamma)$, so that, with $\gamma$ and $n$ as in the above sum,
  \begin{equation*}
    t := \frac{|X(\gamma)|^2 - u(\gamma)}{n}
    =
    \frac{|X(\gamma)|^2 - u(\gamma)}{|X(\gamma)|^2 + u(\gamma) },
    \quad
    1 - t^2 =
    \frac{4 |X(\gamma)|^2 \cdot u (\gamma)  }{n^2} \ge 0.
  \end{equation*}
 Dyadically partitioning $P(\gamma)^{\frac{1}{2}}=\det(\gamma)^{\frac{1}{2}}$, we conclude that \eqref{eq:defsphericalinitialstepsclean} is bounded by
  \begin{multline} \preccurlyeq_{A} \sum_{i=1}^2 \sum_{\ell|d_BN} \frac{1}{\ell} \sum_{j} \frac{1}{T_j^2} \left(1+\frac{\ell^2}{d_BN} \frac{T_j^2}{m+1} \right)^{-A} \\
   \times \sum_{T_j^2 \le n < 4T_j^2} \left( \sum_{\substack{\gamma \in R(\ell;g_i) \\  \gamma \in \Omega^{\star}(1,2T_j)-\Omega^{\star}(1,T_j) \\ \det(\gamma)=n}}  \min\left\{1, \frac{1}{(m+1)^{\frac{1}{2}}} \frac{T_j}{(|X(\gamma)|^2\cdot u(\gamma))^{\frac{1}{4}}} \right\} \right)^2,
    \label{eq:defprelimestimate}
  \end{multline}
  where $T_j = 2^j$ as before. 
The minimum in \eqref{eq:defprelimestimate} lies between $\asymp (m+1)^{-\frac{1}{2}}$ and $1$. Let us consider the $\gamma \in \Omega^{\star}(1,2T_j)-\Omega^{\star}(1,T_j)$ for which
  \begin{equation}
    \min\left\{1, \frac{1}{(m+1)^{\frac{1}{2}}} \frac{T_j}{(|X(\gamma)|^2\cdot u(\gamma))^{\frac{1}{4}}}  \right\} \asymp \frac{1}{(m+1)^{\frac{1}{2}}} \frac{1}{\delta^{\frac{1}{4}}},
    \label{eq:Legendresplit1}
  \end{equation}
   for some given $\delta$ with $1/(m+1)^2 \ll \delta \le 1$. In particular, $|X(\gamma)|^2 \cdot u(\gamma) \ll \delta T_j^4$. Since $|X(\gamma)|^2+u(\gamma)=P(\gamma) \asymp T_j^2$, both cannot be simultaneously small. Hence,
  \begin{equation*}
    \min \{|X(\gamma)|^2, u(\gamma)\}
    =
    \frac{|X(\gamma)|^2 \cdot u(\gamma)}{
      \max\{|X(\gamma)|^2, u(\gamma)\}}
    \ll  \delta  T_j^2.
  \end{equation*}
  Thus, after replacing $\delta$ with its multiple by a scalar of the form $\asymp 1$ if needed (which has no affect on \eqref{eq:Legendresplit1}), we may assume that
  $\min \{|X(\gamma)|^2, u(\gamma)\} \leq \delta (2T_j)^2$, i.e., that $\gamma$ lies in either $\Omega^{\star}(\delta,2T_j)$ or $\Psi^{\star}(\delta,2T_j)$. We now consider dyadic scales $\delta_a$ of $\delta$'s between $\asymp 1/(m^2+1)$ and $1$. The just mentioned arguments then allow us to bound second line in \eqref{eq:defprelimestimate} by
	\begin{equation*} 
   \sum_{T_j^2 \le n < 4T_j^2} \Biggl( \sum_{a} \sum_{\substack{\gamma \in R(\ell;g_i) \\  \gamma \in \Omega^{\star}(\delta_a,2T_j)\cup \Psi^{\star}(\delta_a,2T_j) \\ \det(\gamma)=n}}  (m+1)^{-\frac{1}{2}}\delta_a^{-\frac{1}{4}} \Biggr)^2.
    \label{eq:defprelimestimate-dyadic}
  \end{equation*}
There are at most $ \preccurlyeq 1$ dyadic scales in the range $1/(m^2+1) \ll \delta \le 1$. Thus, after applying Cauchy--Schwarz in order to pull out the sum over $\delta_a$, we bound \eqref{eq:defprelimestimate} by
  \begin{equation}
    \preccurlyeq_{A} \sum_{i=1}^2 \sum_{\ell|d_BN} \frac{1}{\ell} \sum_{j} \frac{1}{T_j^2} \left(1+\frac{\ell^2}{d_BN} \frac{T_j^2}{m+1} \right)^{-A}
        \sum_a \sum_{\substack{\gamma_1,\gamma_2 \in R(\ell;g_i) \\ \gamma_1,\gamma_2 \in \Omega^{\star}(\delta_a,2T_j) \cup \Psi^{\star}(\delta_a,2T_j))\\ \det(\gamma_1)=\det(\gamma_2)}} (m+1)^{-1} \delta_a^{-\frac{1}{2}}.
    \label{eq:defestimate}
  \end{equation}
  The proof is once more concluded as it was for the first proposition.

\end{proof}
\begin{proof}[Proof of Prop. \ref{prop:reductiontocountdefholo}]
  One final time we iterate the initial steps for the holomorphic family $\mathcal{F}^{\hol}$ in the definite case. We are again in the situation where the theta kernel $\Theta^{+,\hol}$ has finite $L^2$-norm, so we obtain, without having to take differences,
  \begin{multline}
    \sum_{\varphi \in \mathcal{F}^{+,\hol}} \frac{\|\varphi^{\JL}\|^2}{(V_{d_B,N})^2} |\varphi(g)|^4 \ll \frac{1}{V_{1,d_BN}} \frac{(k+1)^2 \Gamma(k+1)}{(4 \pi )^{k+2}} \\
    \times \sum_{\ell |d_BN} \frac{1}{\ell} \sum_{n \in \frac{1}{\ell} \mathbb{N}} \frac{1}{n} \left| \sum_{\substack{\gamma \in R(\ell;g)\\ \det(\gamma)=n}} \left( \frac{X(\gamma)^2}{n} \right)^{\frac{k}{2}} \right|^2 Q(k+1, 2 \sqrt{3} \pi n \tfrac{\ell^2}{d_BN}).
    \label{eq:defholoinitialstep}
  \end{multline}
  We simplify this estimate using the lower bound $\|\varphi^{\JL}\|^2 \succcurlyeq \Gamma(k+2) (4 \pi)^{-k-2}$ (see \eqref{eq:holomorphicJL-norm}), the approximations $V_{d_B,N}, V_{1,d_BN}=(d_BN)^{1+o(1)}$, the super-polynomial decay of the normalized incomplete Gamma function $Q$, and the identity $|X(\gamma)|^2=\det(\gamma)-u(\gamma)$. We thereby obtain
\begin{equation}
  \frac{1}{d_BN k} \sum_{\varphi \in \mathcal{F}^{+,\hol}} |\varphi(g)|^4 
  \preccurlyeq_A  \sum_{\ell |d_BN} \frac{1}{\ell}  \sum_{n \in \frac{1}{\ell} \mathbb{N}} \frac{1}{n} \left(1+\frac{\ell^2}{d_BN} \frac{n}{k} \right)^{-A} \left| \sum_{\substack{\gamma \in R(\ell;g)\\ \det(\gamma)=n}} \left( 1- \frac{u(\gamma)}{n} \right)^{\frac{k}{2}} \right|^2,
  \label{eq:defholosecondlast}
\end{equation}
for any $A \ge 0$.
We dyadically partition $P(\gamma)^{\frac{1}{2}}=\det(\gamma)^{\frac{1}{2}}$ and distinguish the two cases:
\begin{enumerate}[(i)]
\item $u(\gamma) \le k^{-1+\eps} \det(\gamma)$, and
\item  $k^{-1+\eps}\det(\gamma) \le u(\gamma)$. 
\end{enumerate}
We separate them by the triangle inequality. Using the inequality $(1-u(\gamma)/n)\le 1$, we see that the contribution of the first case is bounded by
$$
\preccurlyeq_A \sum_{\ell|d_BN} \frac{1}{\ell} \sum_j \frac{1}{T_j^2} \left(1+ \frac{\ell^2}{d_BN} \frac{T_j^2}{k} \right)^{-A} \sum _{\substack{\gamma_1, \gamma_2 \in R(\ell;g)  \cap \Omega^{\star}(k^{-1+\eps}, T_j) \\ \det(\gamma_1) = \det(\gamma_2)}} 1,
$$
where $T_j=2^j, j \in \mathbb{Z}$. In the second case, we see that $(1-k^{\eps-1})^{\frac{k}{2}}$ enjoys super-polynomial decay in $k$.  The contribution of the second case is thus bounded by
$$
\preccurlyeq_{A} k^{-A} \sum_{\ell|d_BN} \frac{1}{\ell} \sum_j \frac{1}{T_j^2} \left(1+ \frac{\ell^2}{d_BN} \frac{T_j^2}{k} \right)^{-A} \sum _{\substack{\gamma_1, \gamma_2 \in R(\ell; g )\cap \Omega^{\star}(1, T_j) \\ \det(\gamma_1) = \det(\gamma_2)}} 1.
$$
The proof is now concluded as it was for the previous propositions.
\end{proof}

\section{Preliminaries on the geometry of numbers}\label{sec:prel-geom-numb}


\subsection{Bounds on successive minima}

\begin{definition}\label{defn:succ-min}
  Let $V$ be an $n$-dimensional real vector space.  Let $L \subseteq V$ be a lattice (i.e., a cocompact discrete subgroup).  Given a compact convex $0$-symmetric subset $\mathcal{K}$ of $V$ with nonempty interior, we define a function $N : V \rightarrow \mathbb{R}_{\geq 0}$ by $N(v) := \inf \{t > 0 : v \in t \mathcal{K} \}$.  Given a positive-definite quadratic form $Q$ on $V$, we define such a function by $N(v) := Q(v)^{1/2}$, or equivalently, by applying the previous definition with $\mathcal{K}$ the unit ball for $Q$.  In either case, we define the \emph{successive minima} $\lambda_1 \leq \dotsb \leq \lambda_n$ of $\mathcal{K}$ on $L$ (or of $Q$ on $L$) as: $\lambda_k$ is the smallest positive real for which there is a linearly independent subset $\{v_1,\dotsc,v_k\}$ of $L$ for which $N(v_j) \leq \lambda_k$ for each $1 \leq j \leq k$.
\end{definition}

\begin{lemma} Let $z\in \mathbb{H}$ with maximal imaginary part under the orbit of the Atkin--Lehner operators $A_0(N)$ of $\Gamma_0(N)$ with $N$ squarefree. Then, we have
$$
\Im(z) \ge \frac{\sqrt{3}}{2N} \quad \text{ and } \quad |cz+d|^2 \ge \frac{(c,N)}{N}
$$
for any $(c,d) \in \mathbb{Z}^2$ distinct from $(0,0)$.
\label{lem:spacing}
\end{lemma}
\begin{proof} This is essentially \cite[Lemma 1]{HT2}.  That reference gives the slightly weaker bound obtained by omitting the factor $(c,N)$, but the stronger bound that we have stated follows from their proof, keeping track of $(c,N)$ at each step rather than bounding it from below by $1$.
\end{proof}

\subsection{Lattice counting}

\begin{lemma} Let $f_{\mathcal{K}}$ be the distance function of a closed convex $0$-symmetric set $\mathcal{K}\subseteq \mathbb{R}^n$ of positive volume. Let $\Lambda\subset \mathbb{R}^n$ be a lattice, and let $ \lambda_1 \le \lambda_2 \le \dots \le \lambda_n$ denote the successive minima (see Definition \ref{defn:succ-min}) of $\mathcal{K}$ on $\Lambda$. Then, there is a basis $v_1,\dots,v_n$ of $\Lambda$ such that $f_{\mathcal{K}}(v_i) \asymp_n \lambda_i$.
  \label{lem:smalllatticegeneraotrs}
\end{lemma}
\begin{proof} This is \cite[Thm 2, p. 66]{GrubLekkGeometryofNumbers}.
\end{proof}

\begin{lemma} Let $\mathcal{K}\subseteq \mathbb{R}^n$ be a closed convex $0$-symmetric set of positive volume. Let $\Lambda\subset \mathbb{R}^n$ be a lattice and let $ \lambda_1 \le \lambda_2 \le \dots \le \lambda_n$ denote the successive minima of $\mathcal{K}$ on $\Lambda$. Then
$$
|\mathcal{K} \cap \Lambda| \asymp_n \prod_{i=1}^n \left( 1+\frac{1}{\lambda_i} \right).
$$
\label{lem:latticepointcount}
\end{lemma}
\begin{proof} The lower bound follows from van der Corput's generalization of Minkowski's first theorem \cite{vanderCorput1936}. It states that for $\mathcal{K}' \subset \mathbb{R}^d$ a closed convex $0$-symmetric set and $\Lambda \subset \mathbb{R}^d$ a lattice, one has
  \begin{equation}\label{eq:mathc-lambd-geq-1}
|\mathcal{K}'\cap \Lambda'|+1
\geq
|\mathrm{int}\,\mathcal{K}'\cap \Lambda'|+1 \ge 2^{1-d} \frac{\vol(\mathcal{K}')}{\vol(\mathbb{R}^d / \Lambda')}.
\end{equation}
Let $d$ be the largest integer such that $\lambda_d \le 1$.  Let $v_i\in \Lambda$, for $i=1,\dots,d$, be a set of linearly independent vectors such that $\lambda_i^{-1}v_i \in \mathcal{K}$. Let $\mathcal{K}'$ be the convex hull of the vectors $\pm\lambda_i^{-1}v_i$ and $\Lambda'$ the span of the vectors $v_i$.  In particular, $\mathcal{K} '$ is nonempty, hence $0 \in \mathcal{K}' \cap \Lambda'$, and so
\begin{equation*}
  2 |\mathcal{K}' \cap \Lambda'| \geq |\mathcal{K}' \cap \Lambda'| + 1.
\end{equation*}
Using \eqref{eq:mathc-lambd-geq-1}, it follows now that
$$
|\mathcal{K}\cap \Lambda| \ge |\mathcal{K}'\cap \Lambda'| \ge 2^{-d} \frac{\vol(\mathcal{K}')}{\vol(\mathbb{R}^d / \Lambda')} = \frac{1}{d!} \prod_{i=1}^d \frac{1}{\lambda_i}.
$$
For the upper bound, we refer to \cite[Prop. 2.1]{Succminupper}.
\end{proof}

\begin{lemma} Let $\Lambda \subset \mathbb{R}^2$ be a lattice of rank $2$ and $B\subseteq \mathbb{R}^2$ a ball of radius $R$ (not necessarily centred at $0$). If $\lambda_1 \le \lambda_2$ are the successive minima of $\Lambda$, then
$$
|B \cap \Lambda| \ll 1+\frac{R}{\lambda_1}+\frac{R^2}{\lambda_1\lambda_2}.
$$
\label{lem:latticeball}
\end{lemma}
\begin{proof} See \cite[Lemma 2.1]{MR3038127}.
\end{proof}

\section{Local preliminaries on orders}\label{sec:order-theor-prel}

\subsection{Quadratic preliminaries}

Let $F$ be a non-archimedean local field of characteristic $\neq 2$.  Let $E/F$ be a separable quadratic extension, thus $F$ is either the split quadratic extension $F \oplus F$ or a quadratic field extension.  We write $\mathfrak{o}$ (resp. $\mathfrak{o}_E$) for the ring of integers in $F$ (resp. $E$), $x \mapsto \bar{x}$ for the canonical involution on $E$, and
\begin{equation*}
  \nr(x) = x \bar{x}, \quad \tr(x) = x + \bar{x}
\end{equation*}
for the norm and trace.  Recall that the different ideal $\mathfrak{d}$ for this extension is the smallest $\mathfrak{o}_E$-ideal for which $\tr(\mathfrak{d}^{-1}) \subseteq \mathfrak{o}$, and in fact $\mathfrak{d}^{-1} = \{ x \in E: \tr(x \mathfrak{o}_E) \subseteq \mathfrak{o} \}$.  If $E/F$ is split or unramified, then $\mathfrak{d} = \mathfrak{o}_E$.

We may regard $E$ as a two-dimensional vector space over $F$.

Let $q : E \rightarrow F$ be a nondegenerate binary quadratic form with the property that for all $e,x \in E$, we have $q(e x) =\nr(e) q(x)$.  In other words, $q$ is an $F$-multiple of $\nr$, specifically $q = q(1)\nr$.

For $x,y \in E$ we set $\langle x, y \rangle \coloneqq q(x+y) - q(x) - q(y) = q(1) \tr(x \bar{y})$ , so that $q(x) = \langle x, x \rangle/2$.

Let $\mathfrak{a} \subset E$ be a fractional $\mathfrak{o}_E$-ideal.  Write $\mathfrak{a}^\vee$ for the dual of $\mathfrak{a}$ with respect to the quadratic form $q$, i.e., $\mathfrak{a}^\vee \coloneqq \{ x \in E : \langle x, \mathfrak{a} \rangle \subseteq \mathfrak{o} \}$.

Let $\mathfrak{n}$ denote the fractional $\mathfrak{o}$-ideal generated by $q(\mathfrak{a})$.

\begin{lemma}
  We have $\mathfrak{a} = \mathfrak{d} \mathfrak{n} \mathfrak{a}^\vee$.
\end{lemma}

\begin{proof}
  Let $\alpha$ be a generator of $\mathfrak{a}$.  Then $\mathfrak{n} = q(1)\nr(\mathfrak{a}) = \mathfrak{o} q(1) \alpha \bar{\alpha}$, $\mathfrak{a}^\vee = \{ q(1)^{-1} x: x \in E, \tr(x \bar{\mathfrak{a}}) \subseteq \mathfrak{o} \} = q(1)^{-1} \bar{\alpha}^{-1} \mathfrak{d}^{-1}$.  Multiplying through, the conclusion follows.
\end{proof}

\begin{corollary}
  Suppose that $E/F$ is unramified and that $q$ is integral on $\mathfrak{a}$, so that $\mathfrak{a} \subseteq \mathfrak{a}^\vee $.  Then the elementary divisors for the $\mathfrak{o}$-module inclusion $\mathfrak{a} \hookrightarrow \mathfrak{a}^\vee$ are $(\mathfrak{n}, \mathfrak{n})$.
\end{corollary}

\begin{proof}
  Our hypotheses imply that $\mathfrak{d} = \mathfrak{o}$ and that $\mathfrak{n}$ is an integral ideal.  The lemma implies that there is an isomorphism (first of $\mathfrak{o}_E$-modules, then of $\mathfrak{o}$-modules) $\mathfrak{a}^\vee / \mathfrak{a} \cong \mathfrak{o}_E / \mathfrak{n} \mathfrak{o}_E \cong (\mathfrak{o}/\mathfrak{n})^2$, whence the conclusion.
\end{proof}

\begin{remark}
  Under the hypotheses of the corollary, the discriminant ideal of the binary quadratic form $(q,\mathfrak{a})$ is $\mathfrak{n}^2$.  More generally, under the hypotheses of the lemma, the discriminant ideal is $\mathfrak{D} \mathfrak{n}^2$, with $\mathfrak{D} =\nr (\mathfrak{d})$.  Conversely, given the discriminant ideal, we may compute $\mathfrak{n}$ as its square-root.
\end{remark}

\subsection{Quaternionic preliminaries: general case}

Let $F$ be a non-archimedean local field, let $B$ be a quaternion $F$-algebra, and let $E$ be a separable quadratic $F$-subalgebra of $B$.  We equip $B$ with the quadratic form $q : B \rightarrow F$ given by the reduced norm, whose bilinearization $\langle\, , \rangle$ as above is described by the reduced trace and the main involution on $B$ via the formula $\langle x, y \rangle = \tr(x \bar{y})$.  We have a canonical decomposition $B = E \oplus E^\perp$, where $E^\perp = \{ x \in B : \langle x, y \rangle = 0 \text{ for all } y \in E \}$.

Let $\mathfrak{o}$ and $\mathfrak{o}_E$ denote the respective maximal orders of $F$ and $E$.  We write $\mathfrak{d}$ for the different ideal, as before.

Let us say that an order $R$ in $B$ is \emph{$E$-adapted} if it is of the form $R = \mathfrak{o}_E \oplus \mathfrak{a}$ for some $\mathfrak{o}_E$-submodule $\mathfrak{a}$ of $E^\perp$ (for the action by either left or right multiplication -- it doesn't matter which, because they are conjugates of each other).

Consider such an order $R$.  Its traceless submodule is given by
\[
  R^0 = \mathfrak{o}_E^0 \oplus \mathfrak{a}.
\]
We aim to compute the dual lattice $(R^0)^\vee$ with respect to $q$.  To that end, it suffices to dualize each summand in the above decomposition, because $(R^0)^\vee = (\mathfrak{o}_E^0)^\vee \oplus \mathfrak{a}^\vee$.  We generally have $(\mathfrak{d}^{-1})^0 \subseteq (\mathfrak{o}_E^0)^\vee \subseteq \frac{1}{2} (\mathfrak{d}^{-1})^0$.  If $E$ is unramified or split, then $(\mathfrak{o}_E^0)^\vee = \frac{1}{2} \mathfrak{o}_E^0 = \frac{1}{2} (\mathfrak{d}^{-1})^0$.  On the other hand, we can compute $\mathfrak{a}^\vee$ using the results of the previous section.  Indeed, the choice of any invertible element $j \in E^\perp$ defines an isomorphism $E \rightarrow E^\perp$, $x \mapsto x j$.  Transporting $q$ and $\mathfrak{a}$ via the inverse of this isomorphism gives us a fractional ideal in $E$ and a quadratic form on $E$ that satisfy the hypotheses of that section.  We obtain
\begin{equation*}
  \mathfrak{a}^\vee = \mathfrak{d}^{-1} \mathfrak{n}^{-1}
  \mathfrak{a},
\end{equation*}
where $\mathfrak{n}$ is the integral $\mathfrak{o}$-ideal characterized by either of the following properties:
\begin{itemize}
\item $\mathfrak{n}$ is generated by $q(\mathfrak{a})$.
\item $\mathfrak{D} \mathfrak{n}^2$ is the discriminant ideal of $(q,\mathfrak{a})$.
\end{itemize}

Let $\mathcal{D}$ denote the discriminant ideal of $R$.  The discriminant ideal of the summand $(q,\mathfrak{o}_E)$ is $\mathfrak{D}$.  Since the discriminant ideal is multiplicative with respect to direct sums, we obtain
\begin{equation}\label{eqn:discriminant-factor}
  \mathcal{D} = \mathfrak{D}^2 \mathfrak{n}^2.
\end{equation}
We may regard this last identity as a formula for $\mathfrak{n}$.

\subsection{Quaternionic preliminaries: unramified case}

Let us restrict henceforth to the case that \emph{$E/F$ is unramified}.  (The ramified case would be relevant for studying, e.g., the ``minimal vectors'' considered in \cite{HNSminimal,2018arXiv181011564Hs, saha2019sup}.)

The above formula then simplifies to
\[\mathfrak{n}^2 = \mathcal{D},\]
and we obtain
\[\mathfrak{a}^\vee / \mathfrak{a} \cong
  (\mathfrak{o}/\mathfrak{n})^2.\]

There are three possibilities for $F \subset E \subset B$, up to isomorphism:
\begin{enumerate}[(i)]
\item $B$ is split and $E \cong F \oplus F$.  We may then find an isomorphism $B \cong \Mat_{2 \times 2}(F)$ under which $E$ identifies with the diagonal subalgebra.
\item $B$ is split and $E/F$ is the unique unramified quadratic field extension.
\item $B$ is non-split and $E/F$ is the unique unramified quadratic field extension.
\end{enumerate}

In case (i), the $E$-adapted orders $R$ are just the Eichler orders.  An Eichler order of level $\mathfrak{q}$ has discriminant $\mathfrak{q}^2$ (cf. \cite[\S 23.4]{voightQA}), hence the ideal $\mathfrak{n}$ defined above is $\mathfrak{q}$.

The case (ii) corresponds to another type of ``minimal vectors'', which we do not consider in this paper.

In case (iii), the maximal order $R$ is $E$-adapted and has discriminant ideal $\mathfrak{p}^2$ (cf. \cite[\S 15.2.11, \S 23.4]{voightQA}), hence $\mathfrak{n} = \mathfrak{p}$.

\subsection{Bounds for commutators of elements of $R^0$}

\begin{lemma} Let $E \slash F$ unramified and $R$ be an $E$-adapted order. Then, we have
  \begin{enumerate}[(i)]
  \item \label{itm:thetasup-paper-merge-20220418:1} $q([\gamma_1,\gamma_2]) \in \mathfrak{n}$ for all $\gamma_1,\gamma_2 \in R$,
  \item \label{itm:thetasup-paper-merge-20220418:2} $[\gamma_1,\gamma_2] \in \mathfrak{n}^{-1}R$ for all $\gamma_1,\gamma_2 \in R^{\vee}$,
  \item \label{itm:thetasup-paper-merge-20220418:3} $q([\gamma_1,\gamma_2]) \in \mathfrak{n}^{-2}$ for all $\gamma_1,\gamma_2 \in R^{\vee}$.
  \end{enumerate}
  \label{lem:commutatorcong}
\end{lemma}
\begin{proof} By the previous discussion, we may write $\gamma_i=\alpha_i+\beta_i$ with $\alpha_i \in \mathfrak{o}_E$ and $\beta_i \in \mathfrak{a}$ if $\gamma_i \in R$, respectively $\mathfrak{n}^{-1}\mathfrak{a}$ if $\gamma_i \in R^{\vee}$. We have $[\alpha_1,\alpha_2]=0$ and $[\alpha_1,\beta_2] = 2\alpha_1\beta_2 \in \mathfrak{a}$, respectively $\in \mathfrak{n}^{-1}\mathfrak{a}$. Lastly, we have $\mathfrak{a}^2=\mathfrak{n} \mathfrak{a}^{\vee} \mathfrak{a} = \mathfrak{n}$ and $(\mathfrak{a}^{\vee})^2=\mathfrak{n}^{-1} \mathfrak{a} \mathfrak{a}^{\vee} = \mathfrak{n}^{-1}$. Hence,
$$
[\beta_1,\beta_2] = \beta_1\beta_2 - \beta_2\beta_1 \in \mathfrak{n}, \text{ respectively } \mathfrak{n}^{-1}.
$$
By the bilinearity of the commutator, we obtain
\eqref{itm:thetasup-paper-merge-20220418:2} and subsequently \eqref{itm:thetasup-paper-merge-20220418:3}. Similarly, we have $q([\beta_1,\beta_2]) \in \mathfrak{n}^2$ for $\beta_i \in \mathfrak{a}$. Thus, by orthogonality, we have
\begin{equation*}
  q([\gamma_1,\gamma_2]) = q(2\alpha_1\beta_2-2\alpha_2\beta_1+[\beta_1,\beta_2]) = q(2\alpha_1\beta_2-2\alpha_2\beta_1)+q([\beta_1,\beta_2]) \in \mathfrak{n},
\end{equation*}
for $\gamma_i \in R$, which gives \eqref{itm:thetasup-paper-merge-20220418:3}.
\end{proof}

\section{Invariants of rational
  quadratic forms}\label{sec:estim-succ-minima}
Let $V$ be an $n$-dimensional $\mathbb{Q}$-vector space and $q : V \rightarrow \mathbb{Q}$ a nondegenerate quadratic form. We normalize the polarization $\langle \,, \rangle$ of $q$ such that $\langle x,x \rangle = 2 q(x)$.

Given a lattice $L \subseteq V$ and a positive-definite quadratic form $Q$ on $V \otimes_{\mathbb{Q}} \mathbb{R}$, one can define the successive minima of the pair $(L,Q)$.  Our aim in this section is to provide certain estimates for those successive minima in terms of other invariants of $(L,Q)$.  Our results may be understood as a generalization of those of Blomer--Michel \cite[\S3, \S4]{MR3103131}, who treated the special case that $q$ is definite and $Q = q$.  We mention also the work of Saha \cite[\S2]{saha2019sup}.

\subsection{Non-archimedean invariants}\label{sec:non-arch-invar}
Let $L \subseteq V$ be a lattice, i.e., a $\mathbb{Z}$-submodule whose rank is the dimension of $V$.  Define
\begin{itemize}
\item the \emph{content} $C$ of $L$ to be the greatest common divisor of $q(L)$,
\item the \emph{level} $N$ of $L$ to be the reciprocal of the content of the dual lattice $L^\vee$, and
\item the (unsigned) \emph{discriminant} $\Delta$ of $L$ to be the absolute value of the determinant of the Gram matrix of $q$ on $L$.  (The discriminant of a quadratic form is traditionally defined without taking absolute values, but the sign will not matter for us.)
\end{itemize}

\begin{remark} In general, the content of $L$ does not agree with the content of the Gram matrix of $q$ on $L$, i.e., the greatest common divisor of the entries, cf. \cite[\S 2.9]{Bosch-Alg}.  However, if $q$ is integral on $L$, then the level of $L$ agrees with the level of the Gram matrix of $q$.
\end{remark}

\begin{remark} For our purposes, the content, respectively level, may be replaced by the first, respectively last, elementary divisor of the Gram matrix of $q$ as these quantities differ by a bounded power of $2$.
\end{remark}

To get acquainted with these quantities, consider for instance the case that $L$ admits a basis $e_1, \dotsc, e_n$ with respect to which $q$ is given by the diagonal quadratic form $q(\sum x_i e_i) = \frac{1}{2} \sum a_i x_i^2$ for some non-zero rational numbers $a_1,\dotsc,a_n$.  Then
\[C = \tfrac{1}{2} \gcd(a_1,\dotsc,a_n),\]
\[N = 2/\gcd(1/a_1,\dotsc,1/a_n) = 2 \cdot \lcm(a_1,\dotsc,a_n),\]
\[\Delta = |a_1 \dotsb a_n|.\]

We may relate the invariants attached to homothetic lattices: the effect of the substitution $L \mapsto m L$ for a non-zero rational scalar $m$ is
\[C \mapsto m^2 C, \quad \quad \quad N \mapsto m^2 N, \quad \quad \quad \Delta \mapsto m^{2 n} N.\]

\subsection{Archimedean invariants}\label{sec:arch-invar}

Next, write $V_{\mathbb{R}} := V \otimes_{\mathbb{Q}} \mathbb{R}$ and let $Q : V_{ \mathbb{R} } \rightarrow \mathbb{R}$ be a positive-definite quadratic form.  We may find a basis $e_1,\dotsc,e_n$ of $V _{ \mathbb{R} }$ so that, writing $x = \sum x_i e_i$, we have $Q(x) = \frac{1}{2}\sum x_i^2$ and $q(x) = \frac{1}{2} \sum a_i x_i^2$ for some non-zero real numbers $a_1,\dotsc,a_n$.  The \emph{dual} $Q^\vee$ of $Q$ with respect to $q$ may be defined by $Q^{\vee}(x) = \frac{1}{2} \sum a_i^2 x_i^2$.  We note that in the scaled coordinates $y = \sum a_i^{-1} y_i e_i$, we have $Q^{\vee}(y) = \frac{1}{2} \sum y_i^2$ and $q(y) = \frac{1}{2} \sum a_i^{-1} y_i^2$.  The \emph{Gram matrix relative to $q$} of $Q$ is defined to be the diagonal matrix with entries $(a_1,\dotsc,a_n)$.

Define
\begin{itemize}
\item the \emph{content} $C$ of $Q$ to be the infimum of the ratio $Q/|q|$ over the set where $q \neq 0$,
\item the \emph{level} $N$ of $Q$ to be the reciprocal of the content of the dual $Q^{\vee}$ of $Q$, and
\item the \emph{discriminant} $\Delta$ of $Q$ to be the absolute value of the reciprocal of the determinant of the Gram matrix relative to $q$ of $Q$.
\end{itemize}

The invariants of $Q$ may be described in terms of coordinates as above by
\[C = 1/\max(|a_1|,\dotsc,|a_n|) = \min(1/|a_1|,\dotsc,1/|a_n|),\]
\[N = \max(1/|a_1|,\dotsc,1/|a_n|) = 1/\min(|a_1|,\dotsc,|a_n|).\]
\[\Delta = 1/ |a_1 \dotsb a_n|.\]

These again behave predictably under homotheties: if we substitute $Q \mapsto Q_m \coloneqq [x \mapsto Q(m^{-1} x)$] for some non-zero real scalar $m$ (which has the effect of multiplying the $Q$-unit ball by $m$), then the coefficients transform like $a_i \mapsto m^2 a_i$ and hence the invariants like
\[C \mapsto m^{-2} C, \quad \quad \quad N \mapsto m^{-2} N, \quad \quad \quad \Delta \mapsto m^{-2 n} \Delta.\]


For later reference, it will be convenient to explicate the definition of ``level'' in terms of matrices.  To that end, we note first that for any basis $e_1,\dotsc,e_n$ of $V$, we may find symmetric matrices $S$ and $P$ that represent $q$ and $Q$ in the sense that, e.g., $q(\sum x_i e_i) = \frac{1}{2} \sum \sum x_i S_{i j} x_j$.  By singular value decomposition, we may find non-singular matrices $A$ and $D$, with $D$ diagonal, so that
\begin{equation}
  P = A^t A \ \ \text{ and } \ \, S = A^t D A.
  \label{eq:singvaluedecomp}
\end{equation}
The level of $Q$ is then the operator norm of the matrix $D^{-1}$.  For instance, if we choose our basis so that $Q(\sum x_i e_i) = \frac{1}{2} \sum x_i^2$ and $q(\sum x_i e_i) = \frac{1}{2} \sum d_i x_i^2$, then the level of $Q$ is $\max |d_i|^{-1}$.

\begin{remark}
  We may relate the above definition of level to more standard notions.  We recall that the form $Q$ is a \emph{majorant} of $q$ if $P S^{-1} P = S$, or equivalently, if $D^2 = 1$.  Suppose that $|q| \leq Q$.  Then the level of $Q$ is always at least $1$, and it is equal to $1$ if and only if $Q$ is a majorant of $q$.  Indeed, the assumption $|q| \leq Q$ implies that $|d_i| \leq 1$ for each $i$, with equality precisely when $D^2 = 1$.
\end{remark}

\subsection{Duality}

We note that replacing $L$ (resp. $Q$) with its dual $L^\vee$ (resp. $Q^\vee$) has the following effect on the invariants:
\begin{equation} (C, N, \Delta) \mapsto (1/N, 1/C, 1/\Delta).
  \label{eq:elem-div-dual}
\end{equation}
In what follows, this relation allows us to reduce slightly the number of computations required.  For instance, we can read off the invariants of the dual of an Eichler order (or its traceless submodule) from those of the Eichler order itself.



\subsection{Adelic invariants}
\label{sec:adelic-invariants}

Let $(Q,L)$ be a pair consisting of a positive-definite quadratic form $Q$ and a lattice $L$ as above.  We define the \emph{content} (resp. \emph{level}, \emph{discriminant}) of the pair to be the product of the corresponding invariants of $Q$ and $L$.

We note that the invariants $C, N, \Delta$ assigned to the pair $(Q,L)$ are invariant by rational homotheties, i.e., replacing $L$ by $m L$ and $Q$ by $Q_m$ for the same non-zero rational scalar $m$, and also under automorphisms of $V$ that preserve $q$.


Furthermore, the discriminant of the pair $(Q,L)$ is the same as the determinant of the Gram matrix of $Q$ with respect to a $\mathbb{Z}$-basis of $L$, as the discriminant of $Q$ is nothing but the inverse of the determinant of the matrix $D$ in the singular value decomposition \eqref{eq:singvaluedecomp}.

\subsection{Statement of result}

\begin{proposition}
  Let $V$ be an $n$-dimensional $\mathbb{Q}$-vector space.  Let $q : V \rightarrow \mathbb{Q}$ be an anisotropic quadratic form.  Let $L \subset V$ be a lattice.  Let $Q : V_{\mathbb{R}} \rightarrow \mathbb{R}$ be a positive-definite quadratic form.  Let $\lambda_1 \leq \dotsb \leq \lambda_n$ denote the successive minima of $Q$ on $L$ (see Definition \ref{defn:succ-min}).  Let $C, N, \Delta$ denote the content, level and discriminant of the pair $(Q,L)$.  Then,
  \begin{enumerate}[(i)]
  \item $\lambda_1 \geq C^{1/2}$,
  \item $\lambda_1 \dotsb \lambda_n \asymp \Delta^{1/2}$, and
  \item $\lambda_1 \dotsb \lambda_{n-1} \gg (\Delta / N)^{1/2}$.
  \end{enumerate}
  In particular, for $n = 3$, we have for all $X > 0$ that

  \[| \{ v \in L : Q(v) \leq X^2 \}| \ll 1 + \frac{X}{\sqrt{C}} + \frac{X^2}{\sqrt{\Delta/N}} + \frac{X^3}{\sqrt{\Delta}}.\]
  \label{prop:anisotropiclatticecount}
\end{proposition}

\begin{remark}
  This last estimate is scale-invariant in the sense that replacing $(Q,L)$ with $(Q_m, m L)$ for a positive rational scalar $m$ has no effect on the RHS.  This feature is not surprising in view of the multiplication-by-$m$ bijection $\{ v \in L : Q(v) \leq X^2 \} \cong \{ v \in m L : Q_m(v) \leq X\}$.  The estimate is likewise invariant under replacing $(Q,X)$ by $(Q_m, X/m^2)$ for some non-zero real number $m$, as one might expect for similar reasons.
\end{remark}

\begin{proof}
  We follow the basic strategy of Blomer--Michel \cite[\S3, \S4]{MR3103131}, who established the corresponding result for $q$ positive-definite and $Q=q$.

  Let $v$ be a non-zero element of $L$.  Since $L$ is anisotropic, we have $q(v) \neq 0$. Let $C_Q, C_L$ denote the content of $Q$, respectively $L$. By the definition of the content, we have $C_L|q(v)$ and $Q \ge C_Q |q|$.  Thus, $Q(v) \geq C_Q |q(v)| \geq C_QC_L = C$, giving (i).

  By Lemma \ref{lem:smalllatticegeneraotrs}, we may find a basis $e_1,\dotsc,e_n$ of $L$ so that the submodules $L_m \coloneqq \sum_{j \leq m} \mathbb{Z} e_j$ have covolume $\asymp \prod_{j \leq m} \lambda_j$ in their real span, with volume defined using the restriction of $Q$.  On the other hand, that covolume is the square root of the Gram determinant of $Q$ on $L_m$.  Write $\det(Q,L_m)$ and $\det(q,L_m)$ for the respective Gram determinants of $Q$ and $q$.  Then \begin{equation}
  	\prod_{j \leq m} \lambda_j \asymp \det(Q,L_m)^{\frac{1}{2}}.  
  	\label{eq:vol-sub-lattice}
  \end{equation} Since $\det(Q,L_n) = \Delta$ by the remark in \S\ref{sec:adelic-invariants}, the case $m=n$ of this estimate gives (ii).

  For the proof of (iii), observe first that in view of (ii), it is equivalent to check that $\lambda_n \ll (N_L N_Q)^{1/2}$, where $N_L$, respectively $N_Q$, is the level of $L$, respectively $Q$.  To that end, write $P = A^t A$ and $S = A^t D A$ for the matrices of $Q$ and $q$, as in \S\ref{sec:arch-invar}, and consider the final matrix entry $(P^{-1})_{n n}$ of the inverse of $P$.  Cramer's rule expresses $(P^{-1})_{n n} = \det(Q,L_{n-1}) / \det(Q, L_n)$, so by the cases $m=n-1,n$ of \eqref{eq:vol-sub-lattice}, we have $(P^{-1})_{n n} \asymp 1/\lambda_n^2$.  On the other hand, since $N_Q$ bounds the operator norm of $D^{-1}$, we have
  \[|(S^{-1})_{n n}| = |\langle A^{-t} e_n, D^{-1} A^{-t} e_n \rangle| \leq N_Q \langle A^{-t} e_n, A^{-t} e_n\rangle = N_Q (P^{-1})_{n n}.\] Cramer's rule likewise expresses $(S^{-1})_{n n} = \det(q,L_{n-1}) / \det(q,L_n)$ as a ratio of Gram determinants.  Since $q$ is anisotropic, both determinants are non-zero.  Since $2N_L (S^{-1})_{n n} \in \mathbb{Z}$, it follows that $1/(2N_L) \leq |(S^{-1})_{n n}|$.  Thus $1 / \lambda_n^2 \asymp (P^{-1})_{n n} \geq 1/(2N_L N_Q)$, giving the required estimate.
\end{proof}



\section{Type I estimates}\label{sec:type-i-estimates}
The local computations of Section \S\ref{sec:order-theor-prel}, together with the behavior of invariants under duality recorded in \S\ref{eq:elem-div-dual}, imply that the elementary divisors of the Gram matrix of the reduced trace form on $g^{-1}R(\ell)^0g$ are given by
\begin{equation}
  \begin{cases} \left(\frac{1}{\ell},\frac{d_B N}{\ell^2},\frac{2d_B N}{\ell}\right), & 2 \!\not | \, d_BN, \\
    \left(\frac{2}{\ell},\frac{d_BN}{\ell^2},\frac{d_BN}{\ell}\right), & 2 | d_BN, 2 \!\not | \, \ell, \\
    \left(\frac{1}{\ell},\frac{d_BN}{\ell^2},\frac{d_BN}{2\ell}\right), & 2 | d_BN, 2 | \ell.
  \end{cases}
  \label{eq:elem-div-order}
\end{equation}
Hence, the content, level, and discriminant of $g^{-1}R(\ell)^0g$ with respect to the reduced norm are comparable to $1/\ell$, $d_BN/\ell$, and $(d_BN)^2/\ell^4$ respectively. Here ``comparable to'' means the ratios are bounded from above and below by positive constants. Suppose that the reduced norm on $R$ is anisotropic. In this case, we wish to apply Proposition \ref{prop:anisotropiclatticecount} to the lattice $g^{-1}R(\ell)^0g$ with $q$ given by the reduced norm. As a first choice, we let $Q=P+\delta^{-1}u$, whose content, level, and discriminant are comparable to $1,\delta^{-1}$, and $\delta^{-2}$ respectively. This yields that the first successive minima of $g^{-1}R(\ell)^0g$ with respect to $P+\delta^{-1}u$ is $\gg \ell^{-\frac{1}{2}}$, and hence also with respect to $\Omega(\delta,1) \cap B_{\infty}^0$. Furthermore,
$$
|g^{-1} R(\ell)^0 g \cap \Omega(\delta, T)| \prec 1 + \ell^{\frac{1}{2}} T + \frac{\ell^{\frac{3}{2}} \delta^{\frac{1}{2}} }{(d_B N)^{\frac{1}{2}}} T^2 + \frac{\ell^2 \delta}{d_B N} T^{3}.
$$
This proves the first half of Theorem \ref{thm:typeI}. Similarly, we have for the choice $Q=P+\delta^{-1} |X|^2$ that the content, level, and discriminant $Q$ are comparable to $1,\delta^{-1}$, and $\delta^{-1}$ respectively. Thus the first successive minima of $g^{-1}R(\ell)^0g$ with respect to $\Psi(\delta,1)\cap B_{\infty}^0$ is $\gg \ell^{-\frac{1}{2}}$ and
$$
|g^{-1} R(\ell)^0 g \cap \Psi(\delta, T)| \prec 1 + \ell^{\frac{1}{2}} T + \frac{\ell^{\frac{3}{2}}}{(d_B N)^{\frac{1}{2}}} T^2 + \frac{\ell^2 \delta^{\frac{1}{2}}}{d_B N} T^{3},
$$
which is the second half of Theorem \ref{thm:typeI}.

We now turn to the case that the quaternion algebra B is split. Here, we proceed in a more ad hoc manner. First, we note that $R(\ell)^0$ is normalized by the Atkin--Lehner operators. Thus, we need only consider $g \in G(\mathbb{R})$ such that $g \cdot i =x+iy$ has maximal imaginary part under the action of the Atkin--Lehner operators. In particular, we have $H(g)=y$. Let $\lambda_1\le \lambda_2 \le \lambda_3$ be the successive minima (Definition \ref{defn:succ-min}) of the closed convex $0$-symmetric set $\Omega(\delta,1)\cap B_{\infty}^0$
with respect to the lattice $g^{-1}R(\ell)^0g$. Since $\Omega(\delta,1)$ is both left and right $K_{\infty}$-invariant, we may further assume that $g=\sm 1 & x \\ & 1 \esm \diag(y^{\frac{1}{2}} , y^{-\frac{1}{2}})$. By Lemma \ref{lem:latticepointcount}, we have
$$
|g^{-1}R(\ell)^0g \cap \Omega(\delta,T) | = |g^{-1}R(\ell)^0g \cap T\Omega(\delta,1)| \asymp 1+\frac{T}{\lambda_1}+\frac{T^2}{\lambda_1\lambda_2}+\frac{T^3}{\lambda_1\lambda_2\lambda_3}.
$$
Let $\beta_0 = \sm a & b \\ c & -a \esm \in R(\ell)^0$, thus $a \in \mathbb{Z}, b \in \frac{1}{\ell}\mathbb{Z}, c \in \frac{N}{\ell}\mathbb{Z}$. Let
$$
\alpha_0= g^{-1} \beta_0 g = \begin{pmatrix} a-cx & \frac{1}{y}(2ax+b-cx^2) \\ cy & cx-a \end{pmatrix}.
$$
Suppose $\beta_0 \neq 0$. If $(a,c)\neq (0,0)$, then by Lemma \ref{lem:spacing} we have $P(\alpha_0) \ge \frac{1}{2}|cz-a|^2 \ge \frac{1}{2\ell}$. Otherwise, we have $u(\alpha_0)= (\frac{b}{y})^2 \ge \frac{1}{2(\ell y)^2}$. Hence, we have $\lambda_1 \gg \min\{ \ell^{-\frac{1}{2}} , \ell^{-1} y^{-1} \delta^{-\frac{1}{2}} \}$. In order to get a lower bound on $\lambda_1\lambda_2$ and $\lambda_1\lambda_2\lambda_3$, we shall give an upper bound on $|g^{-1}R(\ell)^0g \cap \Omega(\delta,T)|$ along the lines of Harcos--Templier \cite{MR3038127}. First, we bound the number of choices of $c$ by $\ll 1+\frac{\ell T}{Ny}$ as $|cy| \le P(\alpha_0)^{\frac{1}{2}} \le T$. For each such choice of $c$, the equation
$$
\frac{1}{y^2} \left|-cz^2+2az+b\right|^2 = u(\alpha_0) \le \delta T^2
$$
defines a circle of radius $\delta^{\frac{1}{2}}T y$ and center $cz^2$ in which we need to count lattice points of the lattice generated by $2z$ and $\frac{1}{\ell}$. This lattice has covolume $2y/\ell$ and first successive minima $\ge (\ell N)^{-\frac{1}{2}}$ by Lemma \ref{lem:spacing}. We may thus apply Lemma \ref{lem:latticeball} to bound the number of $(a,b)$  by $\ll 1+\ell^{\frac{1}{2}}N^{\frac{1}{2}}\delta^{\frac{1}{2}}T y + \ell\delta T^2 y$. We obtain
$$
\frac{T^2}{\lambda_1\lambda_2} + \frac{T^{3}}{\lambda_1\lambda_2\lambda_3} \ll \left( 1+\frac{\ell T}{Ny} \right)\left( 1+\ell^{\frac{1}{2}}N^{\frac{1}{2}}\delta^{\frac{1}{2}}Ty + \ell\delta T^2 y \right).
$$
By letting $T$ tend to $\infty$, it follows that $\lambda_1\lambda_2\lambda_3 \gg N \ell^{-2}\delta^{-1}$.  By taking $T=N^{\frac{1}{2}}/(\ell\delta)^{\frac{1}{2}}$, we obtain
$$
\frac{1}{\lambda_1\lambda_2} \ll \frac{\ell\delta}{N}+\frac{\ell^{\frac{3}{2}}\delta^{\frac{1}{2}}}{N^{\frac{3}{2}}y}+\ell\delta y + \frac{\ell^{\frac{3}{2}}\delta^{\frac{1}{2}}}{N^{\frac{1}{2}}} \ll \frac{\ell^{\frac{3}{2}}\delta^{\frac{1}{2}}}{N^{\frac{1}{2}}} + \ell \delta y.
$$
We thereby conclude the proof of the final case of Theorem \ref{thm:typeI}.


\section{Type II estimates}\label{sec:type-ii-estimates}

\subsection{Bounds for representation numbers of binary quadratic forms}

\begin{lemma}\label{lem:N2-1}
  Let $M$ be a free $\mathbb{Z}$-module of rank $2$.  Let $q : M \rightarrow \mathbb{Z}$ be a nondegenerate integral binary quadratic form.  Let $Q$ be a positive-definite quadratic form on $M \otimes \mathbb{R}$ such that $|q| \leq Q$.  Let $n$ be a non-zero integer, and let $X \geq 1$.  Set
  \[
    S \coloneqq \{ \beta \in M : q(\beta) = n, Q(\beta) \leq X^2 \}.
  \]
  Then,
  \[
    |S| \ll_{\eps} (X |n|)^\eps.
  \]
\end{lemma}
\begin{proof}
  We begin with a preliminary reduction.  Suppose we can find a lattice $M'$ in $M \otimes \mathbb{Q}$ that contains $M$ and on which the $\mathbb{Q}$-bilinear extension of $q$ is integral.  It suffices then to verify the lemma after replacing $M$ with $M'$.  Indeed, doing so enlarges the set $S$.  In particular, we may assume that the quadratic form $M$ is \emph{primitive}.

  Suppose now that $M$ is primitive and anisotropic.  Without loss of generality, $M$ is either positive-definite or indefinite.  By the form-ideal correspondence, we may assume then that $M$ is an invertible ideal of an order $\mathfrak{o}$ in a quadratic field, with $q$ given by the element norm $\nu$ divided by the ideal norm $\nu(M)$ of $M$:
  \[
    q(\beta) = \nu(\beta)/ \nu(M).
  \]
  We will establish the estimate
  \begin{equation}\label{eqn:bound-S-non-split-quadratic-case}
    |S| \ll \log(2 + X^2/|n|) \tau(n),
  \end{equation}
  which suffices in view of the divisor bound.  Let $\mathfrak{o}_{\max}$ denote the maximal order in the quadratic field containing $\mathfrak{o}$.  For each $\beta \in S$, the $\mathfrak{o}$-ideal $M^{-1} \beta$ has norm $|n|$, as does the $\mathfrak{o}_{\max}$-ideal $\mathfrak{o}_{\max} M^{-1} \beta$.  The number of $\mathfrak{o}_{\max}$-ideals of norm $|n|$ is at most $\tau(n)$.  Suppose two elements $\beta_0, \beta \in S$ give rise to the same $\mathfrak{o}_{\max}$-ideal.  Then, $\beta/\beta_0$ is a norm one unit in $\mathfrak{o}_{\max}^{\times}$.  The required estimate follows in the imaginary quadratic case (without the logarithmic factor) because $| \mathfrak{o}_{\max}^{\times}| \leq 6$.  In the real quadratic case, we fix a positive generator $\eta$ for the group $\cong \mathbb{Z}$ of non-root-of-unity norm one units in $\mathfrak{o}_{\max}$ and write $\beta = \pm \beta_0 \eta^{\ell}$ for some $\ell \in \mathbb{Z}$.  It will suffice then to verify that $\ell \ll \log(2 + X^2/|n|)$.  To that end, we estimate $q(\beta_0 + \beta)$ in two ways.  On the one hand, the triangle inequality for the Euclidean norm defined by $Q$ gives the upper bound $|q(\beta_0 + \beta)| \leq Q(\beta_0 + \beta) \ll Q(\beta_0) + Q(\beta) \ll X^2$.  On the other hand, the multiplicativity of $\nu$ gives the identity $q(\beta_0 + \beta) = n \nu(1 \pm \eta^{\ell})$.  The lower bound $\nu(1 \pm \eta^{\ell}) \geq \frac{1}{4} \cdot 1.618^{\ell}$ for fundamental units now yields the required estimate for $\ell$.

  It remains to consider the case that $M$ is isotropic and $q$ nondegenerate. In that case, after applying our preliminary reduction to enlarge $M$ if necessary, we may assume that $M = \mathbb{Z}^2$ and $q(x,y) = x y$.\footnote{Indeed, we may choose a basis $e_1, e_2$ for $M$ with $e_1$ isotropic.  Then, $q$ is given with respect to the coordinates $x e_1 + y e_2$ by $q(x,y) = a x y + b y^2$ for some $a,b \in \mathbb{Z}$, with $a \neq 0$.  Then $q (\tfrac{x - b y}{a}, y) = x y$ and $M \subseteq M' := \left\{ \tfrac{x - b y}{a} e_1 + y e_2 : x, y \in \mathbb{Z}  \right\}$, so $(M',q)$ gives the required enlargement.}
  Since $n \neq 0$, the divisor bound then gives $|S| = 2 \tau(n) \ll_{\eps} |n|^{\eps}$.
\end{proof}

\subsection{Local quaternionic preliminaries}

Let $B$ be a quaternion algebra over the rationals.  We write $d_B$ for its reduced discriminant and $q$ for its reduced norm.


\subsubsection{Non-archimedean preliminaries}

Let $R \subset B$ be an Eichler order of level $N$, with $N$ coprime to $d_B$.

\begin{lemma}
  \label{lem:N2-2}
  For $x,y \in R(\ell)^0$, we have $[x,y] \in \frac{1}{\ell} R^0$ and $q([x,y]) \in \frac{d_B N}{\ell^3} \mathbb{Z}$.
\end{lemma}
\begin{proof} This follows from the local computations in Lemma \ref{lem:commutatorcong} together with the fact that the trace of a commutator is zero.
\end{proof}

\subsubsection{Archimedean preliminaries} Recall the notation
``$\Omega$'' from \S\ref{sec:archimedean-regions}
and ``$P$'' from \eqref{eq:pgamma-coloneqq-a2}.






\begin{lemma}
  \label{lem:N2-3}
  For $0 < \delta \leq 1$ and $x, y \in \Omega(\delta,T)$, we have
  \[q([x,y]) \ll \delta T^4\]
  and
  \[P([x,y]) \ll \delta T^4.\]
\end{lemma}

\begin{proof} For $\delta = 1$, the required estimates reduce via homogeneity to the compactness of the unit ball and the continuity of multiplication.  We turn to the case $0 < \delta < 1$.  We embed $B_{\infty} \hookrightarrow \Mat_{2 \times 2}(\mathbb{C})$ by $\i \mapsto \sm i & 0 \\ 0 & -i \esm$, $\j \mapsto \sm 0 & 1 \\ \pm 1 & 0 \esm$.  Then, $P$ is asymptotic to the restriction of the squared Euclidean norm on the matrix entries, while $q$ is the restriction of the determinant.  We may assume that $x=[a_1,b_1,c_1]$ and $y=[a_2,b_2,c_2]$ are non-zero.  We may assume (by the known $\delta = 1$ case) that $\delta$ is sufficiently small.  The required conclusion then reduces via homogeneity to the following assertion: the commutator of any two matrices of the form
  \begin{equation*}
    \begin{pmatrix}
      a_1 i & O(\delta^{1/2}) \\ O(\delta^{1/2}) & - a_1 i
    \end{pmatrix}
    \quad
    \text{ and }
    \begin{pmatrix}
      a_2 i & O(\delta^{1/2}) \\ O(\delta^{1/2}) & - a_2 i
    \end{pmatrix}
  \end{equation*}
  is of the form $\sm O(\delta) & O(\delta^{1/2}) \\ O(\delta^{1/2}) & O(\delta) \esm$, and in particular has Euclidean norm $O(\delta^{1/2})$ and determinant $O(\delta)$.  Indeed, the product of any two matrices of the indicated form is readily computed to be of the form
  $\sm -a_1a_2 + O(\delta) & O(\delta^{1/2}) \\
  O(\delta^{1/2}) & -a_1a_2 + O(\delta) \esm$, hence the commutator of two such matrices, being a difference of such products, has the required form.
\end{proof}

\subsection{The non-split case}

We retain the above setting and further assume that $B$ is \emph{non-split} 
and that the level $N$ of the Eichler order $R$ is \emph{squarefree}.

The following estimate is, in some sense, the most intricate one in the paper.  It requires us to bound certain matrix counts in the the critical range (see Remark \ref{rmk:crucialrange} below) by essentially $\O(1)$, uniformly in the discriminant and level.  To achieve such uniformity seems to require the delicate argument involving commutators recorded below.

\begin{proposition}
  Let $n$ be a non-zero integer, let $0 < \delta \leq 1$, and let $T \geq \ell^{-\frac{1}{2}}$.  Then the set $S \coloneqq R(\ell)^0 \cap \Omega(\delta,T) \cap q^{-1}(\{n\})$ has cardinality
  \[|S| \ll_{\eps} (\ell T)^\varepsilon \tau(d_B N) \left( 1 + \frac{\ell^2}{d_B N} \min\left\{\delta^{\frac{1}{2}} T^2 , \frac{\delta T^4}{|n|} \right\} \right).\]
  \label{prop:typeII-omega}
\end{proposition}

\begin{remark}
  \label{rmk:crucialrange}
  The critical range is when $n \asymp T^2 \asymp d_B N(1+k) / \ell^2$ and $\delta \asymp (1+k)^{-1}$; in that range, we obtain $|S| \ll_{\eps} (d_B N (1+k) )^\varepsilon$.
\end{remark}

\begin{proof}
  Suppose $S$ is not empty and let $\gamma_1, \gamma_2 \in S$.  Our strategy will be to bound for each $\gamma_1$ the number of possibilities for $\gamma_2$.

  Set $\beta \coloneqq [\gamma_1,\gamma_2] \in \frac{1}{\ell}R^0 $ and $a \coloneqq \tr(\gamma_1 \gamma_2) \in \frac{1}{\ell} \mathbb{Z}$.  Then, $2 \gamma_1 \gamma_2 = a + \beta$, $4 n^2 = a^2 + q(\beta)$.  In particular, $\gamma_2 = (a + \beta) / 2 \gamma_1$, so it suffices to bound the number of possibilities for $a$ and $\beta$.

  Lemmas \ref{lem:N2-2} and \ref{lem:N2-3} give $q(\beta) \ll \delta T^4$ and $q(\beta) \in \frac{d_B N}{\ell^3} \mathbb{Z}$, i.e.,
  \[a^2 = 4 n^2 + O(\delta T^4),\]
  \[a^2 \equiv 4 n^2 \pod{\tfrac{d_B N}{\ell^3}},\] thus
  \[a = \pm 2 n + O\left(\min\left\{\delta^{\frac{1}{2}} T^2 , \frac{\delta T^4}{|n|} \right\} \right),\]
  \[a \equiv a_0 \pod{\tfrac{d_B N}{\ell^2}}\] for some sign $\pm$ and some residue class $a_0$ modulo $d_B N/ \ell^2$ with $a_0^2 \equiv 4 n^2 \pod{\tfrac{d_B N}{\ell^2}}$.  Since $d_B N$ is squarefree, there are at most $\tau(d_B N)$ such classes.  For each $a_0$, the number of possibilities for $a$ is $O(1 + (\ell^2 / d_B N) \cdot \min\{\delta^{\frac{1}{2}} T^2, \delta T^4 /|n|\})$.

  We now bound for each $a$ the number of possibilities for $\beta$.  Let $M$ denote the orthogonal complement in $\frac{1}{\ell} R^0$ of $\gamma_1$, thus $M = \{ \gamma \in \frac{1}{\ell}R^0 : \trace(\gamma \gamma_1) = 0 \}$.  By restricting $q$ to $M$, we obtain an integral binary quadratic form.  Since $B$ is non-split, $M$ is anisotropic.  Since $\trace(\gamma_1 \gamma_2 \gamma_1) = \trace(\gamma_2 \gamma_1 \gamma_1)$, we have $\beta \in M$.  From Lemma \ref{lem:N2-3}, we obtain $P(\beta) \ll T^4$.  Thus, $\beta$ satisfies the system
  \[\beta \in M,
    \quad q(\beta) = 4n^2 - a^2, \quad P(\beta) \ll T^4.
  \]
  Since $q(M) \subseteq \frac{1}{\ell^2} \mathbb{Z}$ and $4 n ^2 - a^2 \ll \delta T^4$, we see by Lemma \ref{lem:N2-1} that the number of possibilities for $\beta$ is $\ll_{\eps} ( \ell T)^\eps$.

  By multiplying together the number of possibilities for $\pm, a_0, a$ and $\beta$, we achieve the required bound.
\end{proof}

\subsection{Extension to the split case} \label{sec:type-II-split-ext}


Recall from \S\ref{sec:results-split-case}, that we may assume our Eichler order of level $N$ in the split quaternion algebra $B = \Mat_{2 \times 2}(\mathbb{Q})$ to be of the shape $g^{-1}Rg$, where $R=\sm \mathbb{Z} & \mathbb{Z}\\ N\mathbb{Z} & \mathbb{Z} \esm$ and $g \in G(\mathbb{R})$. Our aim is to bound the cardinality of the set
\[S \coloneqq g^{-1}R(\ell)^0g \cap \Omega(\delta,T) \cap \det{}^{-1}(\{n\}).\] In the non-split case, we had verified that
\[|S| \ll_{\eps} C\] where
\[C \coloneqq 1 + (\ell T)^\varepsilon \tau(d_B N) \left( 1 + \frac{ \ell^2}{d_B N} \min\left\{\delta^{\frac{1}{2}} T^2 , \frac{\delta T^4}{|n|} \right\}\right).\] We extend this to the split case as follows.

\begin{proposition} Let $n$ be an integer.  If $-n$ is not a square, then $|S| \ll_{\eps} C$.  If $-n$ is a square, then $|S| \ll_{\eps} C + \delta^{1/2} T \ell H(g)$, where $H$ denotes the normalized height function defined in \S\ref{sec:results-split-case}. 
  \label{prop:typeII-split}
\end{proposition}



\begin{proof}
  We proceed as in the original argument, aiming to bound for fixed $\gamma_1 \in gSg^{-1}$ the number of possible $\gamma_2 \in gSg^{-1}$.  As before, we write
  \[\alpha = \trace(\gamma_1 \gamma_2) = \gamma_1 \gamma_2 + \gamma_2
    \gamma_1 \in \tfrac{1}{\ell}\mathbb{Z},\]
  \begin{equation}\label{eq:beta-=-gamma_1}
    \beta = [\gamma_1, \gamma_2] = \gamma_1 \gamma_2 - \gamma_2
    \gamma_1 \in \tfrac{1}{\ell} R^0,
  \end{equation}
  so that
  \[2 \gamma_1 \gamma_2 = \alpha + \beta,\]
  \[ 4n^2 = q(\gamma_1 \gamma_2) = \alpha^2 + q(\beta).\] Since $\gamma_1 = (\alpha + \beta)/2 \gamma_2$, it suffices to count the number of possible pairs $(\alpha,\beta)$.  The pairs with $q(\beta) \neq 0$ may be counted as before after noting that the restriction of $q$ to the orthogonal complement of $\gamma_1$ is nondegenerate. 
  There are at most two pairs with $\beta = 0$, since then $\alpha = \pm 2n$.  We thereby reduce to counting the number of pairs for which
  \[q(\beta) = 0, \quad \beta \neq 0.\]

Recall the notation of \S\ref{sec:cusps-AL-operators}. We note that for each cusp $\mathfrak{a}$, we may and shall choose $\sigma_{\mathfrak{a}} \in \Gamma_0(N) \tau_{t}$ for some $t | N$. With this choice, we have
  \begin{equation}
  	\sigma_{\mathfrak{a}}^{-1} R \sigma_{\mathfrak{a}} = \begin{pmatrix} \mathbb{Z} & w_{\mathfrak{a}}\mathbb{Z} \\  \frac{N}{w_{\mathfrak{a}}}\mathbb{Z} & \mathbb{Z} \end{pmatrix},
  	\label{eq:generalized-cuspwidth}
  \end{equation}
 where $w_{\mathfrak{a}}$ is the cusp width of the cusp $\mathfrak{a}$. This is easily verified locally. We further introduce the following notation: for $\mathfrak{a} \in \mathbb{P}^1(\mathbb{Z})$ and $\kappa \in B^0$, we set $\kappa^{\mathfrak{a}} \coloneqq \sigma_{\mathfrak{a}}^{-1} \kappa \sigma_{\mathfrak{a}}$.

  We observe, by \eqref{eq:beta-=-gamma_1}, that $\ell \beta$ is a non-zero element of $R^0 \subseteq \Mat_{2 \times 2}(\mathbb{Z})^0$.  There is thus a unique (up to sign) primitive element $\beta_0$ of $\Mat_{2 \times 2}(\mathbb{Z})^0$ that generates $\mathbb{Q} \ell \beta \cap \Mat_{2 \times 2}(\mathbb{Z})^0$.  We then have $\beta \in \frac{1}{\ell} \mathbb{Z} \beta_0$.


  
  We may and shall choose $\mathfrak{a}$ so that $\beta_0^{\mathfrak{a}} = \sm 0 & \pm 1 \\ 0 & 0 \esm$.  Note that $\beta$ (equivalently, $\beta_0$) is orthogonal not only to $\gamma_1$ (as was used in the original argument), but also to $\gamma_2$.  From this, we deduce that
  \[\gamma_2^{\mathfrak{a}} = \begin{pmatrix} a & b \\ 0 & -a
    \end{pmatrix}\] for some $a \in \mathbb{Z}$ and $b \in \frac{1}{\ell} \mathbb{Z}$.  We have $n = q(\gamma_2) = -a^2$, which shows that $-n$ must be a square and also that there are at most two possibilities for $a$.  It remains to verify that the number of possibilities for $b$ is $O(1+\delta^{1/2} T \ell H(g))$.  We will show in fact that the number of $b$ is
  \[O(1 + \delta^{1/2} T \ell y_{\mathfrak{a}} / w_{\mathfrak{a}}),\]
  where $\sigma_{\mathfrak{a}}^{-1}z=z_{\mathfrak{a}}=x_{\mathfrak{a}}+iy_{\mathfrak{a}}$. To see this, observe first that the condition $\gamma_2 \in R(\ell)^0 \subseteq \frac{1}{\ell} R^0 \Leftrightarrow \gamma_2^{\mathfrak{a}} \in \sigma_{\mathfrak{a}}^{-1}R(\ell)^0\sigma_{\mathfrak{a}} \subseteq \frac{1}{\ell} \sigma_{\mathfrak{a}}^{-1}R^0\sigma_{\mathfrak{a}}$ yields the congruence $b \equiv 0 \pod{\tfrac{w_{\mathfrak{a}}}{\ell}}$, see equation \eqref{eq:generalized-cuspwidth}.  The condition $\gamma_2 \in g\Omega(\delta,T)g^{-1}$ may be restated as
  \[\gamma_2^{\mathfrak{a}} \in \sigma_{\mathfrak{a}}^{-1}g \Omega(\delta, T) (\sigma_{\mathfrak{a}}^{-1}g)^{-1}.\]
  Let $g'$ be an upper-triangular element of $G(\mathbb{R})$ for which $g' \cdot i = \sigma_{\mathfrak{a}}^{-1}g i=\sigma_{\mathfrak{a}}^{-1} z = z_{\mathfrak{a}}$. Then, by the $K_{\infty}$ invariance of $\Omega(\delta,T)$ on the left and right, we have $\gamma_2^{\mathfrak{a}} \in g' \Omega(\delta, T) (g')^{-1}$. We compute
  \[(g')^{-1} \gamma_2^{\mathfrak{a}} g' = \begin{pmatrix} a & y_{\mathfrak{a}}^{-1} (b + 2 a x_{\mathfrak{a}}) \\ 0 & -a \end{pmatrix} \in \Omega(\delta,T).\] This last condition forces $b$ to lie in an interval of length $O(\delta^{1/2} T y_{\mathfrak{a}})$.  We thereby obtain the required bound for the number of possible $b$'s.
\end{proof}

\subsection{Proof of Theorem \ref{thm:typeII}} 
We may split the set $\Omega(\delta, T)$ into $\Omega(1/16, 4\delta^{\frac{1}{2}} T)$ and the dyadic sets which are comprised of the elements $[a,b,c] +d\in B_{\infty}$ for which
\[
  \tfrac{1}{2} \delta_j T^{2} \leq a^2 + b^2 + c^2 + d^2 \leq \delta_j T^2, \quad b^2 + c^2 \leq \delta T^2,
\]
for some $16 \delta \le \delta_j \leq 1$. We note that these are contained in $\Omega(\delta \delta_j^{-1} , \delta_j^{\frac{1}{2}} T)$. In order for the dyadic sets to contain an element of trace $0$ and norm $n$, one must have $|n| \asymp a^2 \asymp \delta_j T^2$. Hence, if we apply Propositions \ref{prop:typeII-omega} and \ref{prop:typeII-split}, we get
$$
|g^{-1} R(\ell)^0 g \cap \Omega(\delta \delta_j^{-1}, \delta_j^{\frac{1}{2}}T) \cap \det{}^{-1}(\{n\})| \prec 1+ \ell \delta^{\frac{1}{2}}H(g)T +\frac{\ell^2}{d_BN} \delta T^2,
$$
where we used $|n|\asymp \delta_jT^2$, and
$$
|g^{-1} R(\ell)^0 g \cap \Omega(\tfrac{1}{16}, 4 \delta^{\frac{1}{2}}T) \cap \det{}^{-1}(\{n\})| \prec 1+ \ell \delta^{\frac{1}{2}}H(g)T +\frac{\ell^2}{d_BN} \delta T^2.
$$
If $\delta$ happens to be very small, say, if $\delta \le (16 d_BNT^2)^{-1}$, then it suffices to consider only $\delta_j \ge (d_BNT^2)^{-1}$ and the final set $\Omega(\delta d_B N T^2,(d_BN)^{-\frac{1}{2}})$. This avoids a factor $\delta^{-o(1)}$ for a too small $\delta$. We conclude Theorem \ref{thm:typeII}.

\appendix
\section{The Theta Lift} \label{sec:appendix-theta}
In this section, we use the theta correspondence for the reductive dual pair ($\operatorname{O}_{\det}$, $\operatorname{SL}_2$) to derive the necessary properties of the theta kernels in use. The group $\operatorname{O}_{\det}$ is the affine algebraic group over $\mathbb{Q}$ representing the orthogonal group of $(B, \det)$. Recall that $G$ is the linear algebraic group defined over $\mathbb{Q}$ satisfying $G(L)=\lfaktor{L^\times}{(B\otimes L)^\times}$ for any $\mathbb{Q}$-algebra $L$.
Denote by $M$ the algebraic group representing the functor
\begin{equation*}
M(L)=\lfaktor{\Delta(L^\times)}{}{\left\{(g_1,g_2)\in (B\otimes L)^\times \times (B\otimes L)^\times \colon \det g_1=\det g_2 \right\}},
\end{equation*}
for any $\mathbb{Q}$-algebra $L$. 
Then, $M$ is defined over $\mathbb{Q}$ and it is isomorphic to the special orthogonal group $\operatorname{SO}_{\det}$ via the  action $(g_1,g_2).x=g_1 x g_2^{-1}$. We also define the algebraic group $G'$ over $\mathbb{Q}$ to be the simply-connected form of $G$, i.e., $G'(L)=\operatorname{SL}_1(B\otimes L)$ for any $\mathbb{Q}$-algebra $L$. The natural map $G'\times G' \to M$ is an isogeny. The left-hand side is the simply-connected form, i.e.\ the Spin group, and the right-hand side the adjoint one. 

The determinant map provides two exact sequences
\begin{equation*}
  1\to\prod_{v\leq\infty}\langle(I,-I)\rangle\to M(\mathbb{A})\xrightarrow{\iota^{(1)}} (G\times G)(\mathbb{A})\xrightarrow{\det\left(\frac{\bullet}{\bullet}\right)}\lfaktor{\mathbb{A}^{\times 2}}{\mathbb{A}^\times}\to 1,
\end{equation*}
\begin{equation*}
  1\to\prod_{v\leq\infty}\langle(-I,-I)\rangle\to \left(G'\times G'\right)(\mathbb{A})\xrightarrow{\iota'} M(\mathbb{A})\xrightarrow{\det}\lfaktor{\mathbb{A}^{\times 2}}{\mathbb{A}^\times}\to 1.
\end{equation*}

\subsection{Restriction of automorphic representations}\label{sec:automorphic-restriction}
We would like to understand the behavior of irreducible cuspidal representation under pull-back by $\iota^{(1)}$ and $\iota'$. We proceed to discuss some generalities that apply to these isogenies.
Let $H$ be a semisimple algebraic group defined over $\mathbb{Q}$, and fix a maximal compact open subgroup $K_f=\prod_{v<\infty} K_v<H(\mathbb{A}_f)$. Let $K_\infty<H(\mathbb{R})$ a maximal compact real subgroup and set $K=K_\infty K_f$. For an automorphic representation $\Pi$ of $H(\mathbb{A})$, we denote by $\Pi^\infty\subset \Pi$ the dense subset of $K$-finite vectors. That is, every $v\in\Pi^\infty$ is invariant under a finite-index subgroup of $K_f$ and its $K_\infty$-orbit spans a finite dimensional subspace. Then, $\Pi^\infty$ is an admissible $H(\mathbb{A})$ representation.

Assume $\jmath\colon H'\to H$ is a homomorphism of algebraic groups satisfying the following conditions:
\begin{enumerate}
\item $\ker \jmath$ is a finite central subgroup of $H'$,
\item $\Img \jmath$ is a normal subgroup of $H$, and $\operatorname{coker} \jmath$ is a finite abelian group.
\item $H(\mathbb{Q}_p) = K_p j(H'(\mathbb{Q}_p))$ for almost all primes $p$.
\end{enumerate}
In particular, the Lie algebras of $H$ and $H'$ are isomorphic, hence $H'$ is semisimple.  These assumptions are satisfied when $\jmath$ is an isogeny of semisimple algebraic groups and for the inclusion map $\operatorname{SO}_{\det}\to\operatorname{O}_{\det}$.  The group $H_{\mathrm{char}}=\dfaktor{H(\mathbb{Q})}{H(\mathbb{A})}{j(H'(\mathbb{A}))}$ is a compact abelian group, that is often infinite. In particular, $j(H'(\mathbb{A}))$-orbits on $[H]$ can have measure zero and the operation of restricting a function in $L^2([H])$ to an orbit of $j(H'(\mathbb{A}))$ is ill-defined. Nevertheless, we have the following.
\begin{lemma}\label{lem:pull-back-properties}
The pullback $\jmath^*\colon \Res^{H(\mathbb{A})}_{H'(\mathbb{A})} L^2([H])^\infty\to L^2([H'])^\infty$ is a well-defined operator that restricts to an intertwining operator of the cuspidal spectrum $\jmath^*\colon \Res^{H(\mathbb{A})}_{H'(\mathbb{A})} L^2_\mathrm{cusp}([H])^\infty\to L_\mathrm{cusp}^2([H'])^\infty$.  Moreover, $f\in L^2([H])^\infty$ is cuspidal if and only if for any class $[h]\in H_\mathrm{char}$, the vector $\jmath^*R_h f$ is cuspidal for some representative $h\in [h]\subset H(\mathbb{A})$.
\end{lemma}
\begin{remark*}
Notice that $\jmath^*$ does not preserve inner products. 
\end{remark*}
\begin{proof}
Restricting to a $j(H'(\mathbb{A}))$-orbit is a well-defined operation on $L^2([H])^\infty$, because every vector  $v_0\in L^2([H])^\infty$ is invariant under some compact-open subgoup $K_0<H(\mathbb{A}_f)$ and $\faktor{H_{\mathrm{char}}}{K_0}$ is finite. We now show that the $L^2$-norm of $\jmath^* v_0$ is finite.
The push-forward of the probability  Haar measure on $[H]$ to $H_\mathrm{char}$ is invariant under the action of $H(\mathbb{A})$, hence it is the probability Haar measure on $H_\mathrm{char}$. If we disintegrate the Haar measure on $[H]$ under the factor map $[H]\to H_\mathrm{char}$, then the atoms are exactly the $\jmath(H'(\mathbb{A}))$-orbits and the conditional measure on a.e.\@ atom is $\jmath(H'(\mathbb{A}))$-invariant, hence it is the push-forward of the probability Haar measure on $[H']$ to the atom. We can now deduce that
\begin{equation*}
\|v_0\|_2^2=|H_\mathrm{char}\slash K_0|^{-1} \sum_{h\in H_\mathrm{char}\slash K_0} \| \jmath^*(R_h v_0) \|_{2,[H']}^2.
\end{equation*}
Hence, $ \| \jmath^*(v_0) \|_{2,[H']}\leq \sqrt{|H_\mathrm{char}\slash K_0|} \|v_0\|_2<\infty$.

We show next that the image of a cuspidal vector is cuspidal. 
Fix $v_0\in  L^2([H])^\infty$. Let $P<H'$ be a parabolic subgroup defined over $\mathbb{Q}$ and let $N_P$ be its unipotent radical. The kernel $\ker\jmath$ is a central subgroup, hence it is diagonalizable and its intersection with $N_P$ is trivial. Then $\jmath\restriction_{N_P}$ is an isomorphism onto its image $N_{\tilde{P}}$, which is the unipotent radical of a parabolic $\tilde{P}<H$. Specifically, $\tilde{P}$ is the parabolic associated to the same root data as $P$. For every $g\in H'(\mathbb{A})$, we have, writing $c_P$ and $c_{\tilde{P}}$ for the maps assigning to a function its corresponding constant term,
\begin{equation*}
c_P \jmath^* v_0(g)=\int_{[N_P]} \jmath^*(v_0)( n g) \dif n=\int_{[N_{\tilde{P}} ]}v_0(n \jmath(g)) \dif n =c_{\tilde{P}} v_0(\jmath(g)).
\end{equation*}
Hence, the constant term of the push-forward of a cuspidal vector vanishes.
This formula also establishes the last claim.
\end{proof}

Next, we describe the transformation of the Haar measure. If $K<H(\mathbb{A})$ is a compact subgroup, we denote by $[H]_K$ the double quotient $\dfaktor{H(\mathbb{Q})}{H(\mathbb{A})}{K}$.
\begin{lemma}\label{lem:isogeny2isomorphism}
Fix compact open subgroups $K'_f<H'(\mathbb{A}_f)$ and $K_f<H(\mathbb{A}_f)$ satisfying $\jmath^{-1}(K_f)=K'_f$.
Assume the following conditions:
\begin{enumerate}
\item  $\faktor{H_\mathrm{char}}{K_f}=1$.
\item The preimage of $K_f \bmod \jmath(H'(\mathbb{A}))$ under the map $\coker \jmath (\mathbb{Q})\to\coker\jmath(\mathbb{A})$ is trivial.
\item $\ker \jmath\restriction_{H'(\mathbb{A})}< \operatorname{Z}_{H'}(\mathbb{Q}) \cdot K'_f$.
\end{enumerate}
Then, the induced map $\jmath\colon [H']_{K'_f}\to [H]_{K_f}$ is a homeomorphism and an isomorphism of Borel measure spaces, when each space is endowed with the respective probability Haar measure.
\end{lemma}
\begin{proof}
To show the map $[H']_{K'_f}\to [H]_{K_f}$ is surjective, we need to find for every $x\in H(\mathbb{A})$ elements $\gamma\in H(\mathbb{Q})$, $k\in K_f$ and $x'\in H'(\mathbb{A})$, such that $x=\gamma \jmath(x') k$. Equivalently, we need the class of $[\gamma^{-1} x k^{-1}]=[x k^{-1}]$ in $H_\mathrm{char}$ to be trivial and this follows from the assumption that $\faktor{H_\mathrm{char}}{K_f}=1$. 

To verify the map is injective, we consider $x'_1,x'_2\in H'(\mathbb{A})$ satisfying $\jmath(x'_1)= \gamma \jmath(x'_2) k$, with $\gamma\in H(\mathbb{Q})$ and $k\in K_f$. We first demonstrate that $\gamma \in \jmath(H'(\mathbb{Q}))$. Because $\gamma=\jmath(x'_{1,\infty}) k^{-1}\jmath({x'_{2,\infty}}^{-1})$ and $\coker \jmath$ is abelian, we see that the class of $\gamma$ in $\coker\jmath(\mathbb{A})=\lfaktor{\jmath(H'(\mathbb{A}))}{H(\mathbb{A})}$ is the same as the class of $k$. The second assumption then implies that the class of $\gamma$ in $\coker \jmath(\mathbb{Q})=\lfaktor{\jmath(H'(\mathbb{Q}))}{H(\mathbb{Q})}$ is trivial as claimed. Hence, $\gamma=\jmath(\gamma_0)$ for some $\gamma_0\in H'(\mathbb{Q})$. We can now write $k=\jmath({x'_2}^{-1}\gamma_0^{-1}x'_1)\in \jmath(H'(\mathbb{A}))\cap K_f$. Because we assumed $\jmath^{-1}(K_f)=K'_f$, we can write $k=\jmath(k')$ for $k'\in K'_f$ and $\jmath({x'_1}^{-1}\gamma_0 x'_2 k')=1$. Thus, ${x'_1}^{-1}\gamma_0 x'_2 k'=z k'_0$ for some $z$ in the center of $H'(\mathbb{Q})$ and $k'_0\in K'_f$. Because the center of $H'(\mathbb{Q})$ is contained in the center of $H'(\mathbb{A})$, we deduce $x'_1=z^{-1}\gamma_0 x'_2 k' {k'_0}^{-1}\in H'(\mathbb{Q}) x'_2 K'_f$ as required.

We have established that $\jmath\colon [H']_{K'_f}\to [H]_{K_f}$ is a continuous  bijection. It is a homeomorphism because it is also a smooth function between two real manifolds with everywhere non-vanishing differential.  The probability Haar measure on $[H]_{K_f}$ is the unique $H(\mathbb{R})$-invariant Borel probability measure that gives equal mass to each of the finitely many $H(\mathbb{R})$-orbits. The same holds for $[H']_{K'_f}$ and $H'(\mathbb{R})$. Using this characterization, it is easy to check that the push-forward of the probability Haar measure from $[H]_{K_f}$ to $[H']_{K'_f}$ under $\jmath^{-1}$ is the probability Haar measure.
\end{proof}

\subsection{The Theta Transfer}
We now revert to our setting of interest, as described at the start of \S\ref{sec:appendix-theta}.  Recall that $R\subset B$ is an Eichler order. For a finite rational place $v$, we define $R_v=R\otimes_{\mathbb{Z}} \mathbb{Z}_v$ and $\widetilde{K}_{R_v}=R_v^\times$. Define $K_{R_v}$ to be the image of $\tilde{K}_{R_v}$ under the map $B_v^\times\to G(\mathbb{Q}_v)$. Finally, set $\widetilde{K}_R=\prod_{v<\infty} \widetilde{K}_{R_v}$, $K_R=\prod_{v<\infty} K_{R_v}$, and let $K_M$ denote the pre-image of $K_R\times K_R$ under $\iota^{(1)}$. Then, $K_R$ is a compact and open subgroup of $G(\mathbb{A}_f)$. We also assume that $R$ is of squarefree level $N$ (see \S\ref{sec:statements-setup}).

We verify now that the hypotheses of Lemma \ref{lem:isogeny2isomorphism} hold for both of the maps
\begin{equation*}
  G' \times G' \xrightarrow{\iota '} M
  \xrightarrow{\iota^{(1)}} G \times G.
\end{equation*}
Indeed:
\begin{enumerate}
\item We have $\det\left(\frac{\bullet}{\bullet}\right)(K_R\times K_R)=\widehat{\mathbb{Z}}^\times$ and $\det\left(K_M\right)=
\widehat{\mathbb{Z}}^\times$. Because $\dfaktor{\mathbb{Q}^\times\mathbb{A}^{\times 2}}{\mathbb{A}^\times}{\widehat{\mathbb{Z}}^\times}\cong 1$ the equality $\faktor{(G\times G)_\mathrm{char}}{(K_R\times K_R)}=\faktor{M_\mathrm{char}}{K_M}=1$ holds.
\item This condition is easy to verify by applying the maps $\det\left(\frac{\bullet}{\bullet}\right)$ and $\det$, and the fact that a rational number is a square if and only if it is positive and has even valuation at each finite place.
\item Consider the case $\jmath=\iota^{(1)}$. The last condition can be checked locally at each finite place to see that $(I,-I)_v\in (K_{R_v}\times K_{R_v})^{(1)}$ for all $v<\infty$. At the archimedean place we use the diagonal embedding of $(I,-I)$ in $M(\mathbb{A})$ to arrive at $(I,-I)_\infty\in Z_{M}(\mathbb{Q})\cdot K_M$. The argument for $\jmath=\iota'$ follows mutatis mutandis.
\end{enumerate}
Lemma \ref{lem:isogeny2isomorphism} now implies that the following maps are measure preserving homeomorphisms.
\begin{equation}\label{eq:homog-iso}
{\left[G'\times G'\right]}_{K_R^1\times K_R^1}\xrightarrow{\iota'}
{\left[M\right]}_{K_M}\xrightarrow{\iota^{(1)}}
{\left[G\times G\right]}_{K_R\times K_R}
\end{equation}
We get isomorphisms of Hilbert spaces
\begin{equation}\label{eq:L2-iso}
L^2\left(\left[G'\times G'\right]\right)^{K_R^1\times K_R^1}\xleftarrow[\sim]{{\iota'}^*}
L^2\left(\left[M\right]\right)^{K_M}\xleftarrow[\sim]{{\iota^{(1)}}^*}
L^2\left(\left[G\times G\right]\right)^{K_R\times K_R}.
\end{equation}
By Lemma \ref{lem:pull-back-properties}, these restrict to isomorphisms of the respective spaces of cusp forms.

Set $\tilde{U_0}(p^n)=\sm
\mathbb{Z}_p & \mathbb{Z}_p \\ p^n \mathbb{Z}_p & \mathbb{Z}_p
\esm \cap \operatorname{GL}_2(\mathbb{Z}_p)$ -- a compact and open subgroup of $\operatorname{GL}_2(\mathbb{Q}_p)$.
Let $\tilde{U}_R=\prod_{p}{\tilde{U}_p}<\operatorname{GL}_2(\mathbb{A}_f)$ to be defined by $\tilde{U}_p=\tilde{U_0}(1)=\operatorname{GL}_2(\mathbb{Z}_p)$ for all primes $p$ where $G$ is unramified and $R_p$ is maximal. If $G$ ramifies at $p$ or $R_p$ is not maximal, then define $\tilde{U}_p=\tilde{U}_0(p)$. Note that we assume that $R$ has squarefree level. Finally, set $U^1_R=\tilde{U}_R\cap\operatorname{SL}_2(\mathbb{A}_f)$ and let $U_R$ be the projection of $\tilde{U}_R$ to $\operatorname{PGL}_2(\mathbb{A}_f)$.
Similarly to the previous discussion, the natural map $\iota_0\colon[\operatorname{SL}_2]_{U_R^1}\to[\operatorname{PGL}_2]_{U_R}$ is a homeomorphism that sends the probability Haar measure on the left-hand side to the probability Haar measure on the right-hand side. This induces an isomorphism of Hilbert spaces
\begin{equation}\label{L2-iso-symp}
L^2([\operatorname{SL}_2])^{U_R^1}\xleftarrow[\sim]{\iota_0^*} L^2([\operatorname{PGL}_2])^{U_R},
\end{equation}
that descends to an isomorphism of the cuspidal subspaces.

Recall that $B_p=B\otimes \mathbb{Q}_p$. We also denote $B_\infty\coloneqq B\otimes \mathbb{R}$.
We denote by $\rho$ the Weil representation of the reductive dual pair $(\operatorname{O}_{\det},\operatorname{SL}_2)$ associated to the quadratic space $(B, \det)$. We refrain at the moment from specifying the exact space of test functions on $B_{\mathbb{A}}\coloneqq B_\infty\times \prod'_p B_p$ on which we let $\rho$ act.
If ${\Phi}\colon B_{\mathbb{A}}\to\mathbb{C}$ is a test function, then the group $M(\mathbb{A})\cong \operatorname{SO}_{\det}$ acts by determinant preserving transformations, $\rho(l,r;e).{\Phi}(x)={\Phi}(l^{-1}xr)$ and the action of the group $\operatorname{SL}_2(\mathbb{A})$ is described in \cite{MR0165033,MR0333081}. Specifically, the definition of the $\operatorname{SL}_2(\mathbb{A})$-action depends on a global character $\psi\colon \lfaktor{\mathbb{Q}}{\mathbb{A}}\to \mathbb{C}$. We fix $\psi=\prod_v \psi_v$ with $\psi_v$ everywhere unramified and $\psi_\infty(x)=\exp(2\pi i x)$.

Let ${\Phi}=\prod_v {\Phi}_v \colon B_{\mathbb{A}}\to\mathbb{C}$  be a test function with  ${\Phi}_\infty\colon B_\infty\to\mathbb{C}$ the Bergman test function from \cite[\S6]{Theta-supnorm-Is} or a Schwartz function. Assume for $v<\infty$ that ${\Phi}_v$ is Schwartz--Bruhat and that ${\Phi}_v=\mathds{1}_{R_v}$ for almost all $v$.
If ${\Phi}_\infty$ is the Bergman test function we let $\rho$ act on the space of functions defined in \cite[\S3]{Theta-supnorm-Is}, otherwise we let $\rho$ act on the space of Schwartz--Bruhat functions as usual.
The theta kernel associated to ${\Phi}$ is the function $\Theta_{\Phi}\colon M(\mathbb{A})\times \operatorname{SL}_2(\mathbb{A})\to\mathbb{C}$ defined by
\begin{equation}\label{eq:theta-kernel-def}
\Theta_{\Phi}(l,r;s)=\sum_{\xi\in B} (\rho(l,r;s).{\Phi})(\xi).
\end{equation}
The series defining $\Theta_{\Phi}(l,r;s)$ is absolutely convergent, \cite[\S3.6]{Theta-supnorm-Is}, and is of moderate growth on $M(\mathbb{A})\times \operatorname{SL}_2(\mathbb{A})$, \cite{RallisSchiffmann}. Moreover, it is $M(\mathbb{Q})\times \operatorname{SL}_2(\mathbb{Q})$ invariant on the left, cf.\ \cite{MR0165033}, \cite[Proposition 1]{MR0333081}, \cite[\S3.6]{Theta-supnorm-Is}.
\begin{definition}
Let $\varphi,\varphi'\in L^2_{\mathrm{cusp}}\left([G]\right)^\infty$ and $\varphi^*\in L^2_{\mathrm{cusp}}\left([\operatorname{SL}_2]\right)^\infty$.  Fix a test function ${\Phi}$ as above. Then, the theta transfer of $\varphi\otimes \varphi'$ and $\varphi^*$ relative to ${\Phi}$ is defined by
\begin{align*}
(\varphi\otimes \varphi')_{\Phi}(s)&=\int_{[M]} \Theta_{\Phi}(l,r;s) \varphi(l) \varphi'(r)\dif(l,r)\\
&= \int_{[M]} \Theta_{\Phi}(l,r;s) {\iota^{(1)}}^*(\varphi \otimes \varphi')(l,r)\dif(l,r),\\
{\varphi^*}^{\Phi}(l,r)&= \int_{[\operatorname{SL}_2]} \Theta_{\Phi}(l,r;s) \varphi^*(s) \dif s.
\end{align*}
The former is a complex-valued function on $\operatorname{SL}_2(\mathbb{A})$, and the latter is a function on $M(\mathbb{A})$.
Both integrals converge absolutely because $\Theta_{\Phi}$ is of moderate growth and $\varphi$,$\varphi'$, $\varphi^*$ are of rapid decay. By abuse of notation, we will also denote
\begin{equation*}
\varphi_{\Phi}=(\varphi\otimes\overline{\varphi})_{\Phi}.
\end{equation*}
Note also that the modularity of the theta kernel $\Theta_{\Phi}$ implies that $(\varphi\otimes \varphi')_{\Phi}$ is left $\operatorname{SL}_2(\mathbb{Q})$-invariant, and ${\varphi^*}^{\Phi}$ is left $M(\mathbb{Q})$-invariant.
\end{definition}
If, moreover, $\varphi$ is $K_R$-invariant, and ${\Phi}$ is both left and right $\widetilde{K}_R$-invariant then, because the maps in \eqref{eq:homog-iso} are isomorphisms of measure spaces, we have
\begin{equation}\label{eq:theta-lift-G'}
\varphi_{\Phi}(s) = \int_{[G']}\int_{[G']} \Theta_{\Phi}(l,r;s) \varphi(l)\overline{\varphi(r)}\dif l \dif r.
\end{equation}

We will need the following lemma. It is mostly a corollary of \cite{RallisHoweDuality,MoeglinTheta,KudlaRallis}.
\begin{lemma}\label{lem:theta-cuspidality}
Let $\varphi,\varphi'\in L^2_{\mathrm{cusp}}\left([G]\right)^\infty$ and $\varphi^*\in L^2_{\mathrm{cusp}}\left([\operatorname{SL}_2]\right)^\infty$. Assume that $\varphi,\varphi'$ are $K_R$-invariant, and that ${\Phi}$ is both left and right $\widetilde{K}_R$-invariant. Then, $(\varphi\otimes\varphi')_{\Phi}\in L^2_{\mathrm{cusp}}\left([\operatorname{SL}_2]\right)$ and ${\varphi^*}^{\Phi}\in L^2_{\mathrm{cusp}}\left([M]\right)$.
\end{lemma}
\begin{proof}
The fact that ${\varphi^*}^{\Phi}$ is square-integrable and cuspidal is trivial whenever $G$ is anisotropic. If $G\cong \operatorname{PGL}_2$ is split, then this follows from Rallis' tower property \cite{RallisHoweDuality,MoeglinTheta} and the fact that the theta transfer of any cuspidal automorphic representation of $\operatorname{SL}_2\cong\operatorname{Sp}_2$ to the orthogonal group of the hyperbolic plane $\operatorname{O}(1,1)$ vanishes\footnote{This is simple to deduce from the fact that a theta series arising from the $2$-dimensional isotropic quadratic form is a pseudo-Eisenstein series, see, e.g., \cite{nelson-theta-squared}.}.  

That the lift of $\varphi\otimes\varphi'$ is cuspidal follows similarly, except that we need to use the theta transfer from $\operatorname{O}_{\det}$ to $\operatorname{SL}_2$, i.e.\@ we need first to lift ${\iota^{(1)}}^*\varphi\otimes\varphi'$ to $\operatorname{O}_{\det}$. For that purpose we use the homomorphism $\iota\colon M\to \operatorname{O}_{\det}$, which is the composition of the isomorphism $M\cong \operatorname{SO}_{\det}$ with the embedding $\operatorname{SO}_{\det}\hookrightarrow \operatorname{O}_{\det}$. This map satisfies the assumptions of \S\ref{sec:automorphic-restriction}. For every finite place $v$, let $\operatorname{O}_{\det}(R_v)$ to be the group of orthogonal transformations of $B_v$ that send $R_v$ to itself and define $\operatorname{O}_{\det}(\hat{R})=\prod_{v<\infty} \operatorname{O}_{\det}(R_v)$. Then, the conditions of Lemma \ref{lem:isogeny2isomorphism} are easily verified and we deduce that the pull-back $\iota^*$ induces an isomorphism of $L^2([\operatorname{O}_{\det}])^{\operatorname{O}_{\det}(\hat{R})}$ and  $L^2([M])^{K_M}$. We deduce from Lemma \ref{lem:pull-back-properties} that $(\iota^{*})^{-1}{\iota^{(1)}}^* \varphi\otimes \varphi'$ is cuspidal and 
\begin{equation}\label{eq:orthogonal-theta-lift}
(\varphi\otimes \varphi')_{\Phi}(s)= \int_{[\operatorname{O}_{\det}]} \Theta_{\Phi}(\bullet ; s)
(\iota^{*})^{-1}{\iota^{(1)}}^* \varphi\otimes \varphi'  \dif m_{\operatorname{O}_{\det}}.
\end{equation}
Here, we have extended the definition of $\Theta_{\Phi}$ in \eqref{eq:theta-kernel-def} to $\operatorname{O}_{\det}(\mathbb{A})\times \operatorname{SL}_2(\mathbb{A})$ in the obvious way.
The integral in \eqref{eq:orthogonal-theta-lift} is a theta lift of a cuspidal function in $L^2([\operatorname{O}_{\det}])^\infty$ to $L^2([\operatorname{SL}_2])^\infty$. In this case \cite{RallisHoweDuality} verifies that the theta lift of a cuspidal function to $\operatorname{SL}_2$ is cuspidal.
\end{proof}

\begin{lemma}\label{lem:theta-self-adjoint}
Let $\varphi^*\in L^2_{\mathrm{cusp}}([\operatorname{SL}_2])^\infty$ and $\varphi,\varphi' \in L^2_{\mathrm{cusp}}([G])^\infty$. Assume that $\varphi,\varphi'$ are $K_R$-invariant, and that ${\Phi}$ is both left and right $\widetilde{K}_R$-invariant.  Then
\begin{equation*}
\left\langle {\varphi^*}^{\Phi}, {\iota^{(1)}}^*\left(\varphi\otimes\varphi'\right) \right\rangle
=
\left\langle \varphi^*, \left(\varphi\otimes\varphi'\right)_{\overline{{\Phi}}} \right\rangle.
\end{equation*}
\end{lemma}
\begin{proof}
This follows from Fubini and the fact that cusp forms are of rapid decay.
\end{proof}

\begin{proposition}\label{prop:whittaker-unfolding}
Assume $\varphi,\varphi'\in L^2_\mathrm{cusp}([G])^{\infty}$ are $K_R$-invariant, and that ${\Phi}$ is both left and right $\widetilde{K}_R$-invariant.  For $s\in\operatorname{GL}_2(\mathbb{A})$ define $s_1=\sm (\det s)^{-1} & 0 \\ 0 & 1 \esm s$. 
The Whittaker function of the theta lift satisfies
\begin{equation*}
W_{{\iota_0^*}^{-1}(\varphi\otimes \varphi')_{\Phi}}(s)= |\det s |_{\mathbb{A}} \langle T_s^{\Phi} \varphi, \overline{\varphi'}(\bullet \alpha_{\det s}) \rangle_{[G']},
\end{equation*}
where
\begin{equation*}
\left(T^{\Phi}_s \varphi\right)(r)=
\begin{cases}
\int_{G'(\mathbb{A})} \varphi(rl^{-1}) \left(\rho\left(s_1\right).{\Phi}\right)(l \alpha_{\det s})  \dif l & \det s \in \det B_{\mathbb{A}} \\ 
0 & \det s \not \in \det B_{\mathbb{A}}
\end{cases}\
\end{equation*}
and $\alpha_{\det s}\in B_{\mathbb{A}}$ is any element satisfying $\det \alpha_{\det s}=\det s$.
Moreover, we can replace the inner product in $L^2([G'])$ in the formula above by an inner product in $L^2([G])$.
\end{proposition}
\
\begin{proof}
First, observe
\begin{equation*}
  {\iota_0^*}^{-1}(\varphi\otimes \varphi')_{\Phi}(s) = |\det s |_{\mathbb{A}}  \int_{[G']}\int_{[G']} \sum_{\xi \in B} (\rho(s_1).{\Phi})(l^{-1}\xi r \alpha_{\det s}) \varphi(l)\varphi'(r \alpha_{\det s}) \dif l \dif r.
\end{equation*}
Consider both sides of the first equality as functions on $\operatorname{GL}_2(\mathbb{A})^\dagger=\{x\in\operatorname{GL}_2(\mathbb{A})\colon \det x\in \det B_{\mathbb{A}}\}$. Set $\operatorname{GL}_2(\mathbb{Q})^\dagger=\{x\in\operatorname{GL}_2(\mathbb{Q})\colon \det x\in \det B\}$ and note that $\lfaktor{\operatorname{GL}_2(\mathbb{Q})^\dagger}{\operatorname{GL}_2(\mathbb{A})^\dagger}\cong \lfaktor{\operatorname{GL}_2(\mathbb{Q})}{\operatorname{GL}_2(\mathbb{A})}$ because ${\operatorname{GL}_2(\mathbb{Q})}{\operatorname{GL}_2(\mathbb{A})^\dagger}=\operatorname{GL}_2(\mathbb{A})$.  The first equality then follows from \eqref{eq:theta-lift-G'} by noticing that both sides are $\operatorname{GL}_2(\mathbb{Q})^\dagger$-invariant on the left, $Z_{\operatorname{GL}_2}(\mathbb{A})$-invariant, $U_R$-invariant on the right and coincide on $\operatorname{SL}_2(\mathbb{A})$.  A standard unfolding argument in the $l$ variable, see \cite{MR0333081,Theta-supnorm-Is}, applied to the last expression shows for $\det s \in \det B_{\mathbb{A}}$
\begin{equation*}
W_{{\iota_0^*}^{-1}(\varphi\otimes \varphi')_{\Phi}}(s)=|\det s |_{\mathbb{A}}
\int_{[G']} \int_{G'(\mathbb{A})}  \left(\rho\left(s_1\right).{\Phi}\right)(l^{-1} r \alpha_{\det s}) \varphi(l) \varphi'(r \alpha_{\det s}) \dif l \dif r,
\end{equation*}
and $W_{{\iota_0^*}^{-1}(\varphi\otimes \varphi')_{\Phi}}(s)=0$ if $\det s \not \in \det B_{\mathbb{A}}$. Using the change of variables $l^{-1}r\mapsto l$, we can write
\begin{equation*}
T_s^{\Phi}\varphi(r)=\int_{G'(\mathbb{A})} \left(\rho\left(s_1\right).{\Phi}\right)(l^{-1} r \alpha_{\det s}) \varphi(l) \dif l 
=\int_{G'(\mathbb{A})} \varphi(r l^{-1}) \left(\rho\left(s_1\right).{\Phi}\right)(l \alpha_{\det s})  \dif l.
\end{equation*}
This establishes the first formula.

The last formula extends naturally to any $r\in G(\mathbb{A})$ and the result is left $G(\mathbb{Q})$-invariant. If $\det s\not \in \det G(\mathbb{A})$ then we extend $T^{\Phi}_s \varphi$ by zero to $G(\mathbb{A})$.
For any $k\in K_R$, using the invariance properties of ${\Phi}$ and $\varphi$, we can apply the change of variables $k l \alpha_{\det s}^{-1} k^{-1} \alpha_{\det s} \mapsto l$ to see that $T_s^{\Phi}\varphi$ is right $\alpha_{\det s}^{-1} K_R \alpha_{\det s}$-invariant. The same holds for $\varphi'(\bullet \alpha_{\det s})$. Because the groups $\alpha_{\det s} K_R \alpha_{\det s}^{-1}$, $\alpha_{\det s} K_R^1 \alpha_{\det s}^{-1}$ and the isogeny $G'\to G$ satisfy the assumptions of Lemma \ref{lem:isogeny2isomorphism}, we see that we can replace the inner product in $[G']$ by an inner product in $[G]$.
\end{proof}
	
\begin{corollary}\label{cor:disjoint-transfer-vanishes}
Assume ${\Phi}$ is both left and right $\widetilde{K}_R$-invariant. Let $\varphi,\varphi'\in L^2_\mathrm{cusp}([G])^{\infty}$ be $K_R$-invariant, and denote by $\pi$ and $\pi'$ the cuspidal automorphical representations generated by $\varphi$ and $\varphi'$ respectively. If $\pi$ is disjoint from $\pi'^\vee$, then $(\varphi\otimes\varphi')_{\Phi}=0$.
\end{corollary}
\begin{proof}
In this case, we see that ${\iota_0^*}^{-1}(\varphi\otimes\varphi')_{\Phi}$ is cuspidal with a vanishing Whittaker function.
\end{proof}

\begin{corollary}\label{cor:Whittaker-decomposable}
Assume ${\Phi}$ is invariant under the conjugation action of $K_R$. Let $\pi\subset  L^2_\mathrm{cusp}([G])^{\infty}$ be an irreducible representation.
Assume $\varphi,\overline{\varphi'}\in \pi$ are $K_R$-invariant decomposable vectors, i.e. $\varphi,\overline{\varphi'}\mapsto \otimes \varphi_v, \otimes \overline{\varphi'_v}$ in $\pi\cong\bigotimes'\pi_v$. Then
\begin{equation*}
W_{{\iota_0^*}^{-1}(\varphi\otimes \varphi')_{\Phi}}(s)=V^{-1} |\det s|_{\mathbb{A}}\prod_v  \left\langle  \varphi_v \star_{G'(\mathbb{Q}_v)} \left(\rho(s_{v,1}).{\Phi}_v\right)(\bullet \alpha_{\det s, v}) , \pi_v(\alpha_{\det s,v}).\overline{\varphi'_v} \right\rangle,
\end{equation*}
where $V$ is the volume of the (possibly disconnected) real manifold $\dfaktor{G'(\mathbb{Q})}{G'(\mathbb{A})}{K_f^1}$ with respect to the Haar measure of $G'(\mathbb{R})$, see \S\ref{sec:unadelic-volume}, and we normalize the Haar measure on $G'(\mathbb{Q}_p)$ so that $R_p^1$ has unit volume.
\end{corollary}
\begin{proof}
This follows directly from Proposition \ref{prop:whittaker-unfolding}. The constant $V^{-1}$ arises as a measure normalization constant. Specifically, denote by $m_{G'(\mathbb{Q}_p)}$ the Haar measure on $G'(\mathbb{Q}_p)$ satisfying $m_{G'(\mathbb{Q})_p}(R_p^1)=1$. The Haar measure on $G'(\mathbb{A})$ satisfies $m_{G'}=c \bigotimes_v m_{G'(\mathbb{Q}_v)}$ with some measure normalization constant $c>0$. Specifically, this equality holds for linear combinations of standard test functions $\prod_v f_v$ with $f_p=\mathds{1}_{R_p^1}$ for a.e.\@ $p$. To compute $c$ we write $\dfaktor{G'(\mathbb{Q})}{G'(\mathbb{A})}{G'(\mathbb{R})K_f^1}=\left\{\delta_1,\ldots,\delta_h \right\}$ and denote by $\mathcal{F}_i\subset G'(\mathbb{R})$ a fundamental domain for the right  action of $\Gamma_i=G(\mathbb{Q}) \cap \delta_i K_R^1 \delta_i^{-1}$ on $G'(\mathbb{R})$. Then $\bigsqcup_{i=1}^h \delta_i  \mathcal{F}_i K_R^1\subset G(\mathbb{A})$ is a fundamental domain for the left action of $G(\mathbb{Q})$ on $G(\mathbb{A})$, and we deduce
\begin{equation*}
1=m_G\left(\bigsqcup_{i=1}^h \delta_i  \mathcal{F}_i K_R^1\right)=c\sum_{i=1}^h m_{G'(\mathbb{R})}\left(\mathcal{F}_i\right) =cV.
\end{equation*}
\end{proof}

\begin{lemma}\label{lem:Whittaker1=trace}
Fix ${\Phi}_v=\mathds{1}_{R_v}$ for all finite places $v<\infty$ and let ${\Phi}_\infty$ be a Schwartz function or the Bergman test function from \cite{Theta-supnorm-Is}. Fix $s=(\iota_0(s_\infty),u_f)$ with $s_\infty\in \operatorname{SL}_2(\mathbb{R})$ and $u_f\in U_R$. Assume that $\varphi\in L^2_\mathrm{cusp}([G])^\infty$ has weight $m$ and is a $K_R$-invariant newvector in an irreducible cuspidal automorphic representation $\pi$. If $\rho(\bullet, \bullet ;s){\Phi}_\infty$ is $K_\infty\times K_\infty$-isotypical of weight $(-m,m)$, then
\begin{align*}
W_{{\iota_0^*}^{-1}\varphi_{\Phi}}(s)&=\frac{\|\varphi\|_2^2}{V}   \Tr \left(\Res^{G(\mathbb{R})}_{G'(\mathbb{R})}\pi_\infty\right)\left({\rho(s_\infty).{\Phi}_\infty}\restriction_{G'(\mathbb{R})}\right)\\
&=\frac{\|\varphi\|_2^2}{V} \overline{\left\langle f_{\varphi_\infty, \varphi_\infty} , (\rho(s_\infty).{\Phi}_\infty)  \right\rangle}_{G'(\mathbb{R})},
\end{align*}
where $f_{\varphi_\infty, \varphi_\infty}(g)=\langle \pi(g).\varphi_\infty, \varphi_\infty\rangle$ is the matrix coefficient attached to the archimedean component of $\varphi$ in $\bigotimes'_v \pi_v$, normalized so that $\|\varphi_\infty\|_2=1$.
\end{lemma}

\begin{proof}
It is sufficient to establish the claim when $\|\varphi\|_2=1$.
By  \cite[\S 4]{Theta-supnorm-Is}, the theta transfer $\varphi_{\Phi}$ is $U_R^1$-invariant, thus ${\iota_0^*}^{-1}\varphi_{\Phi}$ is $U_R$-invariant and we can assume without loss of generality that $u_f=e$. Then $\det s=1$ and we take $\alpha_{\det s}=e$.

The newvector $\varphi$ decomposes as $\varphi\mapsto\otimes \varphi_v$ in $\pi\cong\bigotimes' \pi_v$. We normalize $\varphi_\infty$ to have norm $1$, then $\prod_p \|\varphi_p\|_2=1$ as well. We also normalize the measure on $G'(\mathbb{Q}_p)$ so that $R_p^1=K_p^1$ has unit volume. Corollary  \ref{cor:Whittaker-decomposable} now implies $W_{{\iota_0^*}^{-1}\varphi_{\Phi}}(s)= V^{-1} \prod_v\left\langle  \varphi_v\star_{G'(\mathbb{Q}_v)} \left(\rho(s_v).{\Phi}_v\right),\varphi_v \right\rangle$.
For a finite place $p$, we have $s_p=e$ and $\varphi_p$ is $K_p$-invariant, hence $\varphi_p\star \mathds{1}_{K_p^1}=\varphi_p$. We conclude 
\begin{equation*}
W_{{\iota_0^*}^{-1}\varphi_{\Phi}}(s)=V^{-1} \langle \varphi_\infty \star (\rho(s_\infty).{\Phi}_\infty)\restriction_{G'(\mathbb{R})}, \varphi_\infty \rangle.
\end{equation*}
This expression coincides with the trace if the convolution operator $\star_{G'(\mathbb{R})}(\rho(s_\infty).{\Phi}_\infty)$ annihilates the orthogonal complement to $\varphi_\infty$ in $\pi_\infty$.  This follows from the facts that every $K_\infty$-isotypical component of the admissible unitary representation $\pi_\infty$ is at most $1$-dimensional and the assumption that $\rho(\bullet, \bullet ;s){\Phi}_\infty$ is $K_\infty\times K_\infty$-isotypical.

To show the formula in terms of a matrix coefficient we use Fubini to deduce
\begin{align*}
\langle \varphi_\infty \star (\rho(s_\infty).{\Phi}_\infty)\restriction_{G'(\mathbb{R})}, \varphi_\infty \rangle &=\int_{G'(\mathbb{R})} \int_{G(\mathbb{R})} \varphi_\infty(g h^{-1}) (\rho(s_\infty).{\Phi}_\infty)(h) \overline{\varphi_\infty (g)} \dif h \dif g\\
&=\int_{G'(\mathbb{R})} \int_{G(\mathbb{R})} \varphi_\infty(g) \overline{\varphi_\infty (gh)} (\rho(s_\infty).{\Phi}_\infty)(h)  \dif g \dif h\\
 &=\left\langle \overline{f_{\varphi_\infty,\varphi_\infty}} , \overline{\rho(s_\infty).{\Phi}_\infty} \right\rangle_{G'(\mathbb{R})}.
\end{align*}
The conditions of Fubini's theorem are satisfied because the test function ${\Phi}_\infty\restriction_{G'(\mathbb{R})}$  is in $L^q(G'(\mathbb{R}))$ for all $q\geq 1$ and  $f_{\varphi_\infty,\varphi_\infty}\in L^p(G'(\mathbb{R}))$ for some $p\geq 2$.
\end{proof}

\begin{proposition}\label{prop:theta-JL}
Fix ${\Phi}_v=\mathds{1}_{R_v}$ for all finite places $v<\infty$. 
Let $\pi\subset L^2_\mathrm{cusp}([G])^\infty$ be an irreducible cuspidal automorphic representation, and denote by $\pi^{\mathrm{JL}}$ its Jacquet--Langlands transfer to $L^2_{\mathrm{cusp}}([\operatorname{PGL}_2])^\infty$.  In case $G$ is split, we define $\pi^{\mathrm{JL}}=\pi$. 
Assume $\varphi\in\pi$, $\varphi'\in \pi^\vee$ are non-vanishing $K_R$-invariant vectors, then ${\iota_0^*}^{-1}(\varphi\otimes\varphi')_{\Phi}\in\pi^{\mathrm{JL}}$.

Moreover, if ${\Phi}_\infty$ is $\rho\left(K_\infty,K_\infty;\operatorname{SO}_2(\mathbb{R})\right)$-isotypical with weight $(-m,m,\kappa)$, $\pi$ has conductor $K_R$ and $\varphi$ is a newvector of weight $m$, then either $\varphi_{\Phi}$ vanishes or ${\iota_0^*}^{-1}\varphi_{\Phi}$ is a newvector of weight $\kappa$ in the Jacquet--Langlands transfer $\pi^{\mathrm{JL}}$.
\end{proposition}
\begin{proof}
Any smooth vector in $\pi^{K_R}$ is a linear combination of $K_R$-invariant factorizable vectors in the representation $\pi\cong \bigotimes'_v \pi_v$. Thus, we assume without loss of generality that $\varphi$ and $\varphi'$ are factorizable in $\pi$ and $\pi^\vee$ respectively.

The function ${\iota_0^*}^{-1}(\varphi\otimes \varphi')_{\Phi}$ is cuspidal  by Lemma \ref{lem:theta-cuspidality}. Because $\varphi\mapsto \otimes_v \varphi_v$ and $\varphi'\mapsto \otimes_v \varphi'_v$ are factorizable, Corollary \ref{cor:Whittaker-decomposable} implies that the Whittaker function of ${\iota_0^*}^{-1}(\varphi\otimes \varphi')_{\Phi}$ decomposes into a product.

Assume $(\varphi\otimes\varphi')_{\Phi}$ does not vanish. Let $S$ be a finite set of rational places containing the archimedean place, all places where $G$ ramifies, all places where $\pi$ ramifies and all places where either $\varphi_v$ or $\varphi'_v$ is not spherical. Shimizu \cite{MR0333081} computes the local Whittaker function $|\det s|_v\langle   \varphi_v\star_{G'(\mathbb{Q}_v)} {\Phi}_v(\bullet \alpha_{\det s,v}), \pi_v(\alpha_{\det s ,v}).\overline{\varphi'_v} \rangle_{\pi_v}$ for every place $v\not\in S$ and it coincides with the Whittaker function of a spherical newvector in $\pi^{\mathrm{JL}}_v$. Hence, every irreducible component $\sigma\cong \bigotimes'_v \sigma_v$ of the representation generated by ${\iota_0^*}^{-1}(\varphi\otimes\varphi')_{\Phi}$ satisfies $\sigma_v\cong \pi^{\mathrm{JL}}_v$ for all $v\not\in S$. Using the strong multiplicity one property for $\operatorname{PGL}_2$, we deduce that $\sigma=\pi^{\mathrm{JL}}$ for every irreducible component $\sigma$ as above. Hence, the representation generated by ${\iota_0^*}^{-1}(\varphi\otimes \varphi')_{\Phi}$ is $\pi^{\mathrm{JL}}$.

Assume next that $\pi$ has conductor $K_R$ and that $\varphi$ is a newvector of weight $m$. Then, the assumption that ${\Phi}_\infty$ has weight $(-m,m,\kappa)$ implies that $\varphi_{\Phi}$ has  weight $\kappa$. By \cite[\S 4]{Theta-supnorm-Is} the theta transfer $\varphi_{\Phi}$ is $U_R^1$-invariant. Because the conductor of the Jacquet--Langlands transfer is exactly $U_R$ and ${\iota_0^*}^{-1}\varphi_{\Phi}={\iota_0^*}^{-1}(\varphi\otimes \overline{\varphi})_{\Phi}\in\pi^{\mathrm{JL}}$, we deduce that ${\iota_0^*}^{-1}\varphi_{\Phi}$ is a newvector as claimed.
\end{proof}

\begin{corollary}\label{cor:theta-new-orthogonal}
	Assume $\varphi,\varphi'\in L^2_\mathrm{cusp}([G])^\infty$ are $K_R$-invariant and generate disjoint automorphic cuspidal representations, then $\langle \varphi_{\Phi}, \varphi'_{\Phi} \rangle=0$.
\end{corollary}
\begin{proof}
	The Jacquet--Langlands transfers of disjoint automorphic representations are disjoint. Hence, Proposition \ref{prop:theta-JL} above implies that ${\iota_0^*}^{-1}\varphi_{\Phi}$, ${\iota_0^*}^{-1}\varphi'_{\Phi}$ generate mutually orthogonal subrepresentations of $L^2_\mathrm{cusp}([\operatorname{PGL}_2])^\infty$.
\end{proof}

\begin{proposition}\label{prop:theta-double-lift}
Fix ${\Phi}_v=\mathds{1}_{R_v}$ for all finite places $v<\infty$ and assume that ${\Phi}_\infty$ is $\rho\left(K_\infty,K_\infty;\mathbf{SO}_2(\mathbb{R})\right)$-isotypical with weight $(-m,m,\kappa)$. 
Let $\pi\subset L^2_\mathrm{cusp}([G])^\infty$ be an irreducible cuspidal automorphic representation with conductor $K_R$. Assume $0\neq\varphi\in\pi$ is a newvector of weight $m$. Then
\begin{equation}
\left(\varphi_{\Phi}\right)^{\overline{{\Phi}}}=\left(\frac{\|\varphi_{\Phi}\|_2}{\|\varphi\|_2^2}\right)^2{\iota^{(1)}}^* \left(\varphi \otimes \overline{\varphi}\right).
\end{equation}
\end{proposition}
\begin{proof}
Assume that
\begin{equation}\label{eq:phiMM-newvector}
\left(\varphi_{\Phi}\right)^{\overline{{\Phi}}}=\alpha {\iota^{(1)}}^* \left(\varphi \otimes \overline{\varphi}\right),
\end{equation} 
for some $\alpha\in\mathbb{C}$.
Then, $\alpha\|\varphi\|_2^4=\langle \left(\varphi_{\Phi}\right)^{\overline{{\Phi}}}, {\iota^{(1)}}^* \left(\varphi \otimes \overline{\varphi}\right) \rangle$ and Lemma \ref{lem:theta-self-adjoint} implies that $\alpha=\|\varphi_{\Phi}\|_2^2/\|\varphi\|_2^4$.

Because $\left(\varphi_{\Phi}\right)^{\overline{{\Phi}}}$ is continuous and cuspidal by Lemma \ref{lem:theta-cuspidality}, to establish \eqref{eq:phiMM-newvector} it is enough to show that ${\iota^{(1)}}^{*-1}\left(\varphi_{\Phi}\right)^{\overline{{\Phi}}}$ is orthogonal to the orthogonal complement of $\mathbb{C}\left(\varphi \otimes \overline{\varphi}\right)$ in $L^2_{\mathrm{cusp}}([G\times G])^{K_R\times K_R}$. Both $\left(\varphi \otimes \overline{\varphi}\right)$ and ${\iota^{(1)}}^{*-1}\left(\varphi_{\Phi}\right)^{\overline{{\Phi}}}$ transform with weight $(m,-m)$ under $K_\infty\times K_\infty$. Hence, it is enough to check orthogonality in the $(m,-m)$ isotypical subspace
\begin{equation*}
V_m= L^2_{\mathrm{cusp}}([G\times G])^{((K_\infty,m)\cdot K_R)\times ((K_\infty,-m)\cdot K_R)}.
\end{equation*}
Denote by $V_m^0$ the orthogonal compliment of $\mathbb{C}\left(\varphi \otimes \overline{\varphi}\right)$ in $V_m$.
We can choose an orthonormal basis for $V_m^0$ consisting of vectors $\psi\otimes\psi'$ with $\psi,\overline{\psi'}\in L^2_\mathrm{cusp}([G])^{(K_\infty,m)\cdot K_R}$ and $\psi$, $\psi'$ generate irreducible cuspidal automorphic representations of $G(\mathbb{A})$. Because $\pi$ has conductor $K_R$, 
either the representation generated by $\psi$ is disjoint from $\pi$ or the representation generated by $\psi'$ is disjoint from $\pi^\vee$. Fix $\psi$, $\psi'$ as above. We need to show $\left\langle {\iota^{(1)}}^{*-1}\left(\varphi_{\Phi}\right)^{\Phi}, \psi\otimes \psi'\right\rangle=0$. Denote by $\sigma$, $\sigma'$ the irreducible automorphic representations generated by $\psi$, $\psi'$ respectively. We apply Lemma \ref{lem:theta-self-adjoint} to deduce $\left\langle {\iota^{(1)}}^{*-1}\left(\varphi_{\Phi}\right)^{\overline{{\Phi}}}, \psi\otimes \psi'\right\rangle= \left\langle \varphi_{\Phi}, \left(\psi\otimes \psi'\right)_{\Phi}\right\rangle$. If $\sigma'\neq \sigma^\vee$, then $\left(\psi\otimes \psi'\right)_{\Phi}=0$ by Corollary \ref{cor:disjoint-transfer-vanishes}. If $\sigma'=\sigma^\vee$, then $\sigma$ is disjoint from $\pi$, and $(\iota_0^*)^{-1}\left(\psi\otimes \psi'\right)_{\Phi}\in \sigma^{\mathrm{JL}}$ by Proposition \ref{prop:theta-JL}.  The Jacquet--Langlands transfers of disjoint representations are disjoint. Hence, $\pi^{\mathrm{JL}}\perp \sigma^{\mathrm{JL}}$ and $\left\langle \varphi_{\Phi}, \left(\psi\otimes \psi'\right)_{\Phi}\right\rangle=0$ as claimed.
\end{proof}

\subsection{Explicit Theta Kernels}\label{sec:explicit-theta}
\begin{definition}\label{def:archimedean-test-functions}
We now define the archimedean test functions on $B_\infty$ that give rise to the theta series from \S\ref{sec:theta-functions-definitions}.
\begin{align*}
{\Phi}_\infty^{-,k}(g)&=
{X(g)}^k
e^{-2\pi P(g)},\\
{\Phi}_\infty^{-,\mathrm{hol}}(g)&=\frac{k-1}{4\pi}\begin{cases}
(\det g)^{k-1}\overline{X(g)}^{(-k)} e^{-2\pi \det g} & \det g > 0\\
0 & \det g \le 0
\end{cases},\\
{\Phi}_\infty^{+,m}(g)&=(2m+1)(\det g)^m P_m\left(\frac{|X(g)|^2-u(g)}{\det g}\right)e^{-2\pi \det g},\\
{\Phi}_\infty^{+,\mathrm{hol}}(g)&=(k+1) X(g)^k e^{-2\pi \det g}.
\end{align*}
The first two test functions are defined when $G(\mathbb{R})$ is split and the last two are defined when $G(\mathbb{R})$ is ramified.
\end{definition}

\begin{lemma}\label{lem:Minfty-weight}
Let ${\Phi}_\infty$ be one of the kernels in Definition \ref{def:archimedean-test-functions} above and set $\kappa=k,k,2m+2,k+2$ for the different kernels respectively. Then, $\rho(k_\theta).{\Phi}_\infty=e^{i\kappa \theta} {\Phi}_\infty$ for all $k_\theta\in\operatorname{SO}_2(\mathbb{R})$.
\end{lemma}
\begin{proof}
This is verified by Vign\'eras' method \cite{MR0480352}. 
In all cases under consideration except ${\Phi}_\infty^{-,\mathrm{hol}}$, the test function is Schwartz, hence it is enough to check that ${\Phi}_\infty$ satisfies the PDE in \cite[\S3.3]{Theta-supnorm-Is} and then use Lemma 3.4, \emph{op.\@ cit.\@}. In case ${\Phi}_\infty={\Phi}_\infty^{-,\mathrm{hol}}$, the test function is not Schwartz and a technical argument is required to circumvent this issue. This case is treated \cite[\S6]{Theta-supnorm-Is}. 
We proceed to verify the three other cases.

Recall the notation $x=[a,b,c]+d \in B_{\infty}$ from \S\ref{sec:Kinftynotation}. The Laplace operator with Fourier multiplier $-4\pi^2 \det(x)$ is then given by $\Delta = \frac{1}{4} ( \frac{\partial^2}{\partial a^2} \mp (\frac{\partial^2}{\partial b^2}+\frac{\partial^2}{\partial c^2})+\frac{\partial^2}{\partial d^2})$, where the sign is $-$ if $B$ is indefinite and $+$ otherwise. The differential equation in \S3.3 of \cite{Theta-supnorm-Is} for the test function ${\Phi}_{\infty}$ is equivalent to
\begin{equation}
-\Delta {\Phi}(x) + (2 \pi)^2 \det(x) {\Phi}(x) = 2\pi \kappa {\Phi}(x).
\label{eq:K-isotope-PDE}
\end{equation}
We note that for each of the remaining test functions, we may write ${\Phi}_{\infty}(x) = Q(x) e^{-2 \pi P(x)}$, where $P(x)=a^2+b^2+c^2+d^2$ and $Q$ a harmonic polynomial of homogeneous degree. For the first and last test function this may be seen by a well-known criteria (c.f.\@ \cite[Thm 9.1]{Iw97}) noting that $[i,0,0]+1 \in B_{\infty} \otimes \mathbb{C} $ is an isotropic vector. For the third test function this follows from \cite[p. 405]{LPS-S2-distr-II}. With this in mind, we have $\Delta Q = 0$ and $(a\frac{\partial}{\partial a}+b\frac{\partial}{\partial b}+c\frac{\partial}{\partial c}+d\frac{\partial}{\partial d})Q=\deg(Q)Q$, which allows one to easily verify that ${\Phi}_{\infty}$ satisfies \eqref{eq:K-isotope-PDE} in the remaining cases.

\end{proof}

\begin{proposition}\label{prop:theta-adelic2classic}
Let ${\Phi}={\Phi}_\infty\cdot\prod_{v<\infty} \mathds{1}_{R_v}$, where ${\Phi}_\infty$ is any one of the test functions in Definition \ref{def:archimedean-test-functions} above. Set $\kappa=k,k,2m+2,k+2$ for the different kernels respectively. Denote by $\theta_g$  the matching classical theta function from \S\ref{sec:theta-functions-definitions}. For $\tau\in \operatorname{SL}_2(\mathbb{Q})$, we  denote by $(\tau)_\infty$ the image of $\tau$ in the archimedean coordinate of $\operatorname{SL}_2(\mathbb{A})$.
Then, for every $l\mid d_B N$ and $g\in G(\mathbb{A})$
\begin{equation*}
\Theta_{\Phi}(g,g; (\tau_{\ell})_\infty s_\infty U_R^1)=\frac{\mu(\gcd(\ell,d_B))}{\ell} \theta_{g,\ell}(z) e^{i\kappa \theta},
\end{equation*}
where $\mu$ is the M\"obius function, $s_\infty=\sm y^{1/2} & xy^{-1/2} \\ 0 & y^{-1/2} \esm \sm \cos \theta &\sin \theta \\ -\sin \theta & \cos \theta\esm$ and $z=x+iy$, i.e.\@ $s_\infty.i =z$.  Moreover $\theta_{g,\ell}(z)$ is a $\Gamma_0(d_B N)$-invariant function on $\mathbb{H}$ of moderate growth at the cusps.
\end{proposition}
\begin{proof}
We already know $\Theta_{\Phi}$ has weight $\kappa$ in the archimedean symplectic variable $s_\infty$. Moreover, in \cite[\S 3.5]{Theta-supnorm-Is} it is shown that $\Theta_{\Phi}$ is $U_R^1$ invariant. Denote by $(\tau_{\ell})_f$ the diagonal image of $\tau_{\ell}$ in $\operatorname{SL}_2(\mathbb{A}_f)$. The left $\operatorname{SL}_2(\mathbb{Q})$-invariance of the theta kernel implies $\Theta_{\Phi}(g,g; (\tau_{\ell})_\infty s_\infty)=\Theta_{\Phi}(g,g; s_\infty(\tau_{\ell})_f^{-1})$. For every prime $p\nmid d_B N$, we have $\tau_{\ell}\in\operatorname{SL}_2(\mathbb{Z})\subset\operatorname{SL}(\mathbb{Z}_p)= U_p^1$. If $p\mid l$ then $\tau_{\ell} \equiv w \bmod p$, where $w=\sm
0 & 1 \\ -1 & 0
\esm$, hence $\tau_{\ell} \in w U_p^1$. If $p\mid \frac{d_B N}{l}$ then $\tau_{\ell}\equiv e \bmod p$, hence $\tau_{\ell}\in U_p^1$. Because $\rho(U_p^1).\mathds{1}_{R_p}=\mathds{1}_{R_p}$, we can write
\begin{equation*}
\rho\left(g,g;s_\infty(\tau_{\ell})_f^{-1}\right).{\Phi}=\rho(s_\infty).{\Phi}_\infty(g_\infty^{-1}\bullet g_\infty)\prod_{p\nmid l} \mathds{1}_{g_p R_p g_p^{-1}} \prod_{p\mid l} \rho(w).\mathds{1}_{g_p R_p g_p^{-1}}.
\end{equation*}
The Weil action of $w$ is by the Fourier transform for $p\nmid d_B$ and it is by the negative of the Fourier transform for $p \mid d_B$. Specifically, it is shown in \cite[Section \S4]{Theta-supnorm-Is} that for $p\mid d_B N$
\begin{equation*}
\rho(w).\mathds{1}_{R_p}=\gamma_p p^{-1}\mathds{1}_{R_p^\vee},
\end{equation*}
where $\gamma_p=1$ if $B_p$ is split and $\gamma_p=-1$ if $B_p$ is ramified. We conclude that $$\rho\left(g,g,s_\infty(\tau_{\ell})_f^{-1}\right).{\Phi}=\frac{\mu(\gcd(l, d_B))}{\ell}\rho(s_\infty).{\Phi}_\infty(g_\infty^{-1}\bullet g_\infty)\prod_{p\nmid \ell} \mathds{1}_{g_p R_p g_p^{-1}} \prod_{p\mid \ell} \mathds{1}_{g_p R_p^\vee g_p^{-1}}.$$
Because $\bigcap_{p\nmid \ell} g_p R_p g_p^{-1}\bigcap_{p\mid \ell} g_p R_p^\vee g_p^{-1}=R(\ell;g_f)$, we have for $\xi\in B$ that
\begin{equation*}
\left( \prod_{p\nmid \ell} \mathds{1}_{g_p R_p g_p^{-1}} \prod_{p\mid \ell} \mathds{1}_{g_p R_p^\vee g_p^{-1}} \right)(\xi)=\mathds{1}_{R(\ell;g_f)}(\xi),
\end{equation*}
 and we can write
\begin{align*}
\Theta_{\Phi}(g,g; (\tau_{\ell})_\infty s_\infty U_R^1)&=\frac{\mu(\gcd(l,d_B))}{\ell}\sum_{\xi\in R(\ell;g_f)} \left(\rho(s_\infty).{\Phi}_\infty\right)(g_\infty^{-1}\xi g_\infty)\\
&=\frac{\mu(\gcd(l,d_B))}{\ell}\sum_{x\in R(\ell;g)} \left(\rho(s_\infty).{\Phi}_\infty\right)(x).
\end{align*}
The last equality holds because $g_\infty^{-1} R(\ell;g_f) g_\infty=R(\ell;g)$.
The claim now follows from Lemma \ref{lem:Minfty-weight} above and the formul{\ae} for the Weil action of the diagonal and unipotent subgroups.
The moderate growth of $\theta_{g,\ell}$ now follows from the moderate growth of $\Theta_{\Phi}$ in the symplectic variable $s$. The $\Gamma_0(d_B N)$-modularity of $\theta_{g,\ell}$ follows from the left $\operatorname{SL}_2(\mathbb{Q})$-invariance  and right $U_R^1$-invariance of $\Theta_{\Phi}$ in the symplectic variable, and the fact that $\tau_{\ell}$ normalizes $\Gamma_0(d_B N)$.
\end{proof}

\begin{proposition} Let ${\Phi}={\Phi}_{\infty}\cdot\prod_{v<\infty} \mathds{1}_{R_v}$, with ${\Phi}_{\infty}$ given by any of test functions listed in Definition \ref{def:archimedean-test-functions}. Let $\mathcal{G}$ be any of the families of automorphic forms corresponding to $\Theta_{\Phi}$ according to table \ref{table:choice-theta}. Then, for any $\varphi \in \mathcal{G} \subseteq L^2([G])^{\infty}$, a $K_R$-invariant Hecke eigenform, we have $\varphi_{\Phi}=V^{-1} \varphi^\mathrm{JL}$, where $\varphi^{\mathrm{JL}}$ is the arithmetically normalized Jacquet--Langlands lift of $\varphi$, as defined in \S\ref{sec:JL-lifts}.
\label{prop:theta-lift-all-families}
\end{proposition}
\begin{proof} Let $\kappa$ be the entry of table \ref{table:choice-theta} corresponding to $\mathcal{G}$ and $\Theta_{\Phi}$. Lemma \ref{lem:Minfty-weight} shows that $\varphi_{\Phi}$ is of weight $\kappa$ and Proposition \ref{prop:theta-JL} that ${\iota_0^*}^{-1}\varphi_{\Phi}$ is newvector (or zero) of level $U_R$ of the Jacquet--Langlands transfer $\pi^{\JL}$ of the representation $\pi$ generated by $\varphi$. The subspace of vectors in $\pi^{\JL}$ satisfying these two properties is one-dimensional. This implies that $\varphi_{\Phi}$ is proportional to $\varphi^{\JL}$. In order to find the constant of proportionality $\rho_1$, we compute and compare the Whittaker functions at the identity. The Whittaker function of $\varphi^{\JL}$ is recorded in \S\ref{sec:JL-lifts} and those of $\varphi_{\Phi}$ we shall compute with the aid of Lemma \ref{lem:Whittaker1=trace}.
\paragraph{The case $\mathcal{G}=\mathcal{F}^{-}$, ${\Phi}_{\infty}={\Phi}_{\infty}^{-,0}$, $\kappa=0$} Suppose that $\varphi \in \mathcal{F}^{-}_{\frac{1}{4}+t^2} \subseteq \mathcal{F}^{-}$. Then, the representation $\pi_\infty$ is a principal series representation obtained by normalized induction of the character $\sm
\lambda & \ast \\ 0 & \mu
\esm\mapsto \sgn(\lambda/\mu)^\alpha  |\lambda / \mu|^{it}$ for some $\alpha \in \{0,1\}$. The equality of Whittaker functions yields the following equation for the constant of proportionality $\rho_1$:
\begin{equation*}
2\rho_1 K_{it}(2\pi)=V^{-1} \Tr \left(\Res^{\operatorname{PGL}_2(\mathbb{R})}_{\operatorname{SL}_2(\mathbb{R})} \pi_\infty\right)({\Phi}_\infty^{-,0}\restriction_{G'(\mathbb{R})}).
\end{equation*}
Because ${\Phi}_\infty^{-,0}\restriction_{G'(\mathbb{R})}$ is bi-$K_\infty$-invariant, the trace is the Fourier transform of the Abel--Satake transform of ${\Phi}_\infty^{-,0}\restriction_{G'(\mathbb{R})}$.  Compute first the Abel--Satake transform
\begin{align*}
\mathcal{S}{\Phi}_\infty^{-,0}\restriction_{G'(\mathbb{R})} (\tau)&=e^{\tau/2}\int_{-\infty}^\infty {\Phi}_\infty^{-,0}\left(
\begin{pmatrix}
e^{\tau/2} & 0 \\ 0 & e^{-\tau/2}
\end{pmatrix}
\begin{pmatrix}
1 & n \\ 0 & 1
\end{pmatrix}
\right) \dif n\\
&=
e^{\tau/2}\int_{-\infty}^\infty e^{-\pi (2\cosh \tau+e^\tau n^2)} \dif n=e^{-2\pi \cosh \tau}.
\end{align*}
The trace is proportional to the Fourier transform of $\mathcal{S}{\Phi}_\infty^{-,0}\restriction_{G'(\mathbb{R})} $. Using our measure normalization this becomes
\begin{equation*}
\Tr \left(\Res^{\operatorname{PGL}_2(\mathbb{R})}_{\operatorname{SL}_2(\mathbb{R})} \pi_\infty\right)({\Phi}_\infty^{-,0}\restriction_{G'(\mathbb{R})})=\int_{-\infty}^\infty e^{-2\pi \cosh \tau} e^{i t \tau} \dif\tau=2K_{it}(2\pi).
\end{equation*}
Hence, $\rho_1=V^{-1}$.

\paragraph{The case $\mathcal{G}=\mathcal{F}^{-, \hol}$, ${\Phi}_{\infty}={\Phi}_{\infty}^{-,k}$, $\kappa=k$} The equality of Whittaker functions, yields the following equation for the constant of proportionality $\rho_1$:
\begin{equation*}
\rho_1e^{-2\pi}=V^{-1}\overline{ \left\langle f_{\varphi_\infty, \varphi_\infty} , {{\Phi}_\infty^{-,k}}  \right\rangle}_{G'(\mathbb{R})}.
\end{equation*}
The representation $\pi_\infty$ is the discrete series representation with parameter $k$, and $\varphi_\infty\in \pi_\infty$ is the $L^2$-normalized minimal weight vector. The matrix coefficient in this case is exactly\footnote{This can be computed succinctly using the model of the discrete series as a subrepresentation of $L^2(G(\mathbb{R}))$. } $ f_{\varphi_\infty, \varphi_\infty}(g)=\overline{X(g)}^{(-k)}$ and we compute
\begin{align*}
\rho_1e^{-2\pi}&=V^{-1}\int_{G'(\mathbb{R})} e^{-2\pi P(g)} \dif g =V^{-1}2\pi\int_0^\infty e^{-2\pi \cosh \tau} \sinh \tau \dif \tau\\
&=
V^{-1} 2\pi \int_1^\infty e^{-2\pi \xi} \dif \xi=V^{-1} e^{-2\pi}.
\end{align*}

\paragraph{The remaining cases} Here, we verify\footnote{For $\varphi\in \mathcal{F}^{+}_m$, this can be computed easily from the model of the representation spanned by spherical harmonics, using the identity $P_m(\langle \mathrm{v},\mathrm{v}'\rangle)=\frac{4\pi}{2m+1}\sum_{n=-m}^{m} Y_{mn}(\mathrm{v})\overline{Y_{mn}(\mathrm{v'})}$, and the orthogonality of spherical harmonics.\\
For $\varphi\in \mathcal{F}^{+,\hol}$, this can be computed easily from the model of the representation on the space of homogeneous binary complex polynomials of degree $k$.} that the matrix coefficient satisfies $f_{\varphi_\infty, \varphi_\infty}=e^{2\pi}d_{\pi_{\infty}}^{-1} {\Phi}_\infty$, where $\varphi_{\infty} \in \pi_{\infty}$ is the archimedean component of $\varphi$, $L^2$-normalized, and $d_{\pi_{\infty}}$ the (formal) degree of $\pi_{\infty}$. The equality of Whittaker functions, then yields the following equation for the constant of proportionality $\rho_1$:
\begin{equation*}\begin{aligned}
\rho_1e^{-2\pi}&=V^{-1} \overline{\left\langle f_{\varphi_\infty, \varphi_\infty} , {{\Phi}_\infty}  \right\rangle}_{G'(\mathbb{R})} = V^{-1}e^{-2\pi} d_{\pi_{\infty}} \overline{\left\langle f_{\varphi_\infty, \varphi_\infty}, f_{\varphi_\infty, \varphi_\infty}  \right\rangle}_{G'(\mathbb{R})} \\
&= V^{-1}e^{-2\pi} \|\varphi_\infty\|_2^4=V^{-1}e^{-2\pi},
\end{aligned}\end{equation*}
where we have used Schur--Weyl orthogonality for matrix coefficients.
\end{proof}

\bibliography{refs}{} \bibliographystyle{alpha}

\def\cprime{$'$} \def\cprime{$'$} \def\cprime{$'$} \def\cprime{$'$}
\begin{thebibliography}{BHMM21}

\bibitem[AL70]{MR0268123}
A.~O.~L. Atkin and J.~Lehner.
\newblock Hecke operators on {$\Gamma _{0}(m)$}.
\newblock {\em Math. Ann.}, 185:134--160, 1970.

\bibitem[Ass17]{AssingSup}
Edgar Assing.
\newblock {O}n sup-norm bounds part {I}: ramified {M}aa{\ss} newforms over
  number fields.
\newblock {\em {P}reprint}, 2017.
\newblock {\tt arXiv:1710.00362}.

\bibitem[Ass21]{AssingBeyondNew}
Edgar Assing.
\newblock The sup-norm problem beyond the newform.
\newblock {\em {P}reprint}, 2021.
\newblock {\tt arXiv:2111.01893}.

\bibitem[Ber31]{SergeBernstein1931}
Serge Bernstein.
\newblock Sur les polynomes orthogonaux relatifs {\`a} un segment fini (seconde
  partie).
\newblock {\em Journal de {M}ath{\'e}matiques Pures et {A}ppliqu{\'e}es},
  10:219--286, 1931.

\bibitem[B{\'e}r77]{Berard}
P.~H. B{\'e}rard.
\newblock On the wave equation on a compact {R}iemannian manifold without
  conjugate points.
\newblock {\em {M}ath. {Z}.}, 155(3):249--276, 1977.

\bibitem[BH10]{MR2587342}
Valentin Blomer and Roman Holowinsky.
\newblock Bounding sup-norms of cusp forms of large level.
\newblock {\em Invent. Math.}, 179(3):645--681, 2010.

\bibitem[BHM16]{MR3474814}
Valentin Blomer, Gergely Harcos, and Djordje Mili{\'c}evi{\'c}.
\newblock Bounds for eigenforms on arithmetic hyperbolic {$3$}-manifolds.
\newblock {\em Duke Math. J.}, 165(4):625--659, 2016.

\bibitem[BHMM20]{MR4046009}
Valentin Blomer, Gergely Harcos, P\'{e}ter Maga, and Djordje Mili\'{c}evi\'{c}.
\newblock The sup-norm problem for {$\rm GL(2)$} over number fields.
\newblock {\em J. Eur. Math. Soc. (JEMS)}, 22(1):1--53, 2020.

\bibitem[BHMM21]{BeyondSphericalSup}
Valentin Blomer, Gergely Harcos, P\'{e}ter Maga, and Djordje Mili\'{c}evi\'{c}.
\newblock Beyond the spherical sup-norm problem.
\newblock {\em {P}reprint}, 2021.
\newblock {\tt arXiv:2107.05973}.

\bibitem[BHW93]{Succminupper}
Ulrich Betke, Martin Henk, and J{\"o}rg~M. Wills.
\newblock Successive-minima-type inequalities.
\newblock {\em Discrete Comput. Geom.}, 9(2):165--175, 1993.

\bibitem[BK15]{MR3368079}
Jack Buttcane and Rizwanur Khan.
\newblock {$L^4$}-norms of {H}ecke newforms of large level.
\newblock {\em Math. Ann.}, 362(3-4):699--715, 2015.

\bibitem[Blo13]{MR3082245}
Valentin Blomer.
\newblock On the 4-norm of an automorphic form.
\newblock {\em J. Eur. Math. Soc. (JEMS)}, 15(5):1825--1852, 2013.

\bibitem[BM11]{MR2852302}
Valentin Blomer and Philippe Michel.
\newblock Sup-norms of eigenfunctions on arithmetic ellipsoids.
\newblock {\em Int. Math. Res. Not. IMRN}, (21):4934--4966, 2011.

\bibitem[BM13]{MR3103131}
Valentin Blomer and Philippe Michel.
\newblock Hybrid bounds for automorphic forms on ellipsoids over number fields.
\newblock {\em J. Inst. Math. Jussieu}, 12(4):727--758, 2013.

\bibitem[Bor63]{Borel-cl-finite}
Armand Borel.
\newblock Some finiteness properties of adele groups over number fields.
\newblock {\em Inst. Hautes \'{E}tudes Sci. Publ. Math.}, (16):5--30, 1963.

\bibitem[Bos20]{Bosch-Alg}
Siegfried Bosch.
\newblock {\em Algebra}.
\newblock Springer Spektrum, Berlin, ninth edition, [2020] \copyright 2020.

\bibitem[Bum97]{MR1431508}
Daniel Bump.
\newblock {\em Automorphic Forms and Representations}, volume~55 of {\em
  Cambridge Studies in Advanced Mathematics}.
\newblock Cambridge University Press, Cambridge, 1997.

\bibitem[Cas73]{MR337789}
William Casselman.
\newblock On some results of {A}tkin and {L}ehner.
\newblock {\em Math. Ann.}, 201:301--314, 1973.

\bibitem[Com21]{Comtatramified}
F\'{e}licien Comtat.
\newblock Optimal sup norm bounds for newforms on {$\rm GL_2$} with maximally
  ramified central character.
\newblock {\em Forum Math.}, 33(1):1--16, 2021.

\bibitem[GL87]{GrubLekkGeometryofNumbers}
P.~M. Gruber and C.~G. Lekkerkerker.
\newblock {\em Geometry of numbers}, volume~37 of {\em North-Holland
  Mathematical Library}.
\newblock North-Holland Publishing Co., Amsterdam, second edition, 1987.

\bibitem[Har03]{MR1990914}
Gergely Harcos.
\newblock An additive problem in the {F}ourier coefficients of cusp forms.
\newblock {\em Math. Ann.}, 326(2):347--365, 2003.

\bibitem[HC19]{MR3991392}
Fei Hou and Bin Chen.
\newblock Level aspect subconvexity for twisted {$L$}-functions.
\newblock {\em J. Number Theory}, 203:12--31, 2019.

\bibitem[HL94]{HL94}
Jeffrey Hoffstein and Paul Lockhart.
\newblock Coefficients of {M}aass forms and the {S}iegel zero.
\newblock {\em Ann. of Math. (2)}, 140(1):161--181, 1994.
\newblock With an appendix by Dorian Goldfeld, Hoffstein and Daniel Lieman.

\bibitem[HM06]{MR2207235}
Gergely Harcos and Philippe Michel.
\newblock The subconvexity problem for {R}ankin-{S}elberg {$L$}-functions and
  equidistribution of {H}eegner points. {II}.
\newblock {\em Invent. Math.}, 163(3):581--655, 2006.

\bibitem[HN18]{2018arXiv181011564Hs}
Yueke {Hu} and Paul~D. {Nelson}.
\newblock New test vector for {W}aldspurger's period integral, relative trace
  formula, and hybrid subconvexity bounds.
\newblock {\em {P}reprint}, 2018.
\newblock {\tt arXiv:1810.11564}.

\bibitem[HNS17]{HNSminimal}
Yueke Hu, Paul~{D}. Nelson, and Abhishek Saha.
\newblock Some analytic aspects of automorphic forms on {G}{L}(2) of minimal
  type.
\newblock to appear in Comm. Math. Helv., 2017.

\bibitem[HS20]{MR4190047}
Yueke Hu and Abhishek Saha.
\newblock Sup-norms of eigenfunctions in the level aspect for compact
  arithmetic surfaces, {II}: newforms and subconvexity.
\newblock {\em Compos. Math.}, 156(11):2368--2398, 2020.

\bibitem[HT12]{HT2}
Gergely Harcos and Nicolas Templier.
\newblock On the sup-norm of {M}aass cusp forms of large level: {II}.
\newblock {\em Int. Math. Res. Not. IMRN}, (20):4764--4774, 2012.

\bibitem[HT13]{MR3038127}
Gergely Harcos and Nicolas Templier.
\newblock On the sup-norm of {M}aass cusp forms of large level. {III}.
\newblock {\em Math. Ann.}, 356(1):209--216, 2013.

\bibitem[IS95]{IS95}
Henryk Iwaniec and Peter Sarnak.
\newblock {$L^\infty$} norms of eigenfunctions of arithmetic surfaces.
\newblock {\em Ann. of Math. (2)}, 141(2):301--320, 1995.

\bibitem[Iwa90]{MR1067982}
Henryk Iwaniec.
\newblock Small eigenvalues of {L}aplacian for {$\Gamma_0(N)$}.
\newblock {\em Acta Arith.}, 56(1):65--82, 1990.

\bibitem[Iwa97]{Iw97}
Henryk Iwaniec.
\newblock {\em Topics in classical automorphic forms}, volume~17 of {\em
  Graduate Studies in Mathematics}.
\newblock American Mathematical Society, Providence, RI, 1997.

\bibitem[JL70]{MR0401654}
Herv{\'e} Jacquet and R.~P. Langlands.
\newblock {\em Automorphic forms on {${\rm GL}(2)$}}.
\newblock Lecture Notes in Mathematics, Vol. 114. Springer-Verlag, Berlin,
  1970.

\bibitem[KR94]{KudlaRallis}
Stephen~S. Kudla and Stephen Rallis.
\newblock A regularized {S}iegel-{W}eil formula: the first term identity.
\newblock {\em Ann. of Math. (2)}, 140(1):1--80, 1994.

\bibitem[KS20]{Theta-supnorm-Is}
Ilya Khayutin and Raphael~S. Steiner.
\newblock Theta functions, fourth moments of eigenforms, and the sup-norm
  problem {I}.
\newblock {\em {P}reprint}, 2020.
\newblock {\tt arXiv:2009.07194}.

\bibitem[LPS87]{LPS-S2-distr-II}
A.~Lubotzky, R.~Phillips, and P.~Sarnak.
\newblock Hecke operators and distributing points on {$S^2$}. {II}.
\newblock {\em Comm. Pure Appl. Math.}, 40(4):401--420, 1987.

\bibitem[M{\oe}g97]{MoeglinTheta}
Colette M{\oe}glin.
\newblock Quelques propri\'{e}t\'{e}s de base des s\'{e}ries th\'{e}ta.
\newblock {\em J. Lie Theory}, 7(2):231--238, 1997.

\bibitem[Nel15]{MR3356036}
Paul~D. Nelson.
\newblock Evaluating modular forms on {S}himura curves.
\newblock {\em Math. Comp.}, 84(295):2471--2503, 2015.

\bibitem[Nel16a]{nelson-variance-73-2s}
Paul~{D}. Nelson.
\newblock Quantum variance on quaternion algebras, {I}.
\newblock {\em {P}reprint}, 2016.
\newblock {\tt arXiv:1601.02526}.

\bibitem[Nel16b]{nelson-theta-squared}
Paul~{D}. Nelson.
\newblock The spectral decomposition of {$|\theta|^2$}.
\newblock preprint, 2016.

\bibitem[Nel17]{nelson-variance-IIs}
Paul~{D}. Nelson.
\newblock Quantum variance on quaternion algebras, {II}.
\newblock {\em {P}reprint}, 2017.
\newblock {\tt arXiv:1702.02669}.

\bibitem[Nel19]{nelson-variance-3s}
Paul~{D}. Nelson.
\newblock Quantum variance on quaternion algebras, {III}.
\newblock {\em {P}reprint}, 2019.
\newblock {\tt arXiv:1903.08686}.

\bibitem[Nel20]{Nelson-TwistedSym2}
Paul~D. Nelson.
\newblock Bounds for twisted symmetric square {$L$}-functions via half-integral
  weight periods.
\newblock {\em Forum Math. Sigma}, 8:Paper No. e44, 21, 2020.

\bibitem[Nor21]{NordentoftEisSup}
Asbj{\o}rn~Christian Nordentoft.
\newblock Hybrid subconvexity for class group {$L$}-functions and uniform sup
  norm bounds of {E}isenstein series.
\newblock {\em Forum Math.}, 33(1):39--57, 2021.

\bibitem[PY19]{2019arXiv190810346P}
Ian {Petrow} and Matthew~P. {Young}.
\newblock {The fourth moment of Dirichlet $L$-functions along a coset and the
  Weyl bound}.
\newblock {\em arXiv e-prints}, page arXiv:1908.10346, Aug 2019.

\bibitem[Ral84]{RallisHoweDuality}
S.~Rallis.
\newblock On the {H}owe duality conjecture.
\newblock {\em Compositio Math.}, 51(3):333--399, 1984.

\bibitem[RS75]{RallisSchiffmann}
S.~Rallis and G.~Schiffmann.
\newblock Distributions invariantes par le groupe orthogonal.
\newblock In {\em Analyse harmonique sur les groupes de {L}ie ({S}\'{e}m.,
  {N}ancy-{S}trasbourg, 1973--75)}, pages 494--642. Lecture Notes in Math.,
  Vol. 497. 1975.

\bibitem[Sah17]{MR3713048}
Abhishek Saha.
\newblock On sup-norms of cusp forms of powerful level.
\newblock {\em J. Eur. Math. Soc. (JEMS)}, 19(11):3549--3573, 2017.

\bibitem[Sah20]{saha2019sup}
Abhishek Saha.
\newblock Sup-norms of eigenfunctions in the level aspect for compact
  arithmetic surfaces.
\newblock {\em Math. Ann.}, 376(1-2):609--644, 2020.

\bibitem[Saw21]{MR4307129}
Will Sawin.
\newblock A geometric approach to the sup-norm problem for automorphic forms:
  the case of newforms on {$GL_2(\Bbb F_q(T))$} with squarefree level.
\newblock {\em Proc. Lond. Math. Soc. (3)}, 123(1):1--56, 2021.

\bibitem[Shi72]{MR0333081}
Hideo Shimizu.
\newblock Theta series and automorphic forms on {${\rm GL}_{2}$}.
\newblock {\em J. Math. Soc. Japan}, 24:638--683, 1972.

\bibitem[Ste20]{MR4099641}
Raphael~S. Steiner.
\newblock Sup-norm of {H}ecke-{L}aplace eigenforms on {$S^3$}.
\newblock {\em Math. Ann.}, 377(1-2):543--553, 2020.

\bibitem[Ste21]{SteinerSmallDiameter}
Raphael~S. Steiner.
\newblock Small diameters and generators for arithmetic lattices in
  {$\mathrm{SL}_2(\mathbb{R})$} and certain {R}amanujan graphs.
\newblock {\em {P}reprint}, 2021.
\newblock {\tt CHANGE-ME}.

\bibitem[Tem10]{MR2734340}
Nicolas Templier.
\newblock On the sup-norm of {M}aass cusp forms of large level.
\newblock {\em Selecta Math. (N.S.)}, 16(3):501--531, 2010.

\bibitem[Tem15]{MR3372076}
Nicolas Templier.
\newblock Hybrid sup-norm bounds for {H}ecke-{M}aass cusp forms.
\newblock {\em J. Eur. Math. Soc. (JEMS)}, 17(8):2069--2082, 2015.

\bibitem[Tit76]{Tits-Classification}
J.~Tits.
\newblock Classification of buildings of spherical type and {M}oufang polygons:
  a survey.
\newblock In {\em Colloquio {I}nternazionale sulle {T}eorie {C}ombinatorie
  ({R}oma, 1973), {T}omo {I}}, pages 229--246. Atti dei Convegni Lincei, No.
  17. 1976.

\bibitem[Tom22]{TomaDivSup}
R.~Toma.
\newblock Hybrid bounds for the sup-norm of automorphic forms in higher rank.
\newblock {\em {P}reprint}, 2022.
\newblock {\tt arXiv:2111.09923v2}.

\bibitem[vdC36]{vanderCorput1936}
Johannes~G. van~der Corput.
\newblock Verallgemeinerung einer {M}ordellschen {B}eweismethode in der
  {G}eometrie der {Z}ahlen, {Z}weite {M}itteilung.
\newblock {\em Acta Arithmetica}, 2(1):145--146, 1936.

\bibitem[Vig77]{MR0480352}
Marie-France Vign\'{e}ras.
\newblock S\'{e}ries th\^{e}ta des formes quadratiques ind\'{e}finies.
\newblock In {\em S\'{e}minaire {D}elange-{P}isot-{P}oitou, 17e ann\'{e}e
  (1975/76), {T}h\'{e}orie des nombres: {F}asc. 1, {E}xp. {N}o. 20}, page~3.
  1977.

\bibitem[Voi18]{voightQA}
John Voight.
\newblock Quaternion algebras, 2018.
\newblock Available at https://math.dartmouth.edu/~jvoight/quat.html.

\bibitem[Wal85]{MR783511}
J.-L. Waldspurger.
\newblock Sur les valeurs de certaines fonctions {$L$} automorphes en leur
  centre de sym\'etrie.
\newblock {\em Compositio Math.}, 54(2):173--242, 1985.

\bibitem[Wat08]{watson-2008s}
Thomas~C. Watson.
\newblock Rankin triple products and quantum chaos.
\newblock {\em {P}reprint}, 2008.
\newblock {\tt arXiv:0810.0425}.

\bibitem[Wei64]{MR0165033}
Andr{\'e} Weil.
\newblock Sur certains groupes d'op\'erateurs unitaires.
\newblock {\em Acta Math.}, 111:143--211, 1964.

\bibitem[Xia07]{Xiasupnorm}
Honggang Xia.
\newblock On {$L^\infty$} norms of holomorphic cusp forms.
\newblock {\em J. Number Theory}, 124(2):325--327, 2007.

\bibitem[You17]{MR3635360}
Matthew~P. Young.
\newblock Weyl-type hybrid subconvexity bounds for twisted {$L$}-functions and
  {H}eegner points on shrinking sets.
\newblock {\em J. Eur. Math. Soc. (JEMS)}, 19(5):1545--1576, 2017.

\end{thebibliography}
\end{document}